\let\ORIlabel\label
\let\ORIrefstepcounter\refstepcounter
\AddToHook{package/hyperref/before}{
   \let\label\ORIlabel 
   \let\refstepcounter\ORIrefstepcounter}

\documentclass[onefignum,onetabnum]{siamonline220329}

\usepackage{amsfonts, mathtools, bm, amssymb}
\usepackage{graphicx}
\usepackage{xcolor}
\usepackage{booktabs}
\usepackage{array,multirow}
\usepackage{tabularx}
\usepackage{sidecap}

\usepackage{tikz}
\usetikzlibrary{decorations.pathreplacing}
\usetikzlibrary{patterns}

\usetikzlibrary {patterns.meta}
\pgfdeclarepattern{
  name=hatch,
  parameters={\hatchsize,\hatchangle,\hatchlinewidth},
  bottom left={\pgfpoint{-.1pt}{-.1pt}},
  top right={\pgfpoint{\hatchsize+.1pt}{\hatchsize+.1pt}},
  tile size={\pgfpoint{\hatchsize}{\hatchsize}},
  tile transformation={\pgftransformrotate{\hatchangle}},
  code={
    \pgfsetlinewidth{\hatchlinewidth}
    \pgfpathmoveto{\pgfpoint{-.1pt}{-.1pt}}
    \pgfpathlineto{\pgfpoint{\hatchsize+.1pt}{\hatchsize+.1pt}}
    \pgfpathmoveto{\pgfpoint{-.1pt}{\hatchsize+.1pt}}
    \pgfpathlineto{\pgfpoint{\hatchsize+.1pt}{-.1pt}}
    \pgfusepath{stroke}
  }
}

\newcolumntype{Y}{>{\centering\arraybackslash}X}

\tikzset{
  hatch size/.store in=\hatchsize,
  hatch angle/.store in=\hatchangle,
  hatch line width/.store in=\hatchlinewidth,
  hatch size=5pt,
  hatch angle=0pt,
  hatch line width=.5pt,
}

\usepackage{hyperref}

\def\N			{\mathbb N}

\def\R			{\mathbb R}

\def\Sphere		{\mathbb{S}_1}

\def\K		    {\mathbb{K}}

\def\M	        {\mathcal M}
\def\Lebesgue	{\mathrm L}

\def\Radon		{\mathcal R}
\def\NRCDT		{\mathcal N}
\def\maxNRCDT	{\NRCDT_{\mathrm{m}}}
\def\meanNRCDT	{\NRCDT_{\mathrm{a}}}

\def\GL			{\mathrm{GL}}

\def\d			{\mathop{}\!\mathrm{d}}
\def\D			{\mathrm D}

\def\bfx			{\mathbf{x}}
\def\bfy			{\mathbf{y}}

\def\bfzero		{\mathbf{0}}
\def\bfA			{\mathbf{A}}
\def\bfI			{\mathbf{I}}
\def\bftheta		{{\boldsymbol{\theta}}}

\def\P			{\mathcal P}

\def\X			{\mathcal X}

\def\Glue		{\mathcal I}

\def\F			{\mathbb F}
\def\G			{\mathbb G}
\def\H			{\mathbb H}

\DeclareMathOperator{\diam}{diam}
\DeclareMathOperator{\supp}{supp}

\DeclareMathOperator{\conv}{conv}
\DeclareMathOperator{\mean}{mean}

\DeclareMathOperator{\std}{std}

\DeclareMathOperator*{\argmin}{arg\,min}

\newcommand{\subsubset}{\subset\joinrel\subset}

\def\mNRCDT			{\textsubscript{m}NR-CDT}
\def\aNRCDT			{\textsubscript{a}NR-CDT}

\DeclareFontFamily{U}{mathx}{\hyphenchar\font45}
\DeclareFontShape{U}{mathx}{m}{n}{<-> mathx10}{}
\DeclareSymbolFont{mathx}{U}{mathx}{m}{n}
\DeclareMathAccent{\widebar}{0}{mathx}{"73}

\newsiamremark{remark}{Remark}
\newsiamremark{example}{Example}

\headers{Normalized Radon-CDT}{M.~Beckmann, R.~Beinert and J.~Bresch}

\title{Normalized Radon Cumulative Distribution Transforms for Invariance and Robustness in Optimal Transport Based Image Classification%
\thanks{Preliminary and exploratory ideas have been presented in our conference paper~\cite{Beckmann2024a}.}}

\author{Matthias Beckmann\thanks{Center for Industrial Mathematics, University of Bremen, Germany \& Department of Electrical and Electronic Engineering, Imperial College London, UK 
  (\email{research@mbeckmann.de}).}
\and Robert Beinert\thanks{Institut f\"{u}r Mathematik, Technische Universit\"{a}t Berlin, Germany 
  (\email{beinert@math.tu-berlin.de}).}
\and Jonas Bresch\thanks{Institut f\"{u}r Mathematik, Technische Universit\"{a}t Berlin, Germany 
  (\email{bresch@math.tu-berlin.de}).}}

\ifpdf
\hypersetup{
  pdftitle={Normalized Radon-CDT},
  pdfauthor={M. Beckmann, R. Beinert and J. Bresch}
}
\fi

\begin{document}

\maketitle

\begin{abstract}
The Radon cumulative distribution transform (R-CDT),
is an easy-to-compute feature extractor
that facilitates image classification tasks
especially in the small data regime.
It is closely related to the sliced Wasserstein distance and
provably guaranties the linear separability of image classes
that emerge from translations or scalings.
In many real-world applications,
like the recognition of watermarks in filigranology,
however,
the data is subject to general affine transformations
originating from the measurement process.
To overcome this issue,
we recently introduced the so-called max-normalized R-CDT
that only requires elementary operations and
guaranties the separability under arbitrary affine transformations.
The aim of this paper is to continue 
our study of the max-normalized R-CDT
especially 
with respect to its robustness
against non-affine image deformations.
Our sensitivity analysis shows
that its separability properties 
are stable provided 
the Wasserstein-infinity distance
between the samples can be controlled.
Since the Wasserstein-infinity distance only allows
small local image deformations,
we moreover introduce
a mean-normalized version of the R-CDT.
In this case,
robustness relates to the Wasserstein-2 distance
and also covers image deformations 
caused by impulsive noise
for instance.
Our theoretical results are supported by numerical experiments
showing the effectiveness of our novel feature extractors
as well as their robustness against
local non-affine deformations and impulsive noise.
\end{abstract}



\section{Introduction}

Automatic pattern recognition and data classification
play a crucial role in various scientific disciplines and applications,
like medical imaging, biometrics, computer vision or document analysis,
to name just a few.
As of today,
end-to-end deep neural networks provide the state of the art
if sufficient training data is available. 
In the small data regime,
however,
or,
if performance guarantees are important,
hand-crafted feature extractors and classifiers
are still the first choice.
Ideally,
the feature representation is designed to transform the different classes
to linearly separable subsets.
This can,
for instance,
be achieved by applying
the Radon cumulative distribution transform (R-CDT)
introduced in~\cite{Kolouri2016},
which is based on one-dimensional optimal transport maps,
also called cumulative distribution transform~\cite{Park2018, Aldroubi2022},
that are generalized to two-dimensional data
by applying the Radon transform~\cite{Radon1917, Helgason1999},
known from computerized tomography~\cite{Ramm1996, Natterer2001}.
This approach shows great potential in many applications
\cite{Kolouri2017, Shifat-E-Rabbi2021, Aldroubi2021, DiazMartin2024, Zhuang2025}
and is closely related to the sliced Wasserstein distance \cite{Bonneel2015, Shifat-E-Rabbi2023, Liu2024, Park2025, Piening2025}.
A similar approach for data on the sphere is studied in \cite{Quellmalz2023,Quellmalz2024},
for multi-dimensional optimal transport maps in \cite{Moosmueller2023, Cloninger2025},
and for optimal Gromov--Wasserstein transport maps in \cite{Beier2022,Beinert2023}.

In our recent work~\cite{Beckmann2024a}, we introduced a novel normalization of the R-CDT, we referred to as max-normalized R-CDT (\mNRCDT), to enhance linear separability in the context of affine transformations.
This was inspired by the special needs for applying pattern recognition techniques in filigranology---the study of watermarks.
These play a central role in provenance research like dating of historical manuscripts, scribe identification and paper mill attribution.
For automatic classification,
the main issues are the enormous number of classes
with only few members per class,
see WZIS%
\footnote{Wasserzeichen-Informationssystem: \url{www.wasserzeichen-online.de}.},
as well as the uncertainty with respect to the position, size, orientation and slight distortion of the watermark in the digitized image.
A first end-to-end processing pipeline for thermograms of watermarks 
including an R-CDT-based classification is proposed
in \cite{Hauser2024}, where classification invariance
with respect to translation and dilation of the watermark is achieved,
but other affine transformations are not included.
In contrast to this, our \mNRCDT{} ensures the linear separability of affinely transformed classes without any restrictions on the affine transformations, as theoretically shown in~\cite[Theorem~1]{Beckmann2024a} and numerically validated
by proof-of-concept experiments in~\cite[\S~4]{Beckmann2024a}.

In this work, we go one step further and analyse the linear separability in \mNRCDT{} space when allowing for small perturbations measured in Wasserstein-$\infty$ space in addition to affine transformations.
Recall that our definition of the \mNRCDT{} assumes a compact support of the considered measure.
To weaken this assumption, we introduce a new normalization of the R-CDT, which we call mean-normalized R-CDT (\aNRCDT).
We study the linear separability in \aNRCDT{} space for measure classes constructed by affinely transforming distinguishable template measures.
We observe that, in contrast to \mNRCDT, our new normalization step poses restrictions on the affine transformations in order to guarantee separability.
However, when considering perturbations of the templates, our new normalization allows for distortions measured in Wasserstein-$2$ distance instead of the more restrictive Wasserstein-$\infty$ metric.

This manuscript is organized as follows.
In Section~\ref{sec:prelim}, we introduce basic concepts and fix our notation.
Section~\ref{sec:RCDT} is devoted to the R-CDT for bivariate measures, where we start with explaining the CDT for probability measures on $\R$ and, then, extend it to the R-CDT for probability measures on $\R^2$ by means of the Radon transform.
In Section~\ref{sec:NRCDT} we first recall our definition of the normalized R-CDT from~\cite{Beckmann2024a} and show elementary properties.
The \mNRCDT{} is explained in Section~\ref{ssec:mNRCDT}, where we also recall the linear separability result from~\cite{Beckmann2024a} and extend it by considering perturbations in Wasserstein-$\infty$ space.
Thereon, our novel \aNRCDT{} is introduced in Section~\ref{ssec:aNRCDT} and we show linear separability in \aNRCDT{} space under affine transformations and, additionally, perturbations in Wasserstein-$2$ space.
Our theoretical findings are illustrated by numerical experiments in Section~\ref{sec:numerics}, showing the effectiveness of our approach.
Section~\ref{sec:conclusion} concludes with a discussion of our results and future research direction.

\section{Preliminaries}
\label{sec:prelim}

Throughout the paper, 
we restrict our attention to 
functions and measures on
the Euclidean space $(\R^d, \lVert \cdot\rVert)$
and the infinite cylinder $\R \times \Sphere$
with $\Sphere \coloneqq \{ \bfx \in \R^2 \mid \lVert\bfx\rVert=1\}$.
Compact subsets of these are indicated
using the symbol $\subsubset$.
For domain $X$,
we denote the Lebesgue spaces by $L^p(X)$;
the space of continuous, bounded functions by $C_b(X)$;
the space of continuous functions vanishing at infinity by $C_0(X)$;
and the space of finite, signed, regular (Borel) measures by $\M(X)$.
Recall that $\M(X)$ is the continuous dual of $C_0(X)$.
For $\mu \in \M(X)$,
its \emph{support} $\supp(\mu)$ is the minimal closed subset $Y \subset X$
such that $\mu(X \setminus Y) = 0$.
For $\mu \in \M(\R^d)$,
the \emph{dimension} of the affine hull of $\supp(\mu)$
is denoted by $\dim (\mu)$,
and the \emph{diameter} is given by
$\diam(\mu) \coloneqq \sup_{\bfx,\bfy \in \supp(\mu)} \lVert \bfx - \bfy \rVert$.

For two domains $X$ and $Y$,
the \emph{push-forward} of a Borel probability measure $\mu \in \P(X)$
via a mapping $T \colon X \to Y$
is defined by $T_\# \mu \coloneqq \mu \circ T^{-1}$.
To compare two measures,
we use the so-called Wasserstein or Kantorovich--Rubinstein metric,
which is defined on spaces of \emph{probability measures
with finite $p$th moment} given by
\begin{align*}
    \P_p(\R^d)
    &\coloneqq
    \Bigl\{ \mu \in \P(\R^d) 
    \Bigm\vert
    \int_{\R^d} \lVert \bfx \rVert^p \d\bfx < \infty \Bigr\},
    \qquad p \in [1, \infty),
    \\
    \P_\infty(\R^d)
    &\coloneqq
    \Bigl\{ \mu \in \P(\R^d)
    \Bigm\vert 
    \sup\nolimits_{\bfx \in \supp(\mu)} \lVert \bfx \rVert < \infty \Bigr\}.
\end{align*}
Moreover,
we introduce the canonical projections 
$P_1(\bfx,\bfy) \coloneqq \bfx$ 
and $P_2(\bfx, \bfy) \coloneqq \bfy$
as well as the \emph{set of transport plans}
\begin{equation*}
    \Pi(\mu,\nu)
    \coloneqq 
    \bigl\{
    \pi \in \P(\R^d \times \R^d)
    \bigm\vert 
    (P_{1})_{\#} \pi = \mu, ~
    (P_{2})_{\#} \pi = \nu
    \bigr\},
    \qquad
    \mu,\nu \in \P(\R^d).
\end{equation*}
Then,
the \emph{Wasserstein-$p$ distance} 
between $\mu,\nu \in \P_p(\R^d)$
is defined as
\begin{subequations}
    \label{eq:wasserstein}
    \begin{align}
        W_p(\mu, \nu) 
        &\coloneqq
        \inf_{\pi \in \Pi(\mu,\nu)}
        \Bigl( 
        \int_{\R^d \times \R^d} 
        \lVert \bfx - \bfy \rVert^p 
        \d \pi(\bfx, \bfy)
        \Bigr)^\frac{1}{p},
        \qquad p \in [1, \infty)
        \\
        W_\infty(\mu, \nu)
        &\coloneqq 
        \inf_{\pi \in \Pi(\mu,\nu)}
        \sup_{(\bfx,\bfy) \in \supp(\pi)} \lVert \bfx - \bfy \rVert.
    \end{align}
\end{subequations}
The Wasserstein-$p$ space $(\P_p(\R^d), W_p)$ is a metric space,
and the infima in \eqref{eq:wasserstein} are attained by
an optimal transport plan $\pi \in \Pi(\mu,\nu)$;
see \cite[Prop.~1 \& 2]{Givens1984}.

\section{Radon Cumulative Distribution Transform}
\label{sec:RCDT}

Following the approach in~\cite{Park2018}, for a probability measure $\mu \in \P(\R)$, we consider its cumulative distribution function $F_\mu \colon \R \to [0,1]$ given by
\begin{equation*}
F_\mu(t) = \mu((-\infty,t]),
\quad  t \in \R,
\end{equation*}
and define the {\em cumulative distribution transform} $\widehat{\mu} \colon \R \to \R$, in short CDT, via
\begin{equation*}
\widehat{\mu} = F_\mu^{[-1]} \circ F_\rho
\end{equation*}
with the generalized inverse, known as quantile function,
\begin{equation*}
F_\mu^{[-1]}(t) = \inf \{s \in \R \mid F_\mu(s) > t\},
\quad t \in \R
\end{equation*}
and a reference $\rho \in \P_2(\R)$ that does not give mass to atoms, e.g.,
$\rho = \chi_{[0,1]} \lambda_\R$, where $\lambda_\R$ denotes the standard Lebesgue measure on $\R$.
Note that, if $\mu \in \P_2(\R)$ has finite second moment, we have
\begin{equation*}
\widehat{\mu} = \argmin_{T_\# \rho = \mu} \int_\R |s-T(s)|^2 \: \d \rho(s),
\end{equation*}
i.e., $\widehat{\mu}$ is the unique map $T \colon \R \to \R$ transporting $\rho$ to $\mu$ while minimizing the cost, cf.~\cite{Villani2003}.
Moreover, $\widehat{\mu}$ is square integrable with respect to $\rho$, i.e., $\widehat{\mu} \in L^2_\rho(\R)$, and, for $\mu, \nu \in \P_2(\R)$, we have
\begin{equation*}
    \lVert \widehat{\mu} - \widehat{\nu} \rVert_\rho
    \coloneqq
    \Bigl(
        \int_\R
        \lvert \mu(t) - \nu(t) \rvert^2
        \: \d \rho(t)
    \Bigr)^{\frac{1}{2}}
    =
    W_2(\mu,\nu).
\end{equation*}

To deal with a probability measure $\mu \in \P(\R^2)$, we adapt the approach in~\cite{Kolouri2016} and apply the so-called Radon transform to obtain a family of probability measures on $\R$.
For a bivariate function $f \in L^1(\R^2)$, 
its {\em Radon transform} $\Radon [f] \colon \R \times \Sphere \to \R$ is defined as the line integral 
\begin{equation*}
    \Radon [f] (t, \bftheta) 
    \coloneqq 
    \int_{\ell_{t,\bftheta}} 
    f(s) 
    \: \d s,
    \quad 
    (t,\bftheta) \in \R \times \Sphere,
\end{equation*}
where $\d s$ denotes the arc length element of the straight line $\ell_{t,\bftheta}$ 
with signed distance $t \in \R$ to the origin and normal direction 
$\bftheta \in \Sphere \coloneqq \{\bfx \in \R^2 \mid \lVert \bfx \rVert = 1\}$,
i.e.,
\begin{equation*}
    \ell_{t,\bftheta} 
    \coloneqq 
    \{ t \bftheta + \tau \bftheta^\perp 
        \mid 
        \tau \in \R\}
    =
    S_\bftheta^{-1}(t)
    \subset \R^2
\end{equation*}
with \textit{slicing operator} $S_\bftheta\colon \R^2 \to \R$ given by
\begin{equation*}
    S_\bftheta(\bfx) \coloneqq \langle \bfx, \bftheta \rangle,
    \quad
    \bfx \in \R^2.
\end{equation*}
This defines the \textit{Radon operator} $\Radon \colon L^1(\R^2) \to L^1(\R \times \Sphere)$
and, for fixed $\bftheta \in \Sphere$, 
we set $\Radon_\bftheta \coloneqq \Radon(\cdot, \bftheta)$, 
which is referred to as the \textit{restricted Radon operator} 
$\Radon_\bftheta \colon L^1(\R^2) \to L^1(\R)$.
It is well known that $\Radon$ preserves mass, cf.~\cite{Natterer2001}, in the sense that,
for any $f \in L^1(\R^2)$, 
we have
\begin{equation} \label{eq:RT_fun_mass}
    \int_{\R} \Radon_\bftheta [f](t) \d t = \int_{\R^2} f(\bfx) \: \d \bfx,
    \qquad
    \int_{\Sphere} \int_{\R} \Radon [f](t, \bftheta) \: \d t \, \d u_{\Sphere}(\bftheta) = \int_{\R^2} f(\bfx) \: \d \bfx,
\end{equation}
where $u_{\Sphere} = \frac{\sigma_{\Sphere}}{2\pi}$ with surface measure $\sigma_{\Sphere}$ on $\Sphere$.
The adjoint $\Radon^* \colon L^\infty(\R \times \Sphere) \to L^\infty(\R^2)$ of Radon operator $\Radon \colon L^1(\R^2) \to L^1(\R \times \Sphere)$, called {\em back projection operator}, is given by 
\begin{equation*}
   \Radon^* [g](\bfx) \coloneqq \int_{\Sphere} g(S_\bftheta(\bfx), \bftheta) \: \d \sigma_{\Sphere}(\bftheta),
    \quad \bfx \in \R^2,
\end{equation*}
and, for fixed $\bftheta \in \Sphere$, the adjoint $\Radon^*_\bftheta \colon L^\infty(\R) \to L^\infty(\R^2)$ of the restricted Radon operator $\Radon_\bftheta \colon L^1(\R^2) \to L^1(\R)$ is given by 
\begin{equation*}
    \Radon^*_\bftheta [h] (\bfx) = h(S_\bftheta(\bfx)),
    \quad \bfx \in \R^2.
\end{equation*}
Furthermore,
$\Radon^* \colon C_b(\R \times \Sphere) \to C_b(\R^2)$
and
$\Radon^*_\bftheta \colon C_b(\R) \to C_b(\R^2)$.
This allows us to translate the concept of the Radon transform to signed, regular, finite measures $\mu \in \M(\R^2)$.
For a fixed direction $\bftheta \in \Sphere$, we generalize the \textit{restricted Radon transform} $\Radon_\bftheta$ to measures by setting
\begin{equation*}
    \Radon_\bftheta \colon \M(\R^2) \to \M(\R), \quad
    \mu \mapsto (S_\bftheta)_\# \mu = \mu \circ S_\bftheta^{-1},
\end{equation*}
which corresponds to the integration along the lines $\ell_{t,\bftheta}$.
As for functions in~\eqref{eq:RT_fun_mass}, we have
\begin{equation*}
    \Radon_\bftheta[\mu](\R) = \mu(\R^2)
    \quad \forall \, \bftheta \in \Sphere,
\end{equation*}
thus, mass is preserved by $\Radon_\bftheta$.
In measure theory, $\Radon_\bftheta$ can be considered as a disintegration family and,
heuristically, 
we generalize the Radon transform 
by integrating $\Radon_\bftheta$ along $\bftheta \in \Sphere$.
Therefore, we define the \textit{Radon transform} $\Radon \colon \M(\R^2) \to \M(\R \times \Sphere)$ via
\begin{equation*}
    \Radon [\mu] \coloneqq \Glue_\#[\mu \times u_{\Sphere}]
\end{equation*}
with {\em glueing operator} $\Glue\colon \R^2 \times \Sphere \to \R \times \Sphere$ given by
\begin{equation*}
    \Glue(\bfx, \bftheta) \coloneqq (S_\bftheta(\bfx), \bftheta),
    \quad (\bfx, \bftheta) \in \R^2 \times \Sphere.
\end{equation*}

In~\cite[Proposition 1]{Beckmann2024a} we have shown that, for $\mu \in \M(\R^2)$, $\Radon [\mu]$ can indeed be disintegrated into the family $\Radon_\bftheta [\mu]$ with respect to the uniform measure $u_{\Sphere} = \tfrac{\sigma_{\Sphere}}{2\pi}$,
i.e., for all $g \in C_0(\R \times \Sphere)$, we have
\begin{equation*}
    \langle \Radon [\mu], g\rangle = \int_{\Sphere} \langle \Radon_{\bftheta} [\mu], g(\cdot, \bftheta)\rangle \: \d u_{\Sphere}(\bftheta).
\end{equation*}
Moreover, in~\cite[Proposition 2]{Beckmann2024a} we have proven that  the measure-valued transforms $\Radon$ and $\Radon_\bftheta$ are the adjoints of the back projection operators $\Radon^*$ and $\Radon_\bftheta^*$ from above.
More precisely, the Radon transform of $\mu \in \M(\R^2)$ satisfies
\begin{equation} \label{eq:RT_duality}
    \langle \Radon [\mu], g\rangle = \langle \mu, \Radon^* [g]\rangle
    \quad \forall \,
    g \in L^\infty(\R \times \Sphere)
\end{equation}
and
\begin{equation*}
    \langle \Radon_\bftheta [\mu], h\rangle = \langle \mu, \Radon_\bftheta^* [h]\rangle
    \quad \forall \,
    h \in L^\infty(\R) ~
    \forall \,
    \bftheta \in \Sphere.
\end{equation*}
This observation suggests 
that the Radon transform for measures 
can equivalently be defined through duality.
However, 
the dual space of $\M(\R)$ is $C_0(\R)$, 
whereas $h \in C_0(\R) \setminus \{0\}$ 
does not imply that $\Radon_\bftheta^* [h] \in C_0(\R)$ 
for $\bftheta \in \Sphere$.
But if $h \in C_0(\R\times \Sphere)$,
we have $\Radon^*[h] \in C_0(\R^2)$.

\begin{proposition}
    Let $h \in C_0(\R\times \Sphere)$.
    Then,
    we have $\Radon^*[h] \in C_0(\R^2)$.
\end{proposition}

\begin{proof}
	Let $(\bfx_n)_{n \in \N} \subseteq \R^2$
	such that $\lVert\bfx_n\rVert \to \infty$ for $n \to \infty$.
	For arbitrarily $\varepsilon > 0$ 
	there exists $M_\varepsilon > 0$ such that 
	$\lvert h(x,\bftheta)\rvert < \tfrac{\varepsilon}{2}$
	for all $\lvert x \rvert \geq M_\varepsilon$
	and $\bftheta \in \Sphere$.
	Hence,
	we have 
	\begin{equation*}
		|\Radon^*[h](\bfx_n)|
		=
		\biggl\lvert \int_{\Sphere} 
		h(\langle \bfx_n, \bftheta \rangle, \bftheta) 
		\: \d u_{\Sphere}(\bftheta)\biggr\rvert
		\leq 
		\frac{\varepsilon}{2}
		+ 
		\int_{\{\lvert \langle \bfx_n, \bftheta\rangle\rvert < M_\varepsilon\}}
		\lvert h(\langle \bfx_n, \bftheta \rangle, \bftheta) \rvert
		\: \d u_{\Sphere}(\bftheta).
	\end{equation*}
	Since $h \in C_0(\R\times \Sphere)$,
	there exits for $\varepsilon' > 0$ 
	a compact set $K_{\varepsilon'}$ such that 
	$\rvert h|_{K_{\varepsilon'} \times \Sphere}\lvert \geq \varepsilon'$
	and hence there exists 
	$C \coloneqq \max_{(x,\bftheta) \in \R\times \Sphere} |h(x, \bftheta)| > 0$.
	Combining both yields 
	\begin{equation*}
		|\Radon^*[h](\bfx_n)|
		\leq 
		\frac{\varepsilon}{2}
		+ 
		C \, u_{\Sphere}(\{\lvert \langle \bfx_n, \bftheta\rangle\rvert < M_\varepsilon\}).
	\end{equation*}
	Since 
	\begin{equation*}
		\lvert \langle \bfx_n, \bftheta\rangle\rvert 
		< 
		M_\varepsilon
		\quad\Leftrightarrow\quad
		\Bigl\lvert \Bigl\langle \frac{\bfx_n}{\lVert\bfx_n\rVert}, \bftheta\Bigr\rangle\Bigr\rvert 
		< 
		\frac{M_\varepsilon}{\lVert\bfx_n\rVert},
	\end{equation*}
	there exists $N \in \N$ such that 
	$u_{\Sphere}(\{\lvert \langle \bfx_n, \bftheta\rangle\rvert < M_\varepsilon\}) < \tfrac{\varepsilon}{2C}$
	and the assertion follows.
\end{proof}

The definition of the Radon transform for measures is compatible with the classical definition for functions in the sense that the Radon transform of an absolutely continuous measure is again absolutely continuous.
To see this, we denote the surface measure by $\sigma_\K$ and the Lebesgue measure by $\lambda_\K$ for the different sets $\K \in \{\R, \R^2, \Sphere, \R \times \Sphere\}$.
    
\begin{proposition}
    Let $f \in L^1(\R^2)$.
    The Radon transform satisfies 
    \begin{equation*}
    \Radon[f \lambda_{\R^2}]  =  \Radon [f] \, \sigma_{\R \times \Sphere}
    \quad\text{and}\quad
    \Radon_\bftheta[f \lambda_{\R^2}] = \Radon_\bftheta [f] \, \lambda_{\R}.
    \end{equation*}
\end{proposition}

\begin{proof}
    We denote by $\langle\cdot, \cdot\rangle_{\M}$ the dual pairing between $\M$ and $C_0$ and by $\langle\cdot, \cdot \rangle_L$ the dual pairing between $L^1$ and $L^\infty$.
    Then, the duality relation in~\eqref{eq:RT_duality} gives
    \begin{equation*}
        \langle \Radon[f \lambda_{\R^2}], g\rangle_{\M} 
        = \langle f \lambda_{\R^2}, \Radon^* [g]\rangle_{\M} = \langle f , \Radon^* [g]\rangle_{L} 
        = \langle \Radon [f] , g\rangle_{L} = \langle \Radon [f] \, \sigma_{\R \times \Sphere} , g\rangle_{\M}
    \end{equation*}
    for all $g \in C_0(\R \times \Sphere)$.
\end{proof}

We have the following connection between the Wasserstein distance of measures and the Wasserstein distance of the corresponding restricted Radon transforms.

\begin{proposition} \label{prop:wass}
    Let $\bftheta \in \Sphere$.
    Then, we have
    \begin{align*}
        W_\infty(\Radon_\bftheta[\mu], \Radon_\bftheta[\nu]) 
        &\leq W_\infty(\mu, \nu)
        \quad \forall \, \mu, \nu \in \P_c(\R^2),
        \\
        W_2(\Radon_\bftheta[\mu], \Radon_\bftheta[\nu])
        &\leq W_2(\mu, \nu)
        \quad \forall \, \mu, \nu \in \P_2(\R^2).
    \end{align*}
\end{proposition}

\begin{proof}
    For $\mu, \nu \in \P_2(\R^2)$,
    let $\pi \in \Pi(\mu,\nu)$
    realize $W_2(\mu,\nu)$.
    Since the push-forward plan
    satisfies
    $(S_\bftheta, S_\bftheta)_\# \pi 
    \in 
    \Pi(\Radon_\bftheta[\mu], \Radon_\bftheta[\nu])$,
    the second inequality follows from
    \begin{align*}
        W_2^2(\mu,\nu)
        &=
        \int_{\R^2 \times \R^2}
        \lVert \bfx - \bfy \rVert^2
        \: \d \pi(\bfx, \bfy)
        \ge
        \int_{\R^2 \times \R^2}
        \lvert \langle \bfx - \bfy, \bftheta \rangle \rvert^2
        \: \d \pi(\bfx, \bfy)
        \\
        &=
        \int_{\R \times \R}
        \lvert t - s \rvert^2
        \: \d (S_\bftheta, S_\bftheta)_\# \pi(t,s)
        \ge
        W_2^2(\Radon_\bftheta[\mu], \Radon_\bftheta[\nu]).
    \end{align*}
    The first inequality can be established analogously.
\end{proof}

We are now prepared to adapt the CDT to a probability measure $\mu \in \P(\R^2)$.
To this end, we first consider its Radon transform $\Radon[\mu] \in \M(\R \times \Sphere)$, which satisfies
\begin{equation*}
\Radon_\bftheta [\mu] \in \P(\R)
\quad \forall \, \bftheta \in \Sphere,
\end{equation*}
and then, for each fixed $\bftheta \in \Sphere$, the CDT $\widehat{\mu}_\bftheta$ of the Radon projection $\mu_\bftheta = \Radon_\bftheta [\mu]$, yielding the {\em Radon cumulative distribution transform} (R-CDT) $\widehat{\Radon}\mu \colon \R \times \Sphere \to \R$ of $\mu$ via
\begin{equation*}
\widehat{\Radon} [\mu](t,\bftheta) = \widehat{\mu}_\bftheta(t),
\quad  (t,\bftheta) \in \R \times \Sphere.
\end{equation*}
In this way, any probability measure $\mu \in \P(\R^2)$ is mapped to its R-CDT $\widehat{\Radon}[\mu] \colon \R \times \Sphere \to \R$.
If $\mu \in \P_2(\R^2)$,
then the Radon projection $\Radon_\bftheta [\mu] \in \P_2(\R)$
has finite second moment as well and we have
$\widehat{\Radon}[\mu] \in L_{\rho \times u_{\Sphere}}^2(\R \times \Sphere)$.
Moreover, for $\mu, \nu \in \P_2(\R^2)$,
the norm distance
\begin{equation*}
    \lVert 
        \widehat{\Radon}[\mu] - \widehat{\Radon}[\nu]
    \rVert_{\rho \times u_{\Sphere}}
    \coloneqq
    \Bigl(
        \int_{\Sphere} \int_\R
        \lvert 
            \widehat{\Radon}[\mu](t, \bftheta)
            -
            \widehat{\Radon}[\nu](t, \bftheta)
        \rvert^2
        \: \d \rho(t) \, \d u_{\Sphere}(\bftheta)
    \Bigr)^{\frac{1}{2}}
\end{equation*}
agrees with the so-called sliced Wasserstein-2 distance \cite{Bonneel2015}.

\paragraph{Linear separability in R-CDT space}

We now show that the above feature representation via R-CDT enhances linear separability of distinct classes that are generated from template measures by certain transformations.
To be more precise, we assume that, for fixed $\bftheta_0 \in \Sphere$, we are given two classes $\F_{\bftheta_0}, \G_{\bftheta_0}$ that are generated by template measures $\mu_0,\ \nu_0 \in \P(\R^2)$ via
\begin{align*}
\F_{\bftheta_0} &= \bigl\{\mu \in \P(\R^2) \mid \exists \, h \in \H\colon \Radon_{\bftheta_0} [\mu] = h_\# \Radon_{\bftheta_0}[\mu_0]\bigr\}, \\
\G_{\bftheta_0} &= \bigl\{\nu \in \P(\R^2) \mid \exists \, h \in \H\colon \Radon_{\bftheta_0} [\nu] = h_\# \Radon_{\bftheta_0}[\nu_0] \bigr\},
\end{align*}
where $\H$ is a convex set of increasing bijections $h \colon \R \to \R$.
Then, we will prove that the transformed function classes in R-CDT space
\begin{equation*}
\widehat{\F}_{\bftheta_0} = \bigl\{\widehat{\Radon}_{\bftheta_0} [\mu] \colon \R \to \R \mid \mu \in \F_{\bftheta_0}\bigr\}, \qquad
\widehat{\G}_{\bftheta_0} = \bigl\{\widehat{\Radon}_{\bftheta_0} [\nu] \colon \R \to \R \mid \nu \in \G_{\bftheta_0}\bigr\}
\end{equation*}
are linearly separable if $\Radon_{\bftheta_0} \F_{\bftheta_0} \cap \Radon_{\bftheta_0} \G_{\bftheta_0} = \emptyset$, where
\begin{equation*}
\Radon_{\bftheta_0}\F_{\bftheta_0} = \{\Radon_{\bftheta_0} [\mu] \mid \mu \in \F_{\bftheta_0}\}, \qquad
\Radon_{\bftheta_0}\G_{\bftheta_0} = \{\Radon_{\bftheta_0} [\nu] \mid \nu \in \G_{\bftheta_0}\},
\end{equation*}
in the sense that for any two non-empty, finite subsets $\F_0 \subset \F_{\bftheta_0}$ and $\G_0 \subset \G_{\bftheta_0}$ there exist a continuous linear functional $\Phi \colon L^2_\rho(\R) \to \R$ and a constant $c \in \R$ such that
\begin{equation*}
\Phi\Bigl(\widehat{\Radon}_{\bftheta_0}[\mu]\Bigr) < c < \Phi\Bigl(\widehat{\Radon}_{\bftheta_0}[\nu]\Bigr)
\quad \forall \, \mu \in \F_0 ~ \forall \, \nu \in \G_0.
\end{equation*}
Note that a slightly different version of this result has first been stated in~\cite[\S~3.2]{Hauser2024} without a proof and we now close this gap in the literature.

\begin{theorem} 
\label{thm:sep-rcdt}
For templates $\mu_0, \nu_0 \in \P_2(\R^2)$ and $\bftheta_0 \in \Sphere$ consider the classes
\begin{align*}
\F_{\bftheta_0} &= \bigl\{\mu \in \P(\R^2) \mid \exists \, h \in \H\colon \Radon_{\bftheta_0} [\mu] = h_\# \Radon_{\bftheta_0}[\mu_0]\bigr\}, \\
\G_{\bftheta_0} &= \bigl\{\nu \in \P(\R^2) \mid \exists \, h \in \H\colon \Radon_{\bftheta_0} [\nu] = h_\# \Radon_{\bftheta_0}[\nu_0] \bigr\},
\end{align*}
where $\H$ is a convex set of increasing bijections $h \colon \R \to \R$.
Then, any non-empty, finite subsets $\F_0 \subset \F_{\bftheta_0}$ and $\G_0 \subset \G_{\bftheta_0}$ are linearly separable in R-CDT space if $\Radon_{\bftheta_0} \F_{\bftheta_0} \cap \Radon_{\bftheta_0} \G_{\bftheta_0} = \emptyset$ in the sense that
\begin{equation*}
\widehat{\F}_0 = \bigl\{\widehat{\Radon}_{\bftheta_0}[\mu] \colon \R \to \R \mid \mu \in \F_0\bigr\}, \qquad
\widehat{\G}_0 = \bigl\{\widehat{\Radon}_{\bftheta_0}[\nu] \colon \R \to \R \mid \nu \in \G_0\bigr\}
\end{equation*}
are linearly separable in $\Lebesgue^2_\rho(\R)$.
\end{theorem}

\begin{proof}
As $\mu_0, \nu_0 \in \P_2(\R^2)$ have finite second moments, $\Radon_{\bftheta_0} [\mu_0], \Radon_{\bftheta_0} [\nu_0] \in \P_2(\R)$ have finite second moments as well,
and we get $\widehat{\Radon}_{\bftheta_0} [\mu_0], \widehat{\Radon}_{\bftheta_0} [\nu_0] \in \Lebesgue^2_\rho(\R)$ so that, in particular, $\widehat{\F}_{\bftheta_0}, \widehat{\G}_{\bftheta_0} \subset \Lebesgue^2_\rho(\R)$.
We now show that $\widehat{\F}_{\bftheta_0}$ and $\widehat{\G}_{\bftheta_0}$ are convex.
To this end, let $\widehat{p}_1, \widehat{p}_2 \in \widehat{\F}_{\bftheta_0}$ and $\alpha \in [0,1]$.
Then, there exist $\mu_1, \mu_2 \in \F_{\bftheta_0}$ such that $\widehat{p}_1 = \widehat{\Radon}_{\bftheta_0} [\mu_1]$ and $\widehat{p}_2 = \widehat{\Radon}_{\bftheta_0} [\mu_2]$.
Set $p_0 = \Radon_{\bftheta_0} [\mu_0]$, $p_1 = \Radon_{\bftheta_0} [\mu_1]$ and $p_2 = \Radon_{\bftheta_0} [\mu_2]$.
By the definition of $\F_{\bftheta_0}$ there are $h_1, h_2 \in \H$ such that $p_i = (h_i)_\# p_0$ for $i \in \{1,2\}$, where $F_{p_i} = F_{p_0} \circ h_i^{-1}$ so that
\begin{equation*}
\widehat{p}_i = F_{p_i}^{[-1]} \circ F_\rho = h_i \circ \bigl(F_{p_0}^{[-1]} \circ F_\rho\bigr) = h_i \circ \widehat{p}_0.
\end{equation*}
Consequently,
\begin{equation*}
\alpha \widehat{p}_1 + (1-\alpha) \widehat{p}_2 = (\alpha h_1 + (1-\alpha) h_2) \circ \widehat{p}_0 = h_\alpha \circ \widehat{\Radon}_{\bftheta_0} [\mu_0]
\end{equation*}
with $h_\alpha = \alpha h_1 + (1-\alpha) h_2 \in \H$ as $\H$ is convex.
Now, choose $\mu_\alpha \in \F_{\bftheta_0}$ such that $\Radon_{\bftheta_0} [\mu_\alpha] = (h_\alpha)_\# \Radon_{\bftheta_0} [\mu_0]$,
e.g., $\mu_\alpha \coloneqq ( \cdot \, \bftheta_0)_\# (h_\alpha)_\# \Radon_{\bftheta_0} [\mu_0]$.
As before, we then have $\widehat{\Radon}_{\bftheta_0} [\mu_\alpha] = h_\alpha \circ \widehat{\Radon}_{\bftheta_0} [\mu_0]$ and, hence, we conclude that $\alpha \widehat{p}_1 + (1-\alpha) \widehat{p}_2 \in \widehat{\F}_{\bftheta_0}$ so that $\widehat{\F}_{\bftheta_0}$ is indeed convex.
Analogously, we obtain the convexity of $\widehat{\G}_{\bftheta_0}$.
Now, let $\F_0 \subset \F_{\bftheta_0}$ and $\G_0 \subset \G_{\bftheta_0}$ be non-empty and finite.
Then, $\conv(\widehat{\F}_0) \subset \widehat{\F}_{\bftheta_0}$ and $\conv(\widehat{\G}_0) \subset \widehat{\G}_{\bftheta_0}$ are convex and compact.
As $\Radon_{\bftheta_0} \F_{\bftheta_0} \cap \Radon_{\bftheta_0} \G_{\bftheta_0} = \emptyset$, we also have $\widehat{\F}_{\bftheta_0} \cap \widehat{\G}_{\bftheta_0} = \emptyset$ and, in particular, $\conv(\widehat{\F}_0) \cap \conv(\widehat{\G}_0) = \emptyset$.
Therefore, by the Hahn-Banach separation theorem $\conv(\widehat{\F}_0)$ and $\conv(\widehat{\G}_0)$ are linearly separable in $\Lebesgue^2_\rho(\R)$, which implies that also $\widehat{\F}_0$ and $\widehat{\G}_0$ are linearly separable in $\Lebesgue^2_\rho(\R)$.
\end{proof}

\begin{remark}
Note that Theorem~\ref{thm:sep-rcdt} is similar to the result in~\cite{Kolouri2016}.
There, however, the authors consider certain classes of functions that can be linearly separated in R-CDT space when considering {\em all} directions $\bftheta \in \Sphere$.
As opposed to this, our result only needs {\em one} $\bftheta_0 \in \Sphere$ such that the corresponding restricted Radon transforms of the classes are distinguishable.
\end{remark}

\begin{table}[t]
	\caption{Summary of common transformations for $\mu \in \M(\R^2)$ 
        with $a,b > 0$ and $c, \varphi \in \R$.
        The unit circle is parametrized by
        $\bftheta(\vartheta) \coloneqq (\cos(\vartheta), \sin(\vartheta))^\top$.
        The Radon transform for the left half of $\Sphere$ follows by symmetry.}
    \begin{minipage}{\linewidth}
    \resizebox{\linewidth}{!}{
        \begin{tabular}{l@{\enspace}l@{\enspace}l@{\enspace}l}
            \toprule
            transformation 
            & 
            $\bfA$
            &
            $\bfy$
            &
            $\Radon_{\bftheta(\vartheta)} [\mu_{\bfA, \bfy}]$,
            $\vartheta \in (-\tfrac{\pi}{2}, \tfrac{\pi}{2})$
            \\
            \midrule
            translation 
            & 
            $\bfI$ 
            &
            $\R^2$ 
            &
            $\Radon_{\bftheta(\vartheta)} [\mu] \circ (\cdot - \langle \bfy, \bftheta(\vartheta) \rangle)$ 
            \\[1ex]
            rotation
            & 
            $\bigl(\begin{smallmatrix} \cos(\varphi) & -\sin(\varphi) \\ \sin(\varphi) & \hphantom{-}\cos(\varphi) \end{smallmatrix}\bigr)$
            &
            $\bfzero$ 
            & 
            $\Radon_{\bftheta(\vartheta - \varphi)} [\mu]$ 
            \\[1ex]
            reflection 
            &
            $\bigl(\begin{smallmatrix} \cos(\varphi) & \hphantom{-}\sin(\varphi) \\ \sin(\varphi) & -\cos(\varphi) \end{smallmatrix}\bigr)$
            &
            $\bfzero$  
            &
            $\Radon_{\bftheta(\varphi - \vartheta)} [\mu]$ 
            \\[1ex]
            anisotropic scaling 
            & 
            $\bigl(\begin{smallmatrix} a & 0 \\ 0 & b \end{smallmatrix}\bigr)$
            &
            $\bfzero$  
            &
            $\Radon_{\bftheta(\arctan(\frac{b}{a} \tan(\vartheta)))}[\mu] \circ ([a^2 \cos^2(\vartheta) + b^2 \sin^2(\vartheta)]^{-1/2} \cdot)$ 
            \\[1ex]
            vertical shear
            & 
            $\bigl(\begin{smallmatrix} 1 & 0 \\ c & 1 \end{smallmatrix}\bigr)$
            &
            $\bfzero$ 
            & 
            $\Radon_{\bftheta(\arctan(c + \tan(\vartheta)))}[\mu] \circ ([1 + c^2 \cos^2(\vartheta) + c \sin(2\vartheta)]^{-1/2} \cdot)$
            \\
            \bottomrule
        \end{tabular}}
    \end{minipage}
    \label{tab:radon-aff-trans}
\end{table}

To study an example satisfying the assumptions of Theorem~\ref{thm:sep-rcdt}, we consider affinely transformed finite measure $\mu \in \M(\R^2)$.
To this end, let $\bfA \in \GL(2)$, $\bfy \in \R^2$
and define $\mu_{\bfA,\bfy} \in \M(\R^2)$ via
\begin{equation*}
\mu_{\bfA,\bfy} \coloneqq (\bfA \cdot + \bfy)_{\#} \mu = \mu \circ (\bfA^{-1}(\cdot-\bfy)).
\end{equation*}
Then, in~\cite[Proposition 3]{Beckmann2024a} we have shown that,
for any $\bftheta \in \Sphere$, the restricted Radon transform of $\mu_{\bfA,\bfy}$ is given by 
\begin{equation*}
    \Radon_{\bftheta} [\mu_{\bfA,\bfy}] = (\lVert A^\top \bftheta\rVert \cdot + \langle \bfy, \bftheta\rangle)_{\#} \Radon_{\frac{A^\top \bftheta}{\lVert A^\top \bftheta \rVert}}[\mu]
    = \Radon_{\frac{A^\top \bftheta}{\lVert A^\top \bftheta \rVert}}[\mu] \circ \Bigl(\tfrac{\cdot - \langle \bfy, \bftheta\rangle}{\lVert A^\top \bftheta \rVert}\Bigr).
\end{equation*}
The effect of common affine transformations on the Radon transform
is given in Table~\ref{tab:radon-aff-trans}.
In order to describe the deformation with respect to $\bftheta$,
we over-parametrize the unit circle $\Sphere$ via
$\bftheta(\vartheta) \coloneqq (\cos(\vartheta), \sin(\vartheta))^\top$, $\vartheta \in \R$.
We see that an affine transformation essentially causes a translation and scaling 
of the transformed measure 
together with a non-affine mapping in $\bftheta$.

Note that, if $\mu = f \lambda_{\R^2}$ is absolutely continuous with respect to the Lebesgue measure $\lambda_{\R^2}$ with density function $f \in L^1(\R^2)$, we have
\begin{equation*}
\mu_{\bfA,\bfy} = \Bigl(\lvert\det(\bfA)\rvert^{-1} \, f(\bfA^{-1}(\cdot-\bfy))\Bigr) \lambda_{\R^2}
\end{equation*}
and, for any $\bftheta \in \Sphere$,
\begin{equation*}
\Radon_\bftheta[\mu_{\bfA,\bfy}] = \biggl(\tfrac{1}{\lVert A^\top \bftheta \rVert} \Radon_{\frac{A^\top\bftheta}{\lVert A^\top \bftheta \rVert}} [f]\Bigl(\tfrac{\cdot - \langle \bfy, \bftheta\rangle}{\lVert A^\top \bftheta \rVert}\Bigr)\biggr) \lambda_\R,
\end{equation*}
i.e., $\mu_{\bfA,\bfy} \in \M(\R^2)$ and $\Radon_\bftheta[\mu_{\bfA,\bfy}] \in \M(\R)$ are absolutely continuous as well.

\begin{example}
The set of affine linear function $\H = \{x \mapsto ax + b \mid a > 0, ~ b \in \R\}$ satisfies the assumptions of Theorem~\ref{thm:sep-rcdt} and corresponds to translation and isotropic scaling, see Table~\ref{tab:radon-aff-trans}.
\end{example}

\section{Normalized Radon Cumulative Distribution Transform}
\label{sec:NRCDT}

Inspecting Table~\ref{tab:radon-aff-trans} reveals that the only affine transformations that satisfy the assumptions of Theorem~\ref{thm:sep-rcdt} are translation and isotropic scaling.
To account for general affine transformations, in~\cite{Beckmann2024a} we introduced a two-step normalization scheme for the R-CDT, which we recall here for the sake of completeness.
To this end, we define the {\em normalized R-CDT} (NR-CDT)
$\NRCDT [\mu] \colon \R \times \Sphere \to \R$ 
of $\mu \in \P_2(\R^2)$ via
\begin{equation*}
    \NRCDT [\mu](t,\bftheta) 
    \coloneqq
    \frac
    {\widehat{\Radon}_\bftheta[\mu](t) 
    - 
    \mean(\widehat{\Radon}_\bftheta[\mu])}
    {\std(\widehat{\Radon}_\bftheta[\mu])},
    \quad (t,\bftheta) \in \R \times \Sphere,
\end{equation*}
where, for $g \in \Lebesgue^2_\rho(\R)$,
\begin{equation*}
    \mean(g) = \int_\R g(s) \: \d \rho(s),
    \qquad
    \std(g) = \sqrt{\int_\R \lvert g(s) - \mean(g) \rvert^2 \: \d \rho(s)}.
\end{equation*}

To ensure
that the NR-CDT is well defined,
we restrict ourselves to measures
whose support is not contained in a straight line.
More precisely,
we consider the class
\begin{equation*}
    \P_2^*(\R^2) = \{\mu \in \P_2(\R^2) \mid \dim(\mu) > 1\},
\end{equation*}
and show that for these measures
the standard deviation of the restricted Radon transform
is bounded away from zero and
cannot vanish.

\begin{lemma} \label{lem:sigma_bounded}
    Let $\mu \in \P_2^*(\R^2)$.
    Then, there exists a constant $c > 0$ such that
    \begin{equation*}
        \std(\widehat{\Radon}_\bftheta [\mu]) \geq c
        \quad \forall \, \bftheta \in \Sphere.
    \end{equation*}
\end{lemma}

The proof is based on the following continuity result.

\begin{lemma}
    For fixed measure $\mu \in \P_2^*(\R^2)$,
    the functions
    $\Sphere \ni \bftheta \mapsto \mean(\widehat{\Radon}_\bftheta [\mu]) \in \R$
    and
    $\Sphere \ni \bftheta \mapsto \std(\widehat{\Radon}_\bftheta [\mu]) \in \R_{\geq 0}$
    are continuous.
\end{lemma}

\begin{proof}
    We rewrite the mean as
    \begin{equation*}
        \mean(\widehat{\Radon}_\bftheta[\mu])
        =
        \int_\R \widehat{\Radon}_\bftheta [\mu] (t) \: \d \rho(t)
        =
        \int_\R t \: \d \Radon_\bftheta [\mu] (t)
        =
        \int_{\R^2} \langle \bfx, \bftheta \rangle \: \d \mu(\bfx).
    \end{equation*}
    Since the integrand is continuous in $\bftheta$
    and uniformly bounded by 
    $\lvert \langle \cdot, \bftheta \rangle \rvert \le \lVert \cdot \rVert$,
    the dominated convergence theorem yields the assertion.
    Analogously,
    we have
    \begin{equation*}
        \std^2(\widehat{\Radon}_\bftheta [\mu])
        =
        \int_\R 
        \lvert
            \widehat{\Radon}_\bftheta [\mu] (t)
            -
            \mean(\widehat{\Radon}_\bftheta [\mu])
        \rvert^2
        \: \d \rho(t)
        =
        \int_{\R^2} 
        \lvert
            \langle \bfx, \bftheta \rangle
            -
            \mean(\widehat{\Radon}_\bftheta [\mu])
        \rvert^2
        \: \d \mu(\bfx).
    \end{equation*}
    The integrand is again continuous in $\bftheta$
    and uniformly bounded by
    \begin{equation*}
        \lvert
            \langle \bfx, \bftheta \rangle
            -
            \mean(\widehat{\Radon}_\bftheta [\mu])
        \rvert^2
        \le 
        2 \lVert \cdot \rVert^2 
        +
        2 \max_{\bftheta \in \Sphere}
        \mean^2(\widehat{\Radon}_\bftheta [\mu]);
    \end{equation*}
    thus,
    the standard deviation is continuous by dominated convergence.
\end{proof}

\begin{proof}[Proof of Lemma~\ref{lem:sigma_bounded}]
    Assume the contrary,
    this is,
    $c = 0$.
    Then,
    due to the continuity of $\bftheta \mapsto \std(\widehat{\Radon}_\bftheta [\mu])$,
    there exists a minimizing and convergent sequence in $\Sphere$
    whose limit $\bftheta$ is attained and satisfies
    $\std(\widehat{\Radon}_\bftheta [\mu]) = 0$,
    i.e.,
    \begin{equation*}
        \int_{\R^2} \lvert \langle \bfx, \bftheta \rangle - \mean(\widehat{\Radon}_\bftheta [\mu]) \rvert^2 \: \d \mu(\bfx) = 0.
    \end{equation*}
    Hence, the support of $\mu$ is contained in the line $\bigl\{\bfx \in \R^2 \bigm| \langle \bfx, \bftheta \rangle = \mean(\widehat{\Radon}_\bftheta [\mu])\bigr\}$ in contradiction to $\mu \in \P_2^*(\R^2)$.
\end{proof}

Lemma~\ref{lem:sigma_bounded} implies the well-definedness and square-integrability of $\NRCDT[\mu]$.

\begin{proposition} \label{prop:nrcdt_l2}
    Let $\mu \in \P_2^*(\R^2)$.
    Then, 
    $\NRCDT [\mu] \in \Lebesgue^2_{\rho \times u_{\Sphere}}(\R \times \Sphere)$.
\end{proposition}

\begin{proof}
For $\mu \in \P_2^*(\R^2)$ 
we have $\widehat{\Radon} [\mu] \in \Lebesgue^2_{\rho \times u_{\Sphere}}(\R \times \Sphere)$ and, 
due to Lemma~\ref{lem:sigma_bounded}, 
there exists a constant $c > 0$ such that 
\begin{align*}
\std(\widehat{\Radon}_\bftheta [\mu]) \geq c 
\quad \forall \, \bftheta \in \Sphere.
\end{align*}
Hence,
\begin{align*}
    \|\NRCDT [\mu]\|_{\rho \times u_{\Sphere}}^2 
    &\leq c^{-2} \int_{\Sphere} \int_\R \lvert \widehat{\Radon}_\bftheta [\mu] (t) 
    - 
    \mean(\widehat{\Radon}_\bftheta [\mu]) \rvert^2 \: \d \rho(t) \, \d u_{\Sphere}(\bftheta) \\
    &\leq 4c^{-2} \, \|\widehat{\Radon} [\mu]\|_{\rho \times u_{\Sphere}}^2 < \infty
\end{align*}
so that $\NRCDT [\mu] \in \Lebesgue^2_{\rho \times u_{\Sphere}}(\R \times \Sphere)$, as stated.
\end{proof}

\subsection{Max-normalized R-CDT}
\label{ssec:mNRCDT}

As in~\cite{Beckmann2024a}, we now consider the more restricted class
\begin{equation*}
    \P_c^*(\R^2) = \{\mu \in \P(\R^2) \mid \supp(\mu) \subsubset \R^2 \land \dim(\mu) > 1\}
\end{equation*}
and define the {\em max-normalized R-CDT} (\mNRCDT) $\maxNRCDT [\mu] \colon \R \to \R$
via
\begin{equation*}
    \maxNRCDT [\mu](t) 
    \coloneqq  
    \max_{\bftheta \in \Sphere} \NRCDT [\mu](t,\bftheta)
    \quad \mbox{for } t \in \R.
\end{equation*}
In~\cite[Proposition 6]{Beckmann2024a} we have seen that $\maxNRCDT [\mu] \in L^\infty_\rho(\R)$ for all $\mu \in \P_c^*(\R^2)$ and we now focus on the linear separability of classes induced by affine transformations of template measures.

\paragraph{Linear separability in \mNRCDT{} space}

Let $\mu_0 \in \P_2^*(\R^2)$ be a template measure and assume that $\mu \in \P_2(\R^2)$ satisfies
\begin{equation*}
\Radon_{\bftheta} [\mu] = (a_\bftheta \cdot + b_\bftheta)_{\#} \Radon_{h(\bftheta)}[\mu]
\quad \mbox{ with } a_\bftheta > 0, ~ b_\bftheta \in \R,
\end{equation*}
where $h \colon \Sphere \to \Sphere$ is bijective.
Then,
\begin{equation*}
\widehat{\Radon} [\mu](t,\bftheta) = a_\bftheta \, \widehat{\Radon} [\mu_0](t,h(\bftheta)) + b_\bftheta
\end{equation*}
so that
\begin{equation*}
\mean(\widehat{\Radon} [\mu](\cdot,\bftheta)) = a_\bftheta \, \mean(\widehat{\Radon} [\mu_0](\cdot,h(\bftheta))) + b_\bftheta,
\quad
\std(\widehat{\Radon} [\mu](\cdot,\bftheta)) = a_\bftheta \, \std(\widehat{\Radon} [\mu_0](\cdot,h(\bftheta))).
\end{equation*}
Consequently,
\begin{equation*}
\NRCDT [\mu](t,\bftheta) = \frac{\widehat{\Radon} [\mu_0](t,h(\bftheta)) - \mean(\widehat{\Radon} [\mu_0](\cdot,h(\bftheta)))}{\std(\widehat{\Radon} [\mu_0](\cdot,h(\bftheta)))} = \NRCDT [\mu_0](t,h(\bftheta))
\end{equation*}
and, if $\mu_0 \in \P_c^*(\R^2)$,
\begin{equation*}
\maxNRCDT [\mu](t) = \max_{\bftheta \in \Sphere} \NRCDT [\mu](t,\bftheta) = \max_{\bftheta \in \Sphere} \NRCDT [\mu_0](t,h(\bftheta)) = \maxNRCDT [\mu_0](t).
\end{equation*}
This observation implies linear separability in \mNRCDT{} space if we consider classes in $\P_c^*(\R^2)$ generated by arbitrary affine-linear transforms, which has been shown in~\cite{Beckmann2024a}.

\begin{theorem}[cf.~{\cite[Theorem 1]{Beckmann2024a}}] 
\label{thm:sep-max-nrcdt}
For template measures $\mu_0, \nu_0 \in \P_c^*(\R^2)$ with
\begin{equation*}
\maxNRCDT [\mu_0] \neq \maxNRCDT [\nu_0]
\end{equation*}
consider the classes
\begin{align*}
\F &= \bigl\{\mu \in \P(\R^2) \mid \exists \, \bfA \in \GL(2), \, \bfy \in \R^2 \colon \mu = (\bfA \cdot + \bfy)_\# \mu_0\bigr\}, \\
\G &= \bigl\{\nu \in \P(\R^2) \mid \exists \, \bfA \in \GL(2), \, \bfy \in \R^2 \colon \nu = (\bfA \cdot + \bfy)_\# \nu_0\bigr\}.
\end{align*}
Then, any non-empty subsets $\F_0 \subseteq \F$ and $\G_0 \subseteq \G$ are linearly separable in \mNRCDT{} space.
\end{theorem}

The proof of Theorem~\ref{thm:sep-max-nrcdt} shows that \mNRCDT{}
maps $\F$ and $\G$ to one-point sets.
More precisely,
\begin{equation*}
    \maxNRCDT[\F] = \{\maxNRCDT [\mu_0]\}
    \qquad \text{and} \qquad
    \maxNRCDT[\G] = \{\maxNRCDT [\nu_0]\}.
\end{equation*}
In the next step, we consider the linear separability of two generated classes when allowing for slight perturbations of the underlying template measures.

\paragraph{Linear separability under perturbations in Wasserstein space}

To study the uncertainty of the max-normalized R-CDT
under perturbations 
with respect to the Wasserstein-$\infty$ distance,
we first analyse how these effect the non-normalized R-CDT.

\begin{proposition} 
    \label{prop:quant-w-inf}
    Let $\mu_0, \mu_\epsilon \in \mathcal{P}(\R^2)$ 
    with $W_\infty(\mu_0, \mu_\epsilon) \le \epsilon$.
    Then 
    \begin{equation*}
        \lVert 
        \widehat{\Radon}_\bftheta [\mu_0] 
        - 
        \widehat{\Radon}_\bftheta [\mu_\epsilon] 
        \rVert_\infty 
        \leq 
        \epsilon.
    \end{equation*}
\end{proposition}

\begin{proof}
    For any measure $\nu \in \P(\R)$,
    the cumulative distribution and the quantile function fulfil
    \begin{equation*}
        t \le F_\nu \bigl(F_\nu^{[-1]}(t)\bigr)
        \quad \forall t \in (0,1)
        \qquad\text{and}\qquad
        F_\nu^{[-1]} \bigl(F_\nu (s)\bigr) \le s
        \quad \forall s \in \R.
    \end{equation*}
    Exploiting 
    $W_\infty(\Radon_\bftheta [\mu_0], \Radon_\bftheta [\mu_\epsilon]) 
    \le 
    W_\infty (\mu_0, \mu_\epsilon) \le \epsilon$
    due to Proposition~\ref{prop:wass},
    and utilizing any $W_\infty$ optimal plan 
    $\pi_\bftheta \in \Pi(\Radon_\bftheta [\mu_0], \Radon_\bftheta [\mu_\epsilon])$,
    we observe
    \begin{align*}
        F_{\Radon_\bftheta[\mu_0]}(s) 
        &=
        \pi_\bftheta ((-\infty, s] \times \R)
        =
        \pi_\bftheta ((-\infty, s] \times (-\infty, s + \epsilon])
        \\
        &\le
        \pi_\bftheta (\R \times (-\infty, s + \epsilon])
        =
        F_{\Radon_\bftheta [\mu_\epsilon]}(s + \epsilon)
    \end{align*}
    for all $s \in \R$.
    Because of the monotonicity,
    we further have
    \begin{align*}
        F_{\Radon_\bftheta [\mu_\epsilon]}^{[-1]}(t)
        & \le
        F_{\Radon_\bftheta [\mu_\epsilon]}^{[-1]}
        \bigl(F_{\Radon_\bftheta [\mu_0]}
        \bigl(F_{\Radon_\bftheta [\mu_0]}^{[-1]}(t)\bigr)\bigr) \\
        & \le
        F_{\Radon_\bftheta [\mu_\epsilon]}^{[-1]}
        \bigl(F_{\Radon_\bftheta [\mu_\epsilon]}
        \bigl(F_{\Radon_\bftheta [\mu_0]}^{[-1]}(t) + \epsilon\bigr)\bigr)
        \le
        F_{\Radon_\bftheta [\mu_0]}^{[-1]}(t) + \epsilon
    \end{align*}
    Interchanging the role of $\mu_0$ and $\mu_\epsilon$,
    we obtain the lower bound,
    which yields the assertion.
\end{proof}

To simplify the notation,
for $\mu \in \P_1(\R^2)$,
we introduce the zero-mean quantile functions
\begin{equation*}
    \widetilde\NRCDT_\bftheta [\mu] (t)
    \coloneqq
    \widehat{\Radon}_\bftheta [\mu] (t) 
    - 
    \mean(\widehat{\Radon}_\bftheta [\mu])
    \quad \forall \, \bftheta \in \Sphere
    ~ \forall \, t \in \R,
\end{equation*}
which corresponds to the first normalization step
of the normalized R-CDT.
Thus, we now transfer the estimate for R-CDT to the zero-mean quantile functions.

\begin{lemma} \label{lem:zero-quant-w-infty}
    Let $\mu_0, \mu_\epsilon \in \P_1(\R^2)$ 
    with $W_\infty (\mu_0, \mu_\epsilon) \le \epsilon$.
    Then
    \begin{equation*}
        \lVert 
        \widetilde\NRCDT_\bftheta [\mu_0] 
        -
        \widetilde\NRCDT_\bftheta [\mu_\epsilon]
        \rVert_\infty 
        \le 
        2 \epsilon.
    \end{equation*}
\end{lemma}

\begin{proof}
    Due to Proposition~\ref{prop:quant-w-inf},
    we have
    \begin{align*}
        \lVert 
            \widetilde\NRCDT_\bftheta [\mu_0] 
            -
            \widetilde\NRCDT_\bftheta [\mu_\epsilon]
        \rVert_\infty 
        &\le 
        \lVert 
            \widehat{\Radon}_\bftheta [\mu_0]
            -
            \widehat{\Radon}_\bftheta [\mu_\epsilon]
        \rVert_\infty
        +
        \lvert
            \mean(\widehat{\Radon}_\bftheta [\mu_0])
            -
            \mean(\widehat{\Radon}_\bftheta [\mu_\epsilon])
        \rvert
        \\
        &\le
        \lVert 
            \widehat{\Radon}_\bftheta [\mu_0]
            -
            \widehat{\Radon}_\bftheta [\mu_\epsilon]
        \rVert_\infty
        +
        \int_\R
        \lvert
            \widehat{\Radon}_\bftheta [\mu_0] (t)
            -
            \widehat{\Radon}_\bftheta [\mu_\epsilon] (t)
        \rvert
        \: \d \rho(t) 
        \le
        2 \epsilon.
        \tag*{\qed}
    \end{align*}
\end{proof}

For the next normalization step in \mNRCDT{},
we have to divide the zero-mean quantile functions
by their standard deviations
\begin{equation*}
    \std(\widehat{\Radon}_\bftheta [\mu])
    = 
    \std(\widetilde\NRCDT_\bftheta [\mu])
    =
    \lVert \widetilde\NRCDT_\bftheta [\mu] \rVert_\rho.
\end{equation*}
Depending on the smallest occurring standard deviation,
a perturbation of the initial measure $\mu_0$ may cause
the following variation 
of the max-normalized R-CDT.

\begin{proposition} \label{prop:max-nrcdt-w-infty}
    Let $\mu_0 \in \P_c^*(\R^2)$
    and $\mu_\epsilon \in \P(\R^2)$
    with $W_\infty(\mu_0, \mu_\epsilon) \le \epsilon$,
    and define 
    $c_0 \coloneqq \min_{\bftheta \in \Sphere} \std(\widehat{\Radon}_\bftheta [\mu_0]) > 0$.
    If $\epsilon < c_0 / 2$,
    then
    \begin{equation*}
        \lVert
            \maxNRCDT [\mu_0]
            -
            \maxNRCDT [\mu_\epsilon]
        \rVert_\infty
        \le
        \frac{\diam(\mu_0) + 2\epsilon}{c_0(c_0 - 2\epsilon)} \, 4\epsilon.
    \end{equation*}
\end{proposition}

\begin{proof}
    Since $W_\infty(\mu_0, \mu_\epsilon) \le \epsilon$ and $\mu_0 \in \P_c^*(\R^2)$, the measure $\mu_\epsilon \in \P(\R^2)$ has compact support and, in particular, $\mu_\epsilon \in \P_1(\R^2)$.
    Employing Lemma~\ref{lem:zero-quant-w-infty}, we have
    \begin{equation} \label{eq:err-std}
        \bigl\lvert
            \lVert \widetilde\NRCDT_\bftheta [\mu_0] \rVert_\rho
            -
            \lVert \widetilde\NRCDT_\bftheta [\mu_\epsilon] \rVert_\rho
        \bigr\rvert
        \le
        \lVert 
            \widetilde\NRCDT_\bftheta [\mu_0] 
            - 
            \widetilde\NRCDT_\bftheta [\mu_\epsilon]
        \rVert_\rho
        \le
        \sqrt{ \int_\R 4 \epsilon^2 \: \d \rho(t)}
        =
        2 \epsilon.
    \end{equation}
    Moreover, $\std(\widehat{\Radon}_\bftheta [\mu_\epsilon]) = \lVert \widetilde\NRCDT_\bftheta [\mu_\epsilon] \rVert_\rho$ is bounded away from zero and, thus, $\maxNRCDT [\mu_\epsilon]$ is well defined.
    Indeed, Lemma~\ref{lem:zero-quant-w-infty} in combination with $2\epsilon < c_0$ gives
    \begin{equation*}
        \lVert \widetilde\NRCDT_\bftheta [\mu_\epsilon] \rVert_\rho \geq \lVert \widetilde\NRCDT_\bftheta [\mu_0] \rVert_\rho - \lVert \widetilde\NRCDT_\bftheta [\mu_0] - \widetilde\NRCDT_\bftheta [\mu_\epsilon] \rVert_\rho \geq c_0 - 2\epsilon > 0.
    \end{equation*}
    On the basis of \eqref{eq:err-std},
    the perturbation after the second normalization step is for all $t \in \R$ bounded by
    \begin{align*}
        \lvert
            \NRCDT_\bftheta [\mu_0] (t) - \NRCDT_\bftheta [\mu_\epsilon] (t)
        \lvert
        &\le
        \biggl\lvert
            \frac
            {\widetilde\NRCDT_\bftheta [\mu_0](t)}
            {\lVert \widetilde\NRCDT_\bftheta [\mu_0] \rVert_\rho}
            -
            \frac
            {\widetilde\NRCDT_\bftheta [\mu_\epsilon](t)}
            {\lVert \widetilde\NRCDT_\bftheta [\mu_0] \rVert_\rho}
        \biggr\rvert
        +
        \biggl\lvert
            \frac
            {\widetilde\NRCDT_\bftheta [\mu_\epsilon](t)}
            {\lVert \widetilde\NRCDT_\bftheta [\mu_0] \rVert_\rho}
            -
            \frac
            {\widetilde\NRCDT_\bftheta [\mu_\epsilon](t)}
            {\lVert \widetilde\NRCDT_\bftheta [\mu_\epsilon] \rVert_\rho}
        \biggr\rvert
        \\
        &\le
        \frac{2\epsilon}{\lVert \widetilde\NRCDT_\bftheta [\mu_0] \rVert_\rho}
        +
        \lvert \widetilde\NRCDT_\bftheta [\mu_\epsilon] (t)\rvert
        \,
        \frac{2 \epsilon}
        {\lVert \widetilde\NRCDT_\bftheta [\mu_0] \rVert_\rho 
        \lVert \widetilde\NRCDT_\bftheta [\mu_\epsilon] \rVert_\rho}
        \\
        &\le
        2 \epsilon \frac
        {\lVert \widetilde\NRCDT_\bftheta [\mu_\epsilon] \rVert_\rho 
        +
        \lVert \widetilde\NRCDT_\bftheta [\mu_\epsilon] \rVert_\infty}
        {c_0 (c_0 - 2 \epsilon)}
        \le
        2 \epsilon \frac
        {(\lVert \widetilde\NRCDT_\bftheta [\mu_0] \rVert_\rho + 2 \epsilon)
        +
        (\diam(\mu_0) + 2 \epsilon)}
        {c_0 (c_0 - 2 \epsilon)}
        \\
        &\le
        \frac{\diam(\mu_0) + 2\epsilon}{c_0(c_0 - 2\epsilon)} \, 4\epsilon.
    \end{align*}
    Since this upper bound is independent of $\bftheta$,
    for fixed $t \in \R$,
    it remains valid for the supremum over $\bftheta \in \Sphere$,
    which yields the assertion.
\end{proof}
    
An affine transformation only causes 
a non-linear deformation of the normalized R-CDT  
in the argument $\bftheta$;
therefore
the estimate in Proposition~\ref{prop:max-nrcdt-w-infty} 
can be immediately generalized to small perturbations of the initial measure
followed by an affine transformation.

\begin{corollary} 
\label{cor:max-nrcdt-w-infty}
    Let $\mu_0 \in \P_c^*(\R^2)$
    and $\mu_\epsilon \in \P(\R^2)$
    with $W_\infty(\mu_0, \mu_\epsilon) \le \epsilon$,
    and define 
    $c_0 \coloneqq \min_{\bftheta \in \Sphere} \std(\widehat{\Radon}_\bftheta [\mu_0]) > 0$.
    For $\bfA \in \GL(2)$, $\bfy \in \R^2$,
    let $\mu \coloneqq (\bfA \cdot + \bfy)_\# \mu_\epsilon$.
    If $2\epsilon < c_0$,
    then
    \begin{equation*}
        \lVert
            \maxNRCDT [\mu_0]
            -
            \maxNRCDT [\mu]
        \rVert_\infty
        \le
        \frac{\diam(\mu_0) + 2\epsilon}{c_0(c_0 - 2\epsilon)} \, 4\epsilon.
    \end{equation*}    
\end{corollary}

The uniform bounds for the uncertainty of the max-normalized R-CDT
can be incorporated into the separation guarantee
in Theorem~\ref{thm:sep-max-nrcdt} to obtain linear separability even in the cases of slight perturbations in $W_\infty$.

\begin{theorem} \label{thm:sep-pert-max-nrcdt}
    For template measures $\mu_0, \nu_0 \in \P_c^*(\R^2)$
    with $\maxNRCDT [\mu_0] \neq \maxNRCDT [\nu_0]$,
    define 
    $c_\mu \coloneqq \min_{\bftheta\in\Sphere} \std(\widehat{\Radon}_\bftheta [\mu_0] )$
    and
    $c_\nu \coloneqq \min_{\bftheta\in\Sphere} \std(\widehat{\Radon}_\bftheta [\nu_0])$.
    Let $\epsilon < \tfrac{\min\{c_\mu, c_\nu\}}{2}$ satisfy
    \begin{equation*}
		4 \epsilon        
        \biggl(
            \frac
            {\diam(\mu_0) + 2\epsilon}
            {c_\mu (c_\mu - 2 \epsilon)}
            +
            \frac
            {\diam(\nu_0) + 2\epsilon}
            {c_\nu (c_\nu - 2 \epsilon)}
        \biggr)
        <
        \lVert 
            \maxNRCDT [\mu_0] - \maxNRCDT [\nu_0]
        \rVert_\infty.
    \end{equation*}
    Consider the classes
    \begin{align*}
        \F 
        &= 
        \bigl\{
            (\bfA \cdot + \bfy)_\# \mu
            \mid 
            \bfA \in \GL(2), 
            \, \bfy \in \R^2,
            \, \mu \in \P(\R^2),
            \, W_\infty(\mu, \mu_0) \le \epsilon 
        \bigr\}, 
        \\
        \G 
        &= 
        \bigl\{
            (\bfA \cdot + \bfy)_\# \nu 
            \mid 
            \bfA \in \GL(2), 
            \, \bfy \in \R^2, 
            \, \nu \in \P(\R^2),
            \, W_\infty(\nu, \nu_0) \le \epsilon
        \bigr\}.
    \end{align*}
    Then,
    any non-empty subsets
    $\F_0 \subseteq \F$ and $\G_0 \subseteq \G$ 
    are linearly separable in \mNRCDT{} space.
\end{theorem}

\begin{proof}
For $\mu \in \F$ there exist $\bfA \in \GL(2)$, $\bfy \in \R^2$ and $\mu_\epsilon \in \P(\R^2)$ with $W_\infty(\mu_\epsilon, \mu_0) \le \epsilon$ such that $\mu = (\bfA \cdot + \bfy)_\# \mu_\epsilon$.
Consequently, Corollary~\ref{cor:max-nrcdt-w-infty} yields
\begin{equation*}
\lVert \maxNRCDT [\mu_0] - \maxNRCDT [\mu] \rVert_\infty \le \frac{\diam(\mu_0) + 2\epsilon}{c_\mu(c_\mu - 2\epsilon)} \, 4\epsilon = C_{\mu,\epsilon}
\end{equation*}
so that
$\maxNRCDT \F \subseteq B_{C_{\mu,\epsilon}}(\maxNRCDT [\mu_0]) \subset \Lebesgue^\infty_\rho(\R)$.
Analogously, for $\nu \in \G$ we obtain
\begin{equation*}
\lVert \maxNRCDT [\nu_0] - \maxNRCDT [\nu] \rVert_\infty \le \frac{\diam(\nu_0) + 2\epsilon}{c_\nu(c_\nu - 2\epsilon)} \, 4\epsilon = C_{\nu,\epsilon}
\end{equation*}
and $\maxNRCDT \G \subseteq B_{C_{\nu,\epsilon}}(\maxNRCDT [\nu_0]) \subset \Lebesgue^\infty_\rho(\R)$.
Since $\lVert \maxNRCDT [\mu_0] - \maxNRCDT [\nu_0] \rVert_\infty > C_{\mu,\epsilon} + C_{\nu,\epsilon}$, the closed balls $B_{C_{\mu,\epsilon}}(\maxNRCDT [\mu_0])$ and $B_{C_{\nu,\epsilon}}(\maxNRCDT [\nu_0])$ are linearly separable in $\Lebesgue^\infty_\rho(\R)$.
This implies the linear separability of $\maxNRCDT \F_0$ and $\maxNRCDT \G_0$ in $\Lebesgue^\infty_\rho(\R)$ for any non-empty subsets $\F_0 \subseteq \F$ and $\G_0 \subseteq \G$.
\end{proof}

\subsection{Mean-normalized R-CDT}
\label{ssec:aNRCDT}

To deal with a more general measure $\mu \in \P_2^*(\R^2)$ and perturbations in $W_2$, we define the {\em mean-normalized R-CDT} (\aNRCDT) $\meanNRCDT [\mu] \colon \R \to \R$ via
\begin{equation*}
\meanNRCDT [\mu](t) = \int_{\Sphere} \NRCDT [\mu](t,\bftheta) \: \d u_{\Sphere} (\bftheta)
\quad \mbox{for } t \in \R.
\end{equation*}

Proposition~\ref{prop:nrcdt_l2} implies the well-definedness and square-integrability of $\meanNRCDT[\mu]$.

\begin{lemma}
Let $\mu \in \P_2^*(\R^2)$.
Then, $\meanNRCDT [\mu] \in L^2_\rho(\R)$.
\end{lemma}

\begin{proof}
For $\mu \in \P_2^*(\R^2)$ we have $\NRCDT [\mu] \in \Lebesgue^2_{\rho \times u_{\Sphere}}(\R \times \Sphere)$ according to Proposition~\ref{prop:nrcdt_l2}.
Consequently, Jensen's inequality gives
\begin{equation*}
\lVert\meanNRCDT [\mu]\rVert_\rho^2 \leq \int_\R \int_{\Sphere} \lvert\NRCDT [\mu](t,\bftheta)\rvert^2 \: \d u_{\Sphere} (\bftheta) = \|\NRCDT [\mu]\|_{\rho \times u_{\Sphere}}^2 < \infty
\end{equation*}
so that $\meanNRCDT [\mu] \in L^2_\rho(\R)$, as stated.
\end{proof}

We again focus on the linear separability of classes generated by template measures.

\paragraph{Linear separability in \aNRCDT{} space}

The linear separability in \aNRCDT{} space of classes in $\P_2^*(\R^2)$ requires more restrictions on the admissible affine-linear transforms.

\begin{theorem} 
\label{thm:sep-mean-nrcdt}
For template measures $\mu_0, \nu_0 \in \P_2^*(\R^2)$ with
\begin{equation*}
\meanNRCDT [\mu_0] \neq \meanNRCDT [\nu_0]
\end{equation*}
consider the classes
\begin{align*}
\F &= \bigl\{\mu \in \P(\R^2) \mid \exists \, \bfA \in \X, \, \bfy \in \R^2 \colon \mu = (\bfA \cdot + \bfy)_\# \mu_0\bigr\}, \\
\G &= \bigl\{\nu \in \P(\R^2) \mid \exists \, \bfA \in \X, \, \bfy \in \R^2 \colon \nu = (\bfA \cdot + \bfy)_\# \nu_0\bigr\},
\end{align*}
where, for some $c \in (0,\frac{1}{2})$,
\begin{equation*}
\X = \biggl\{\bfA \in \GL(2) \biggm| \frac{\sigma_{\max}(\bfA) - \sigma_{\min}(\bfA)}{\sigma_{\min}(\bfA)} \leq \frac{c \lVert\meanNRCDT [\mu_0] - \meanNRCDT [\nu_0]\lVert_\rho}{\max\{\lVert\NRCDT [\mu_0]\rVert_{\rho \times u_{\Sphere}}, \lVert\NRCDT [\nu_0]\rVert_{\rho \times u_{\Sphere}}\}} \biggr\}.
\end{equation*}
Then, any non-empty subsets $\F_0 \subseteq \F$ and $\G_0 \subseteq \G$ are linearly separable in \aNRCDT{} space.
\end{theorem}

\begin{proof}
For $\mu \in \F$ there exist $\bfA \in \X$ and $\bfy \in \R^2$ with $\mu = (\bfA \cdot + \bfy)_\# \mu_0$ implying that
\begin{equation*}
\NRCDT [\mu](t,\bftheta) = \NRCDT [\mu_0](t,h(\bftheta))
\quad
\text{with}
\quad
h(\bfx) = \frac{\bfA^\top\bfx}{\lVert\bfA^\top\bfx\rVert},
\quad \bfx \in \R^2 \setminus \{0\}.
\end{equation*}
Setting $\bftheta(\varphi) = (\cos(\varphi), \sin(\varphi))^\top \in \Sphere$ for $\varphi \in [0,2\pi)$, the definition of $\meanNRCDT$ gives, for all $t \in \R$,
\begin{equation*}
\meanNRCDT [\mu](t) = \int_{\Sphere} \NRCDT [\mu](t,\bftheta) \: \d u_{\Sphere}(\bftheta) = \frac{1}{2\pi} \int_0^{2\pi} \NRCDT [\mu_0](t,h(\bftheta(\varphi))) \: \d \varphi.
\end{equation*}
Now, the parametrization $\eta \colon [0,2\pi) \to \Sphere, ~ \varphi \mapsto h(\bftheta(\varphi))$
is bijective and continuously differentiable on $(0,2\pi)$ with
\begin{equation*}
\dot{\eta}(\varphi) = \D h(\bftheta(\varphi)) \dot{\bftheta}(\varphi) = \frac{1}{\lVert\bfA^\top \bftheta(\varphi)\rVert} \biggl(\bfA^\top \dot{\bftheta}(\varphi) - \frac{\bfA^\top \bftheta(\varphi) \bftheta(\varphi)^\top \bfA \bfA^\top \dot{\bftheta}(\varphi)}{\lVert\bfA^\top \bftheta(\varphi)\rVert^2}\biggr)
\end{equation*}
so that
\begin{equation*}
\lVert\dot{\eta}(\varphi)\rVert^2 = \frac{\lVert\bfA^\top \bftheta(\varphi)\rVert^2 \, \lVert\bfA^\top \dot{\bftheta}(\varphi)\rVert^2 - \lvert\langle \bfA^\top \bftheta(\varphi), \bfA^\top \dot{\bftheta}(\varphi) \rangle\rvert^2}{\rVert\bfA^\top \bftheta(\varphi)\rVert^4}.
\end{equation*}
Consequently,
\begin{equation*}
\meanNRCDT [\mu_0](t) = \int_{\Sphere} \NRCDT [\mu_0](t,\bftheta) \: \d u_{\Sphere}(\bftheta) = \frac{1}{2\pi} \int_0^{2\pi} \NRCDT [\mu_0](t,\eta(\varphi)) \, \lVert\dot{\eta}(\varphi)\rVert \: \d \varphi,
\end{equation*}
which implies that
\begin{align*}
\meanNRCDT [\mu](t) - \meanNRCDT [\mu_0](t) &= \frac{1}{2\pi} \int_0^{2\pi} \NRCDT [\mu_0](t,\eta(\varphi)) \, (1 - \lVert\dot{\eta}(\varphi)\rVert) \: \d \varphi \\
&= \frac{1}{2\pi} \int_0^{2\pi} \NRCDT [\mu_0](t,\eta(\varphi)) \, \lVert\dot{\eta}(\varphi)\rVert \, \bigl(\lVert\dot{\eta}(\varphi)\rVert^{-1} - 1\bigr) \: \d \varphi
\end{align*}
and, hence, with $c_\bfA = \max_{\varphi \in [0,2\pi)} \lvert 1 - \lVert\dot{\eta}(\varphi)\rVert^{-1} \rvert$ follows that
\begin{equation*}
\bigl\lvert (\meanNRCDT [\mu] - \meanNRCDT [\mu_0])(t) \bigr\rvert \leq \frac{c_\bfA}{2\pi} \int_0^{2\pi} \lvert \NRCDT [\mu_0](t,\eta(\varphi)) \rvert \, \lVert\dot{\eta}(\varphi)\rVert \: \d \varphi = c_\bfA \int_{\Sphere} \lvert \NRCDT [\mu_0](t,\bftheta) \rvert \: \d u_{\Sphere}(\bftheta).
\end{equation*}
Thereon, Hölder's inequality gives
\begin{equation*}
\bigl\lvert (\meanNRCDT [\mu] - \meanNRCDT [\mu_0])(t) \bigr\rvert^2 \leq c_\bfA^2 \biggl(\int_{\Sphere} \lvert \NRCDT [\mu_0](t,\bftheta) \rvert \: \d u_{\Sphere}(\bftheta)\biggr)^2 \leq c_\bfA^2 \int_{\Sphere} \lvert \NRCDT [\mu_0](t,\bftheta) \rvert^2 \: \d u_{\Sphere}(\bftheta)
\end{equation*}
so that
\begin{equation*}
\lVert \meanNRCDT [\mu] - \meanNRCDT [\mu_0] \rVert_\rho \leq \Bigl(\max_{\varphi \in [0,2\pi)} \bigl\lvert 1 - \lVert\dot{\eta}(\varphi)\rVert^{-1} \bigr\rvert\Bigr) \, \lVert \NRCDT [\mu_0] \lVert_{\rho \times u_{\Sphere}}.
\end{equation*}
Direct calculations show that
\begin{equation*}
\lVert\bfA^\top \bftheta(\varphi)\rVert^2 \, \lVert\bfA^\top \dot{\bftheta}(\varphi)\rVert^2 - \lvert\langle \bfA^\top \bftheta(\varphi), \bfA^\top \dot{\bftheta}(\varphi) \rangle\rvert^2 = \lvert \det(\bfA) \rvert^2
\end{equation*}
and, thus,
\begin{equation*}
\lVert\dot{\eta}(\varphi)\rVert = \frac{\lvert \det(\bfA) \rvert}{\rVert\bfA^\top \bftheta(\varphi)\rVert^2}
\end{equation*}
with $\lvert \det(\bfA) \rvert = \sigma_{\min}(\bfA) \, \sigma_{\max}(\bfA)$ and $\rVert\bfA^\top \bftheta(\varphi)\rVert_2 \in [\sigma_{\min}(\bfA), \sigma_{\max}(\bfA)]$ for all $\varphi \in [0,2\pi)$.
This gives
\begin{equation*}
\lVert\dot{\eta}(\varphi)\rVert^{-1} \in \biggl[\frac{\sigma_{\min}(\bfA)}{\sigma_{\max}(\bfA)}, \frac{\sigma_{\max}(\bfA)}{\sigma_{\min}(\bfA)}\biggr]
\quad \forall \, \varphi \in [0,2\pi),
\end{equation*}
which in turn implies
\begin{equation*}
\max_{\varphi \in [0,2\pi)} \bigl\lvert 1 - \lVert\dot{\eta}(\varphi)\rVert^{-1} \bigr\rvert \leq \frac{\sigma_{\max}(\bfA)}{\sigma_{\min}(\bfA)} - 1 = \frac{\sigma_{\max}(\bfA) - \sigma_{\min}(\bfA)}{\sigma_{\min}(\bfA)}
\end{equation*}
and the assumption $\bfA \in \X$ guarantees that
\begin{align}
\label{eq:nrcdt-sig-val}
\lVert \meanNRCDT [\mu] - \meanNRCDT [\mu_0] \rVert_\rho &\leq \frac{\sigma_{\max}(\bfA) - \sigma_{\min}(\bfA)}{\sigma_{\min}(\bfA)} \, \lVert \NRCDT [\mu_0] \lVert_{\rho \times u_{\Sphere}} \\
&\leq c \lVert\meanNRCDT [\mu_0] - \meanNRCDT [\nu_0]\lVert_\rho \eqqcolon r_0.
\nonumber
\end{align}
Consequently,
$
\meanNRCDT \F \subseteq B_{r_0}(\meanNRCDT [\mu_0]) \subset \Lebesgue^2_\rho(\R)
$
and, analogously,
$
\meanNRCDT \G \subseteq B_{r_0}(\meanNRCDT [\nu_0]) \subset \Lebesgue^2_\rho(\R).
$
Since $c \in (0,\frac{1}{2})$, $B_{r_0}(\meanNRCDT [\mu_0])$ and $B_{r_0}(\meanNRCDT [\nu_0])$ are linearly separable in $\Lebesgue^2_\rho(\R)$.
This implies the linear separability of $\meanNRCDT \F_0$ and $\meanNRCDT \G_0$ in $\Lebesgue^2_\rho(\R)$ for any non-empty $\F_0 \subseteq \F$ and $\G_0 \subseteq \G$.
\end{proof}

In the next step, we again consider the linear separability of two generated classes when allowing for slight perturbations of the underlying template measures.

\paragraph{Linear separability under perturbations in Wasserstein space}

To study the uncertainty of the mean-normalized R-CDT
under perturbations 
with respect to the Wasserstein-$2$ distance,
we again start with analysing how these effect the non-normalized R-CDT.

\begin{proposition} 
\label{prop:zero-quant-w-2}
    Let $\mu_0, \mu_\epsilon \in \P_2(\R^d)$ 
    with $W_2(\mu_0, \mu_\epsilon) \le \epsilon$.
    Then
    \begin{equation*}
        \lVert 
            \widehat{\Radon}_\bftheta [\mu_0]
            -
            \widehat{\Radon}_\bftheta [\mu_\epsilon]
        \rVert_\rho
        \le
        \epsilon.
    \end{equation*}
\end{proposition}

\begin{proof}
    The statement follows from Proposition~\ref{prop:wass} via
    \begin{align*}
        \lVert
            \widehat\Radon_\bftheta [\mu_0]
            -
            \widehat\Radon_\bftheta [\mu_\epsilon]
        \rVert_\rho
        =
        W_2(\Radon_\bftheta [\mu_0], \Radon_\bftheta [\mu_\epsilon])
        \le
        W_2(\mu_0, \mu_\epsilon)
        \le
        \epsilon.
        \tag*{\qed}
    \end{align*}
\end{proof}

Next, we transfer the estimate for the R-CDT to the zero-mean quantile functions.

\begin{lemma} \label{lem:zero-quant-w-2}
    Let $\mu_0, \mu_\epsilon \in \P_2(\R^2)$
    with $W_2(\mu_0, \mu_\epsilon) \le \epsilon$.
    Then
    \begin{equation*}
        \lVert
            \widetilde\NRCDT_\bftheta [\mu_0]
            -
            \widetilde\NRCDT_\bftheta [\mu_\epsilon]
        \rVert_\rho
        \le
        2 \epsilon.
    \end{equation*}
\end{lemma}

\begin{proof}
    Utilizing Proposition~\ref{prop:zero-quant-w-2}
    and the Cauchy--Schwarz inequality,
    we obtain
    \begin{align*}
        \lVert
            \widetilde\NRCDT_\bftheta [\mu_0]
            -
            \widetilde\NRCDT_\bftheta [\mu_\epsilon]
        \rVert_\rho
        &\le
        \lVert
            \widehat{\Radon}_\bftheta [\mu_0]
            -
            \widehat{\Radon}_\bftheta [\mu_\epsilon]
        \rVert_\rho
        +
        \lvert 
            \mean(\widehat{\Radon}_\bftheta [\mu_0])
            - 
            \mean(\widehat{\Radon}_\bftheta [\mu_\epsilon]) 
        \rvert
        \\
        &\le
        \lVert
            \widehat{\Radon}_\bftheta [\mu_0]
            -
            \widehat{\Radon}_\bftheta [\mu_\epsilon]
        \rVert_\rho
        +
        \int_\R 
        \lvert
            \widehat{\Radon}_\bftheta [\mu_0] (t)
            - 
            \widehat{\Radon}_\bftheta [\mu_\epsilon] (t)
        \rvert
        \: \d \rho(t) 
        \\
        &\le
        2 \,
        \lVert 
            \widehat{\Radon}_\bftheta [\mu_0]
            - 
            \widehat{\Radon}_\bftheta [\mu_\epsilon]
        \rVert_\rho
        \le
        2\epsilon.
        \tag*{\qed}
    \end{align*}
\end{proof}

We can now study the effect of perturbations in $W_2$ on the mean-normalized R-CDT.

\begin{proposition} \label{prop:mean-nrcdt-w-2}
    Let $\mu_0 \in \P_2^*(\R^2)$
    and $\mu_\epsilon \in \P_2(\R^2)$
    with $W_2(\mu_0, \mu_\epsilon) \le \epsilon$,
    and define 
    $c_0 \coloneqq \min_{\bftheta \in \Sphere} \std(\widehat{\Radon}_\bftheta \mu_0)$.
    If $\epsilon < c_0 / 2$,
    then
    \begin{equation*}
        \lVert
            \meanNRCDT [\mu_0]
            -
            \meanNRCDT [\mu_\epsilon]
        \rVert_\rho
        \le
        \frac{4\epsilon}{c_0}.
    \end{equation*}
\end{proposition}

\begin{proof}
    Equation~\eqref{eq:err-std} remains valid
    under the given assumptions.
    Since $2\epsilon < c_0$,
    the standard deviation 
    $\std(\widehat{\Radon}_\bftheta [\mu_\epsilon]) = \lVert \widetilde\NRCDT_\bftheta [\mu_\epsilon] \rVert_\rho$
    is uniformly bounded away from zero
    such that 
    $\mu_\epsilon \in \P_2^*(\R^2)$.
    In the view of Lemma~\ref{lem:zero-quant-w-2},
    we obtain
    \begin{align}
        \lVert
            \NRCDT_\bftheta [\mu_0] - \NRCDT_\bftheta [\mu_\epsilon]
        \rVert_\rho
        &=
        \biggl\lVert
            \frac
            {\widetilde\NRCDT_\bftheta [\mu_0]}
            {\lVert \widetilde\NRCDT_\bftheta [\mu_0] \rVert_\rho}
            -
            \frac
            {\widetilde\NRCDT_\bftheta [\mu_\epsilon]}
            {\lVert \widetilde\NRCDT_\bftheta [\mu_\epsilon] \rVert_\rho}
        \biggr\rVert_\rho
        \nonumber
        \\
        &\le
        \biggl\lVert
            \frac
            {\widetilde\NRCDT_\bftheta [\mu_0]}
            {\lVert \widetilde\NRCDT_\bftheta [\mu_0] \rVert_\rho}
            -
            \frac
            {\widetilde\NRCDT_\bftheta [\mu_\epsilon]}
            {\lVert \widetilde\NRCDT_\bftheta [\mu_0] \rVert_\rho}
        \biggr\rVert_\rho
        \!\! + \!
        \biggl\lVert
            \frac
            {\widetilde\NRCDT_\bftheta [\mu_\epsilon]}
            {\lVert \widetilde\NRCDT_\bftheta [\mu_0] \rVert_\rho}
            -
            \frac
            {\widetilde\NRCDT_\bftheta [\mu_\epsilon]}
            {\lVert \widetilde\NRCDT_\bftheta [\mu_\epsilon] \rVert_\rho}
        \biggr\rVert_\rho
        \nonumber
        \\
        &=
        \frac
        {\lVert \widetilde\NRCDT_\bftheta [\mu_0] 
            - 
            \widetilde\NRCDT_\bftheta [\mu_\epsilon] \rVert_\rho}
        {\lVert \widetilde\NRCDT_\bftheta [\mu_0] \rVert_\rho}
        \! + \!
        \frac
        {\bigl\lvert 
            \lVert \widetilde\NRCDT_\bftheta [\mu_\epsilon] \rVert_\rho
            -
            \lVert \widetilde\NRCDT_\bftheta [\mu_0] \rVert_\rho
        \bigr\rvert}
        {\lVert \widetilde\NRCDT_\bftheta [\mu_0] \rVert_\rho
            \lVert \widetilde\NRCDT_\bftheta [\mu_\epsilon] \rVert_\rho}
        \lVert \widetilde\NRCDT_\bftheta [\mu_\epsilon] \rVert_\rho
        \nonumber
        \\
        &\le
        2 \, 
        \frac
        {\lVert
            \widetilde\NRCDT_\bftheta [\mu_0] 
            - 
            \widetilde\NRCDT_\bftheta [\mu_\epsilon]
        \rVert_\rho}
        {\lVert \widetilde\NRCDT_\bftheta [\mu_0] \rVert_\rho}
        \le
        \frac{4\epsilon}{c_0}.
        \label{eq:norm-nrcdt-w-2}
    \end{align}
    Employing Jensen's inequality,
    we may bound the perturbation after normalization by
    \begin{align*}
        \lVert 
            \meanNRCDT [\mu_0] 
            - 
            \meanNRCDT [\mu_\epsilon] 
        \rVert_\rho^2
        =
        \int_\R 
        \biggl\lvert
        \int_{\Sphere}
        \NRCDT_\bftheta [\mu_0] (t)
        -
        \NRCDT_\bftheta [\mu_\epsilon] (t)
        \: \d u_{\Sphere}(\bftheta)
        \biggr\rvert^2
        \: \d \rho(t)
        \le 
        \Bigl( \frac{4\epsilon}{c_0} \Bigr)^2.
        \tag*{\qed}
    \end{align*}
\end{proof}

The squared bounds for the uncertainty of the mean-normalized R-CDT 
can be incorporated into the separation guarantee
in Theorem~\ref{thm:sep-mean-nrcdt}
to obtain linear separability even in the cases of slight perturbations in $W_2$.

\begin{theorem} 
    \label{thm:sep-pert-mean-nrcdt}
    For template measures $\mu_0, \nu_0 \in \P_2^*(\R^2)$
    with $\meanNRCDT \mu_0 \neq \meanNRCDT \nu_0$,
    define
    $c_\mu \coloneqq \min_{\bftheta\in\Sphere} \std(\widehat{\Radon}_\bftheta [\mu_0])$,
    $C_\mu \coloneqq \lVert\NRCDT [\mu_0]\rVert_{\rho \times u_{\Sphere}}$
    and
    $c_\nu \coloneqq \min_{\bftheta\in\Sphere} \std(\widehat{\Radon}_\bftheta [\nu_0])$,
    $C_\nu \coloneqq \lVert\NRCDT [\nu_0]\rVert_{\rho \times u_{\Sphere}}$.
    Let $c \in (0,\tfrac{1}{2})$,
    $c^\prime \in (c,\tfrac{1}{2})$
    and let $\epsilon < \min\{c_\mu, c_\nu\} / 2$ satisfy
    \begin{equation*}
        \epsilon 
        < 
        \frac{c^\prime - c}{4} \, 
        \min  \{ c_\mu, c_\nu \} \, 
        \frac
        { \lVert\meanNRCDT [\mu_0] - \meanNRCDT [\nu_0]\lVert_\rho \, \max\bigl\{C_\mu, C_\nu\bigr\}}
        {c \lVert\meanNRCDT [\mu_0] - \meanNRCDT [\nu_0]\lVert_\rho + \max\{C_\mu, C_\nu\}}.
    \end{equation*}
    Consider the classes
    \begin{align*}
        \F 
        &\coloneqq 
        \bigl\{
            (\bfA \cdot + \bfy)_\# \mu
            \mid 
            \bfA \in \X, 
            \, \bfy \in \R^2,
            \, \mu \in \P_2(\R^2),
            \, W_2(\mu, \mu_0) \le \epsilon
        \bigr\}, 
        \\
        \G 
        &\coloneqq 
        \bigl\{
            (\bfA \cdot + \bfy)_\# \nu
            \mid 
            \bfA \in \X, 
            \, \bfy \in \R^2,
            \, \nu \in \P_2(\R^2),
            \, W_2(\nu, \nu_0) \le \epsilon
        \bigr\}, 
    \end{align*}
    where
    \begin{equation*}
        \X 
        = 
        \biggl\{
            \bfA \in \GL(2) 
            \biggm| 
            \frac
            {\sigma_{\max}(\bfA) - \sigma_{\min}(\bfA)}
            {\sigma_{\min}(\bfA)}
            \leq 
            \frac
            {c}
            {\max\bigl\{
                C_\mu, 
                C_\nu
            \bigr\}}
            \,
            \lVert\meanNRCDT [\mu_0] - \meanNRCDT [\nu_0]\lVert_\rho
        \biggr\}
    \end{equation*}
    Then, any non-empty subsets $\F_0 \subseteq \F$ and $\G_0 \subseteq \G$ are linearly separable in \aNRCDT{} space.
\end{theorem}

\begin{proof}
    For $\mu \in \F$,
    there are $\bfA \in \X$, $\bfy \in \R^2$ and $\mu_\epsilon \in \P_2(\R^2)$
    with $W_2(\mu_\epsilon, \mu_0) \le \epsilon$
    such that $\mu = (\bfA \cdot + \bfy)_\# \mu_\epsilon$.
    Consequently,
    the proof of Theorem~\ref{thm:sep-mean-nrcdt} 
    in combination with Proposition~\ref{prop:mean-nrcdt-w-2} 
    gives
    \begin{align*}
        \lVert \meanNRCDT [\mu] - \meanNRCDT [\mu_0] \rVert_\rho 
        & \leq 
        \lVert \meanNRCDT [\mu] - \meanNRCDT [\mu_\epsilon] \rVert_\rho 
        +
        \lVert \meanNRCDT [\mu_\epsilon] - \meanNRCDT [\mu_0] \rVert_\rho
        \\
        & \leq
        \frac{\sigma_{\max}(\bfA) - \sigma_{\min}(\bfA)}{\sigma_{\min}(\bfA)} 
        \, \lVert \NRCDT [\mu_\epsilon] \lVert_{\rho \times u_{\Sphere}} 
        +
        \frac{4\epsilon}{c_\mu}.
    \end{align*}
    Utilizing \eqref{eq:norm-nrcdt-w-2},
    we obtain
    \begin{align*}
        \lVert \NRCDT [\mu_\epsilon] \lVert_{\rho \times u_{\Sphere}} 
        &\leq 
        \lVert \NRCDT [\mu_0] \lVert_{\rho \times u_{\Sphere}} 
        +
        \lVert \NRCDT [\mu_\epsilon] - \NRCDT [\mu_0] \lVert_{\rho \times u_{\Sphere}} 
        \\
        &\leq 
        \lVert \NRCDT [\mu_0] \lVert_{\rho \times u_{\Sphere}} 
        +
        \Bigl(
            \int_{\Sphere} \int_\R
            \lvert
            \NRCDT_\bftheta [\mu_\epsilon] (t) 
            - 
            \NRCDT_\bftheta [\mu_0] (t) 
            \lvert^2
            \: \d \rho(t) \, \d u_{\Sphere}(\bftheta)
        \Bigr)^{\frac{1}{2}}
        \\
        &\leq
        \lVert \NRCDT [\mu_0] \lVert_{\rho \times u_{\Sphere}} 
        +
        \frac{4\epsilon}{c_\mu}.
    \end{align*}
    Consequently, 
    the assumption $\bfA \in \X$ and \eqref{eq:nrcdt-sig-val} give
    \begin{equation*}
    \lVert \meanNRCDT [\mu] - \meanNRCDT [\mu_0] \rVert_\rho \leq c \lVert\meanNRCDT [\mu_0] - \meanNRCDT [\nu_0]\lVert_\rho + \biggl(\frac{c \lVert\meanNRCDT [\mu_0] - \meanNRCDT [\nu_0]\lVert_\rho}{\max\{C_\mu, C_\nu\}} + 1\biggr) \frac{4\epsilon}{c_\mu}
    \end{equation*}
    and the choice of $\epsilon > 0$ guarantees that
    \begin{equation*}
    \lVert \meanNRCDT [\mu] - \meanNRCDT [\mu_0] \rVert_\rho < c^\prime \, \lVert\meanNRCDT [\mu_0] - \meanNRCDT [\nu_0] \lVert_\rho
    \eqqcolon r_0.
    \end{equation*}
    Consequently,
    $\meanNRCDT \F \subseteq B_{r_0}(\meanNRCDT [\mu_0]) \subset \Lebesgue^2_\rho(\R)$
    and, analogously,
    $\meanNRCDT \G \subseteq B_{r_0}(\meanNRCDT [\nu_0]) \subset \Lebesgue^2_\rho(\R)$.
    Since $c^\prime \in (c,\frac{1}{2})$, $B_{r_0}(\meanNRCDT [\mu_0])$ and $B_{r_0}(\meanNRCDT [\nu_0])$ are linearly separable in $\Lebesgue^2_\rho(\R)$.
    This implies the linear separability of $\meanNRCDT \F_0$ and $\meanNRCDT \G_0$ in $\Lebesgue^2_\rho(\R)$ for any non-empty $\F_0 \subseteq \F$ and $\G_0 \subseteq \G$.
\end{proof}

\section{Numerical experiments}
\label{sec:numerics}

In order to support our theoretical findings with numerical evidence,
we provide a series of proof-of-concept simulations,
which focus on the (linear) separability 
under affine transformations  
established in Theorem~\ref{thm:sep-max-nrcdt} 
and \ref{thm:sep-mean-nrcdt}
and on the influence of non-affine perturbations
studied in Theorem~\ref{thm:sep-pert-max-nrcdt}
and~\ref{thm:sep-pert-mean-nrcdt}.
Since the original R-CDT \cite{Kolouri2016} already outperforms 
other state-of-the-art classifiers
in the small data regime \cite{Shifat-E-Rabbi2021},
we restrict the experiments to comparisons
with the R-CDT and the na\"\i{}ve Euclidean approach
and omit comparisons with neural network classifiers.
The new \mNRCDT\ and \aNRCDT\ methods significantly increase
the classification accuracy
showing their potential as feature extractor.
The methods are implemented in Julia%
\footnote{The Julia Programming Language -- Version 1.9.2 
(\url{https://docs.julialang.org}).},
and the code is publicly available at
\url{https://github.com/DrBeckmann/NR-CDT}.
All experiments are performed on an off-the-shelve MacBookPro 2020 
with Intel Core i5 (4‑Core CPU, 1.4~GHz) and 8~GB~RAM.

\subsection{Academic Datasets}

For the majority of the numerical simulations,
we rely on academic datasets
that allow us to fully control 
the occurring affine and non-affine perturbations.
Starting from the synthetic symbols in Figure~\ref{fig:academic_dataset},
whose domains are contained in the unit disc, 
we generate datasets with up to twelve classes
by (anisotropic) scaling, rotating, shearing, and shifting
the shown templates
on the pixel grid using bi-quadratic interpolations.
Scaling and shearing is here independently applied twice 
with respect to the horizontal and vertical direction.
Depending on the experiment,
we further apply non-affine transformations
or add impulsive noise.
Finally,
the gray values are normalized
to represent a (absolutely continuous) probability measure,
where each pixel corresponds to a constant density on a square.

\begin{figure}[t]
    \resizebox{\linewidth}{!}{%
    \tiny
    \begin{tabular}{@{\,}*{12}{c@{\,}}}
        symbol 1
        & symbol 2
        & symbol 3
        & symbol 4
        & symbol 5
        & symbol 6
        & symbol 7
        & symbol 8
        & symbol 9
        & symbol 10
        & symbol 11
        & symbol 12
        \\
        \includegraphics[width=0.083\linewidth, clip=true, trim=140pt 30pt 120pt 20pt]{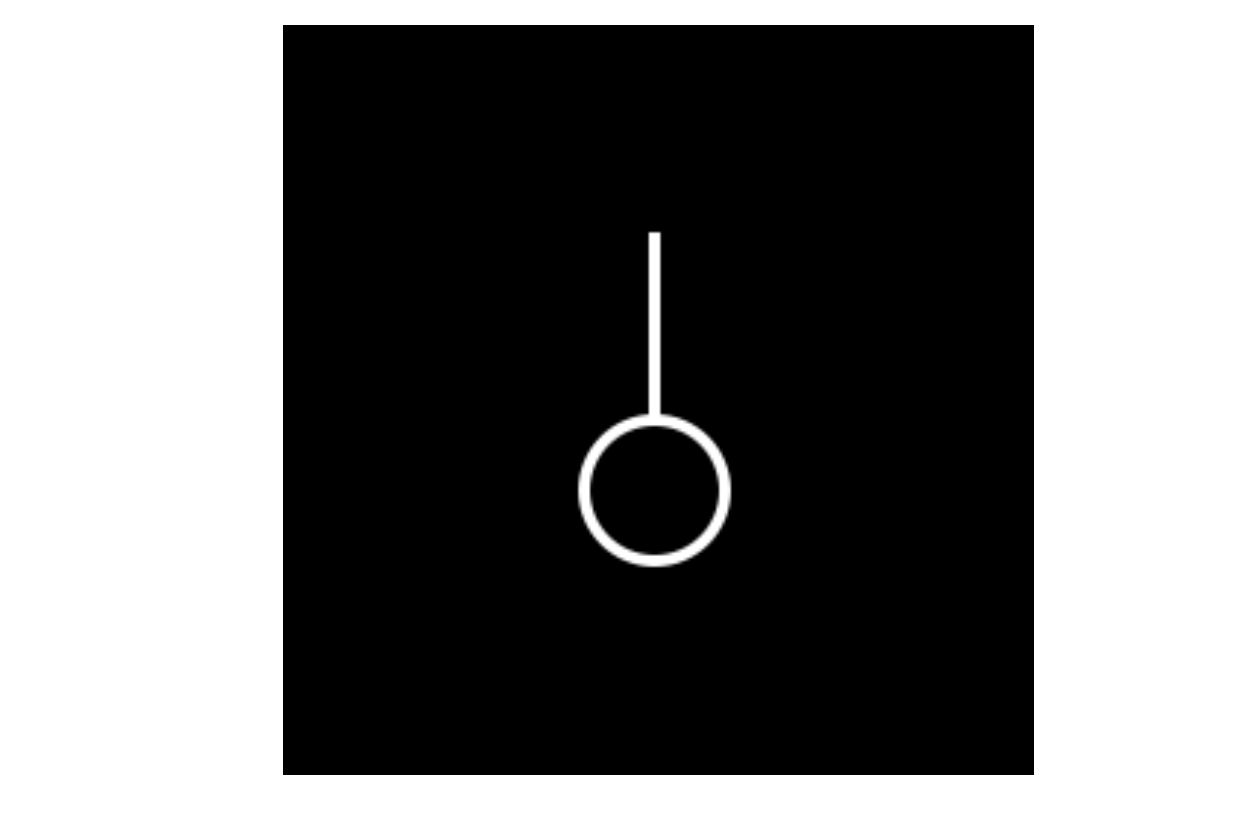}
        &\includegraphics[width=0.083\linewidth, clip=true, trim=140pt 30pt 120pt 20pt]{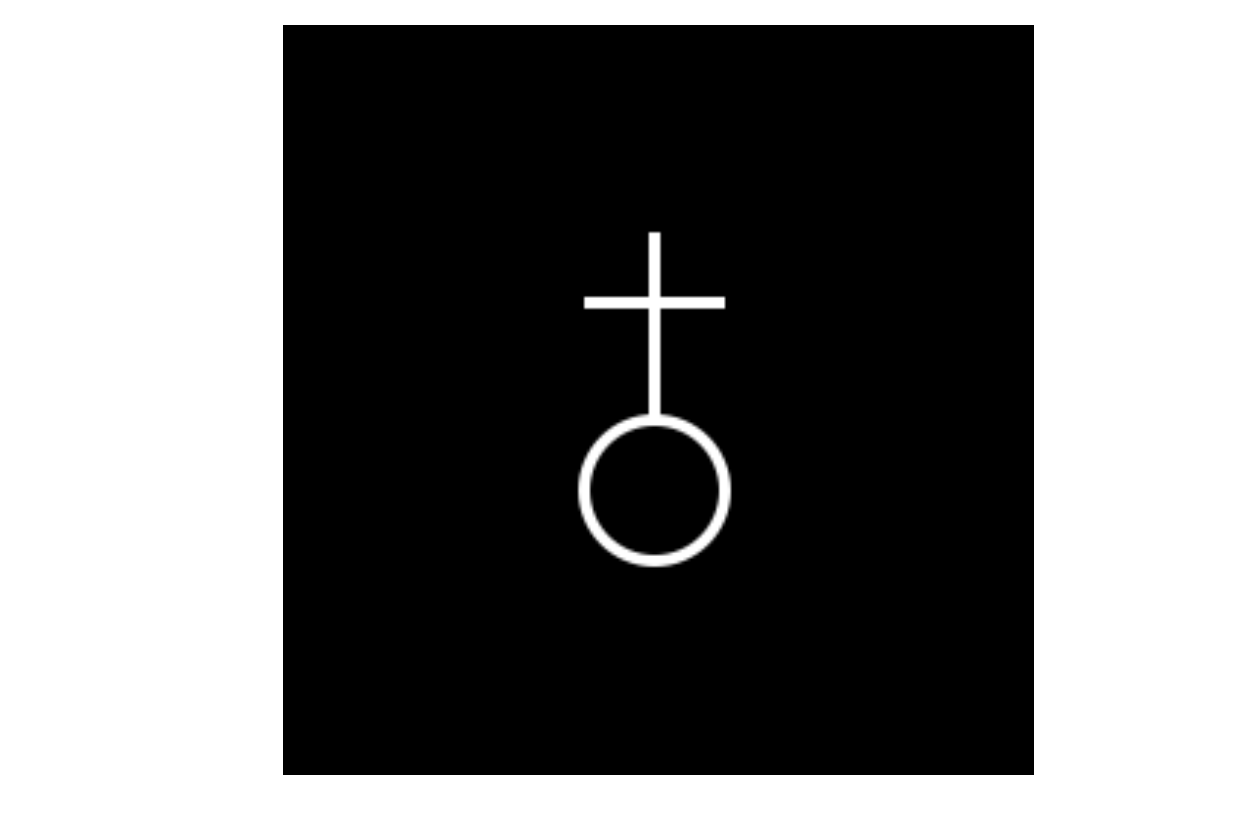}
        &\includegraphics[width=0.083\linewidth, clip=true, trim=140pt 30pt 120pt 20pt]{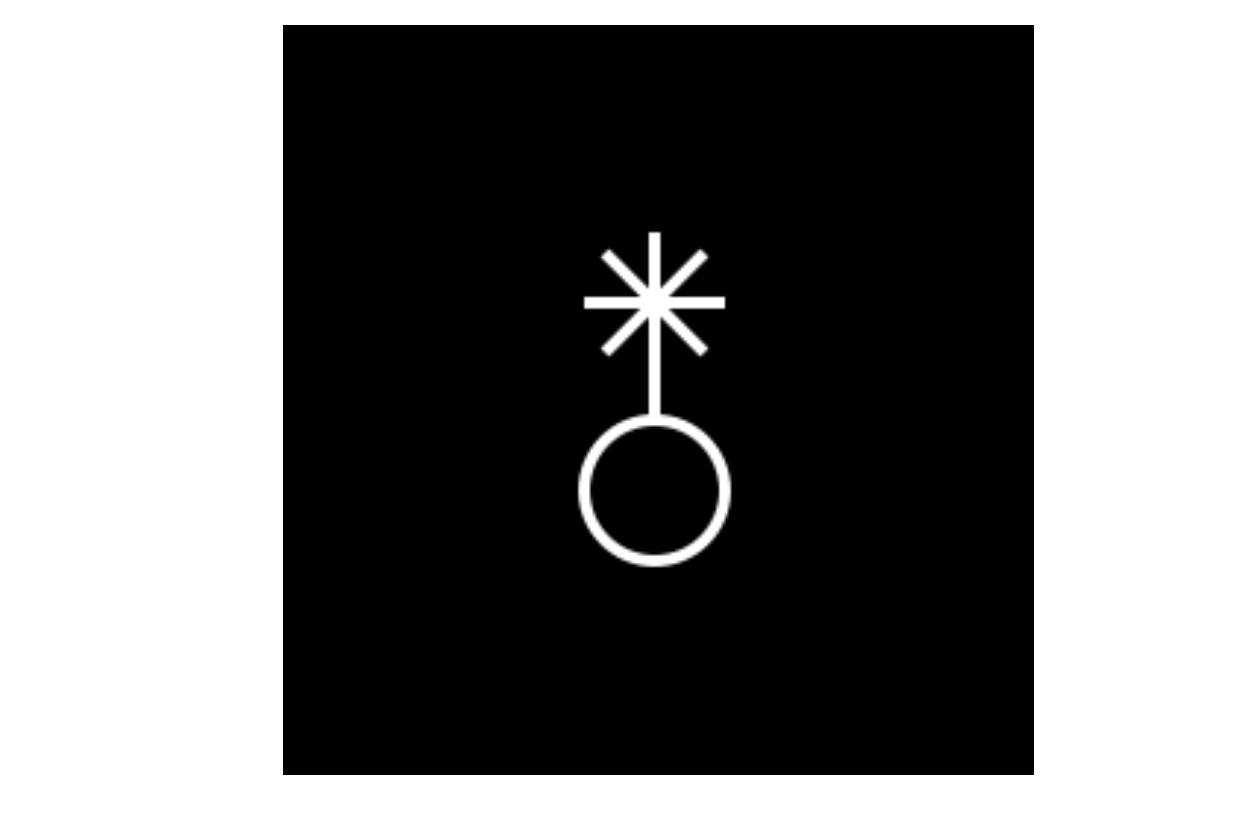}
        &\includegraphics[width=0.083\linewidth, clip=true, trim=140pt 30pt 120pt 20pt]{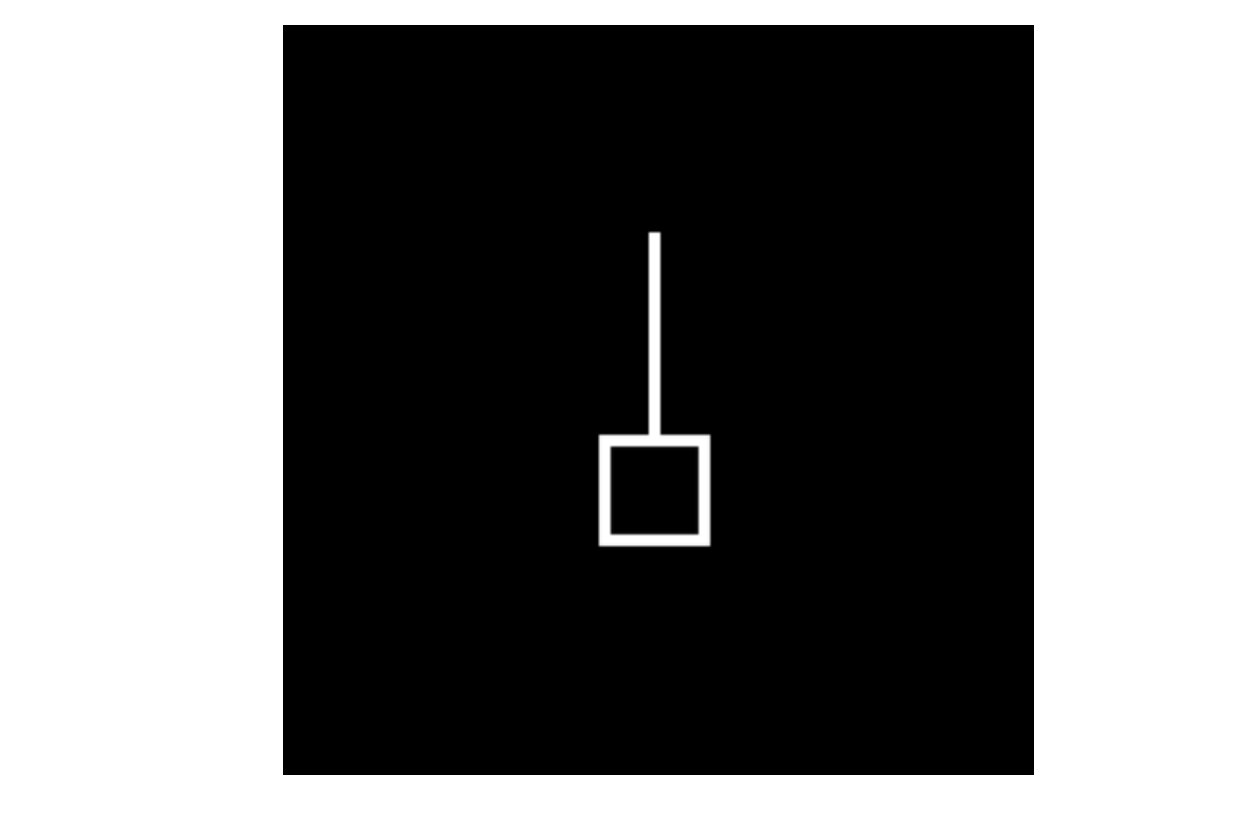}
        &\includegraphics[width=0.083\linewidth, clip=true, trim=140pt 30pt 120pt 20pt]{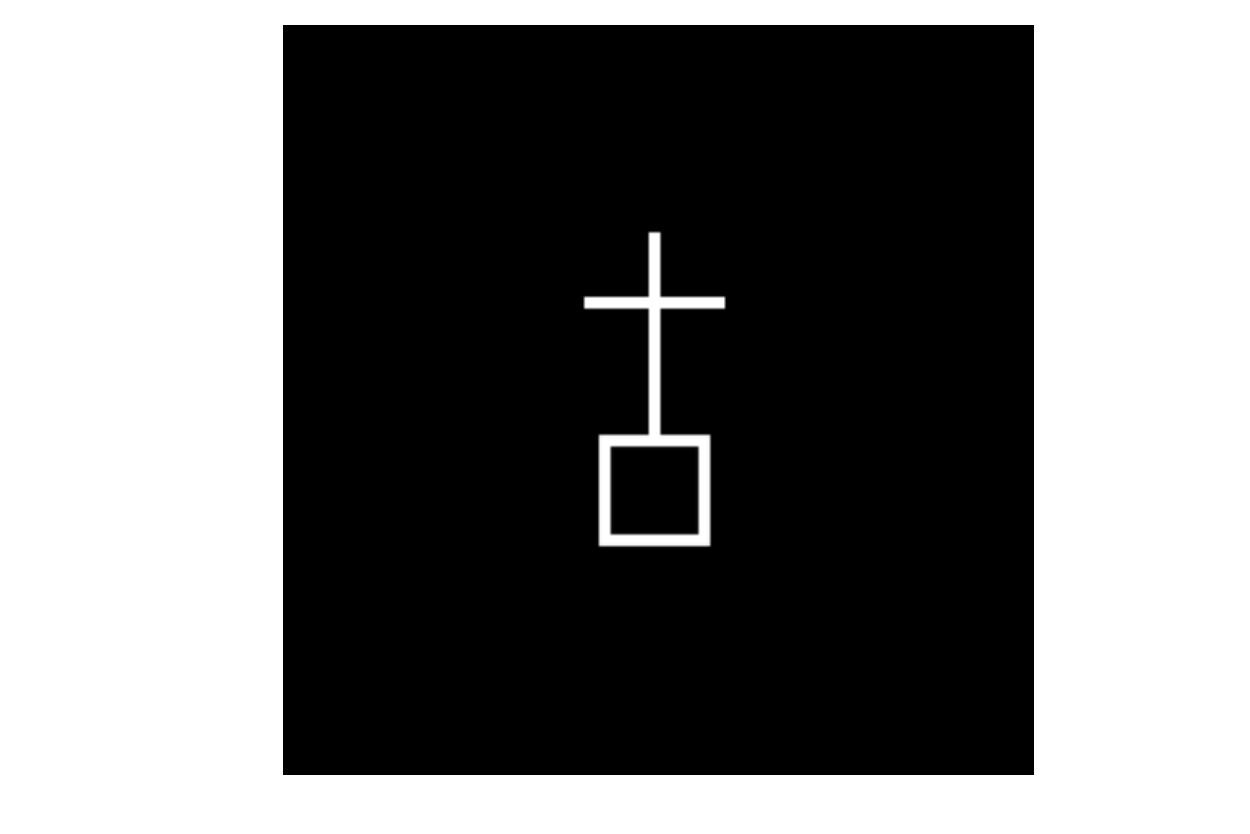}
        &\includegraphics[width=0.083\linewidth, clip=true, trim=140pt 30pt 120pt 20pt]{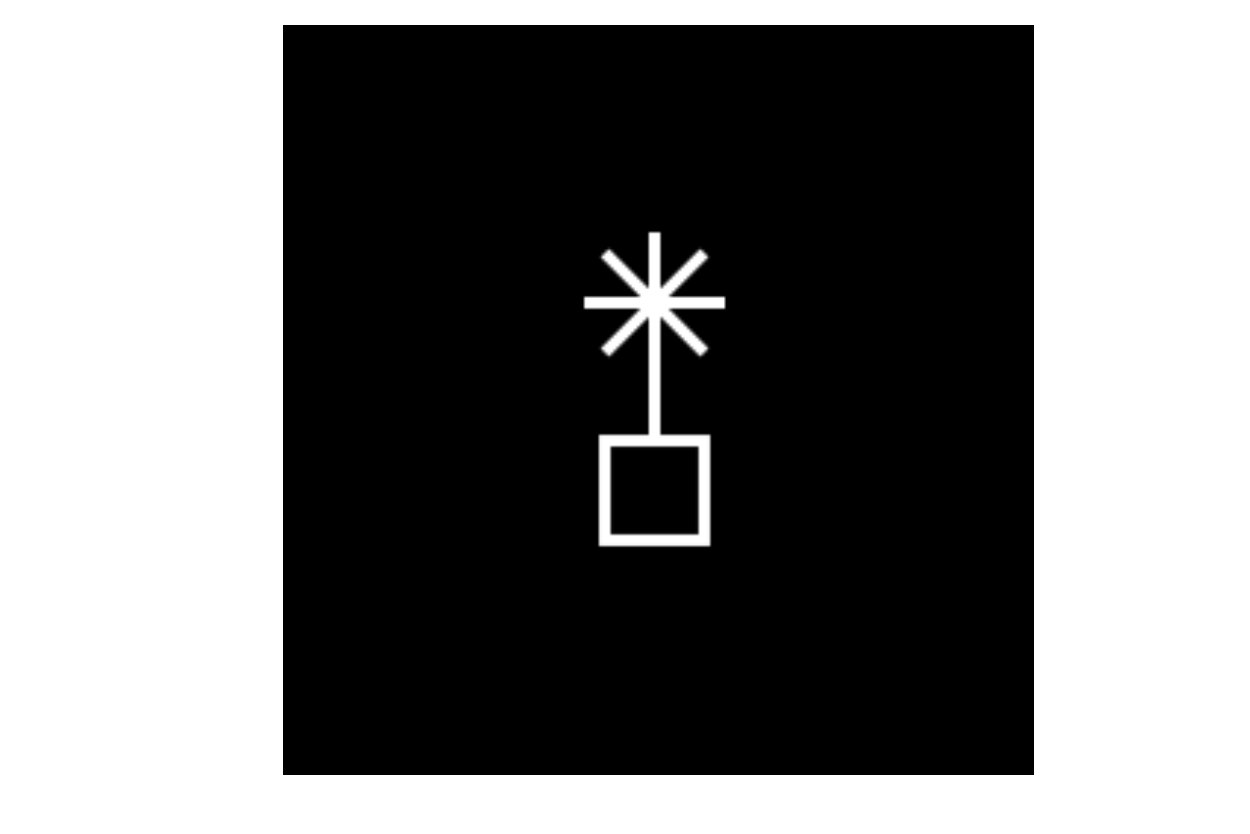} 
        &\includegraphics[width=0.083\linewidth, clip=true, trim=140pt 30pt 120pt 20pt]{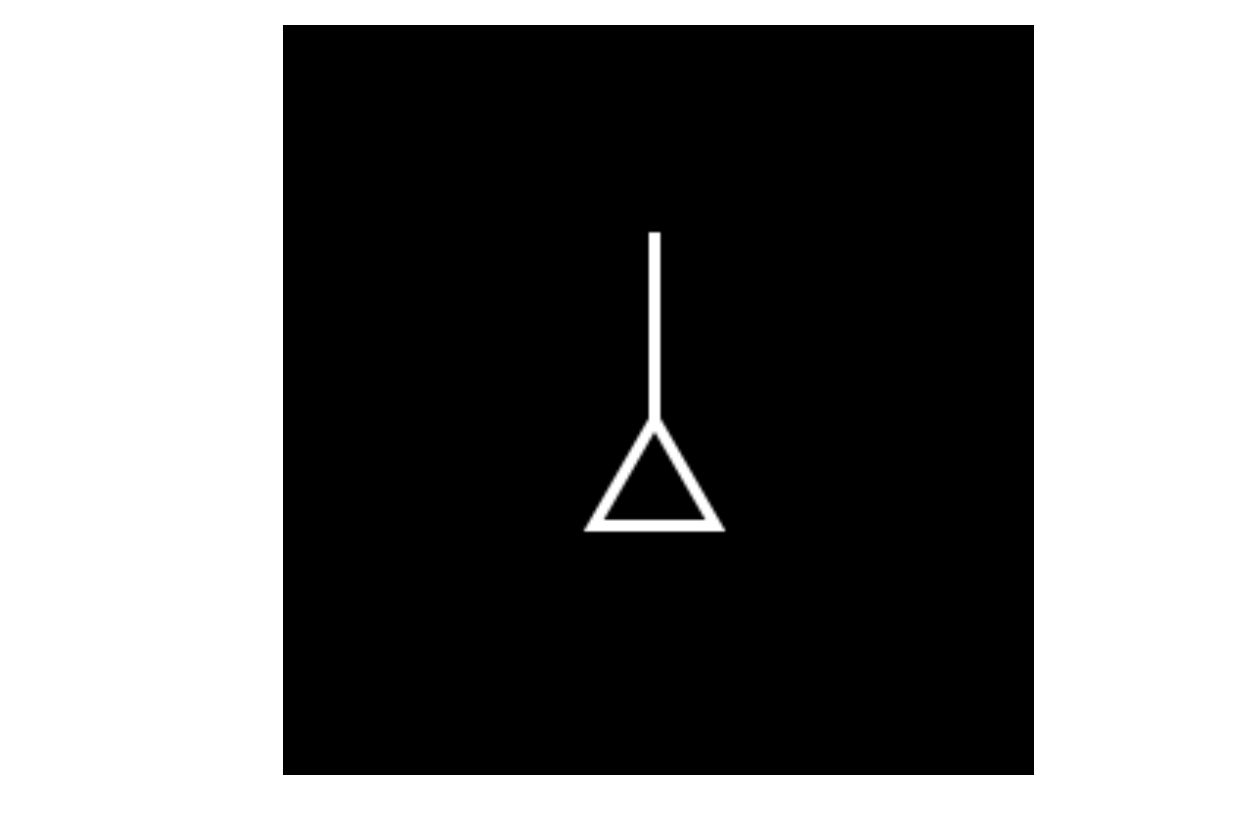}
        &\includegraphics[width=0.083\linewidth, clip=true, trim=140pt 30pt 120pt 20pt]{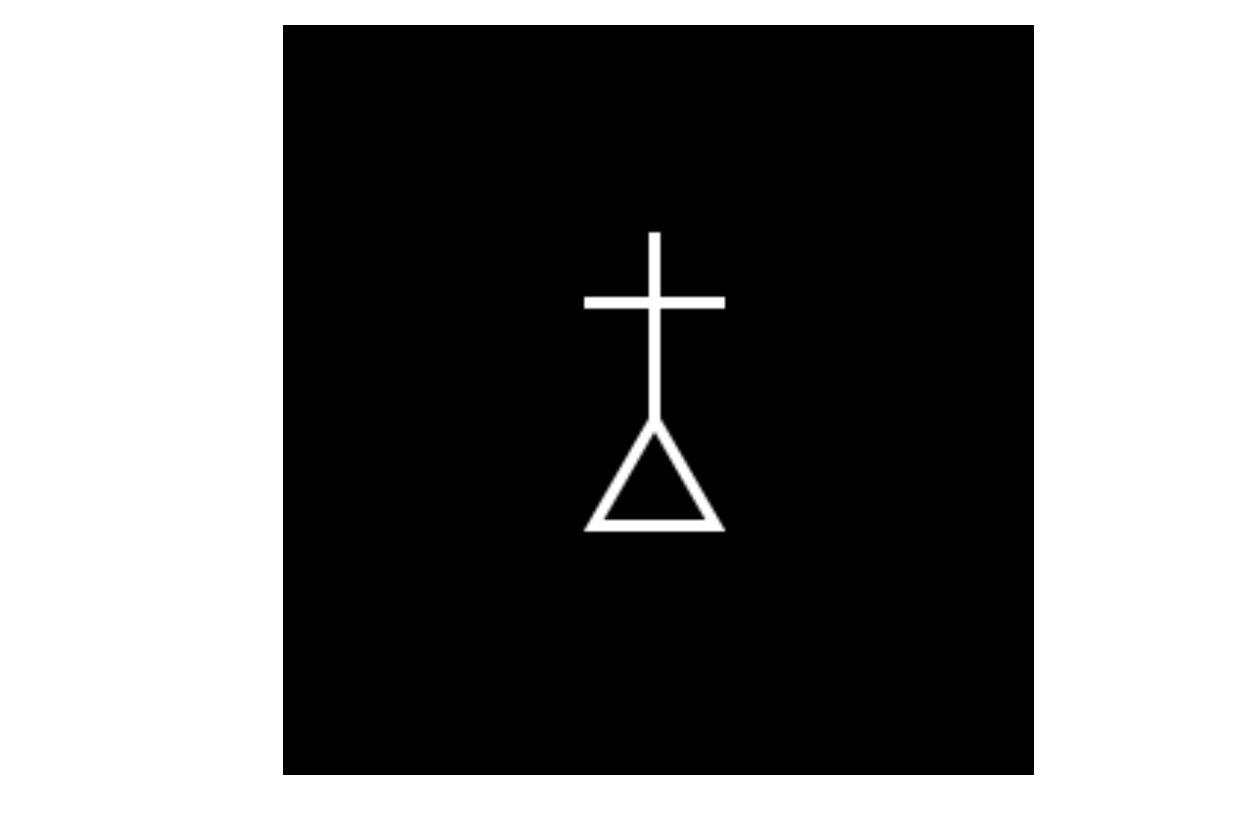}
        &\includegraphics[width=0.083\linewidth, clip=true, trim=140pt 30pt 120pt 20pt]{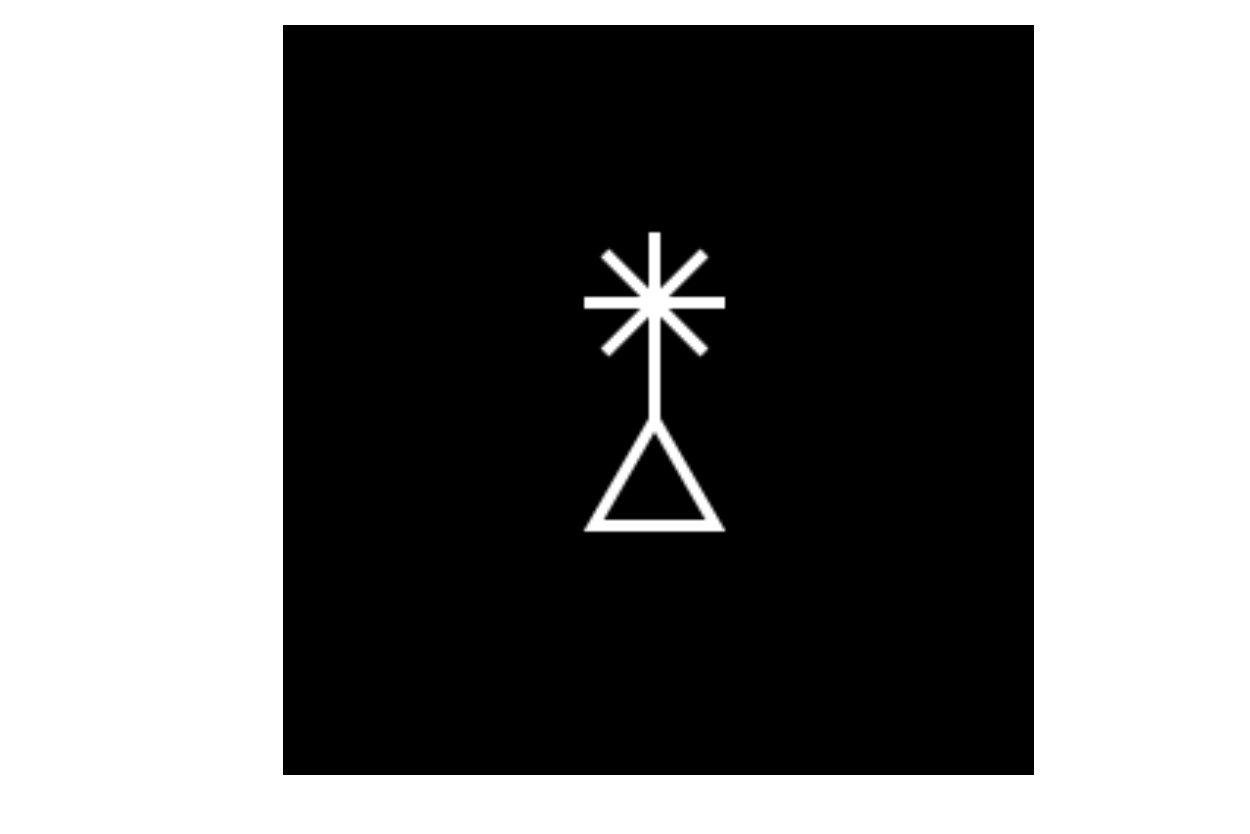}
        &\includegraphics[width=0.083\linewidth, clip=true, trim=140pt 30pt 120pt 20pt]{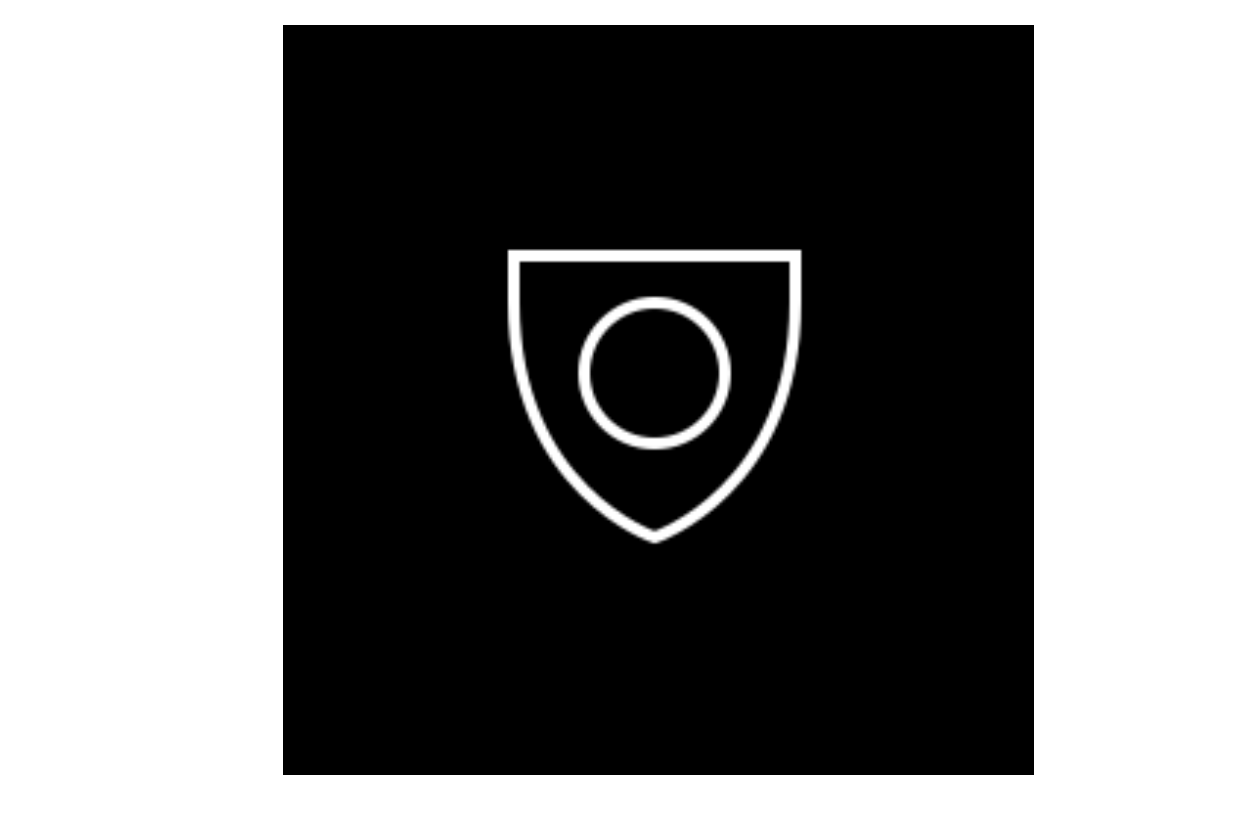}
        &\includegraphics[width=0.083\linewidth, clip=true, trim=140pt 30pt 120pt 20pt]{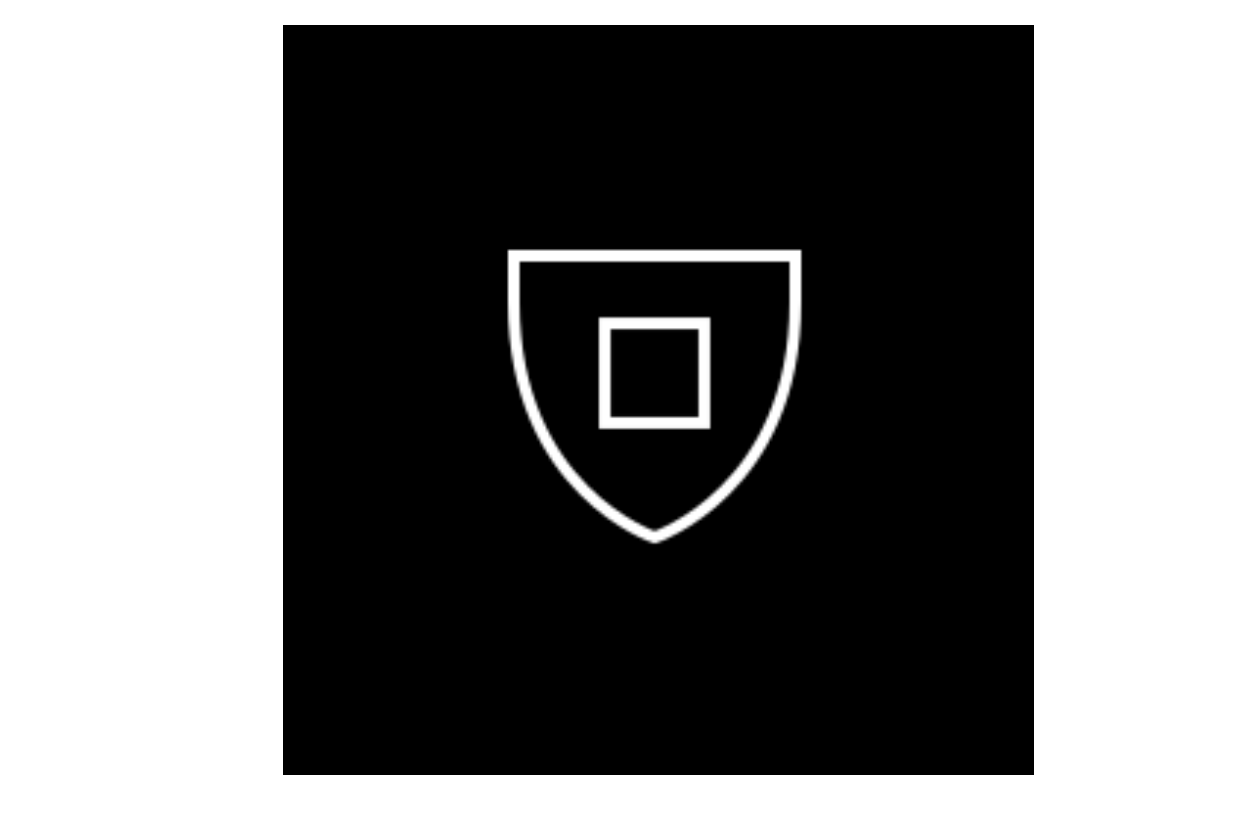}
        &\includegraphics[width=0.083\linewidth, clip=true, trim=140pt 30pt 120pt 20pt]{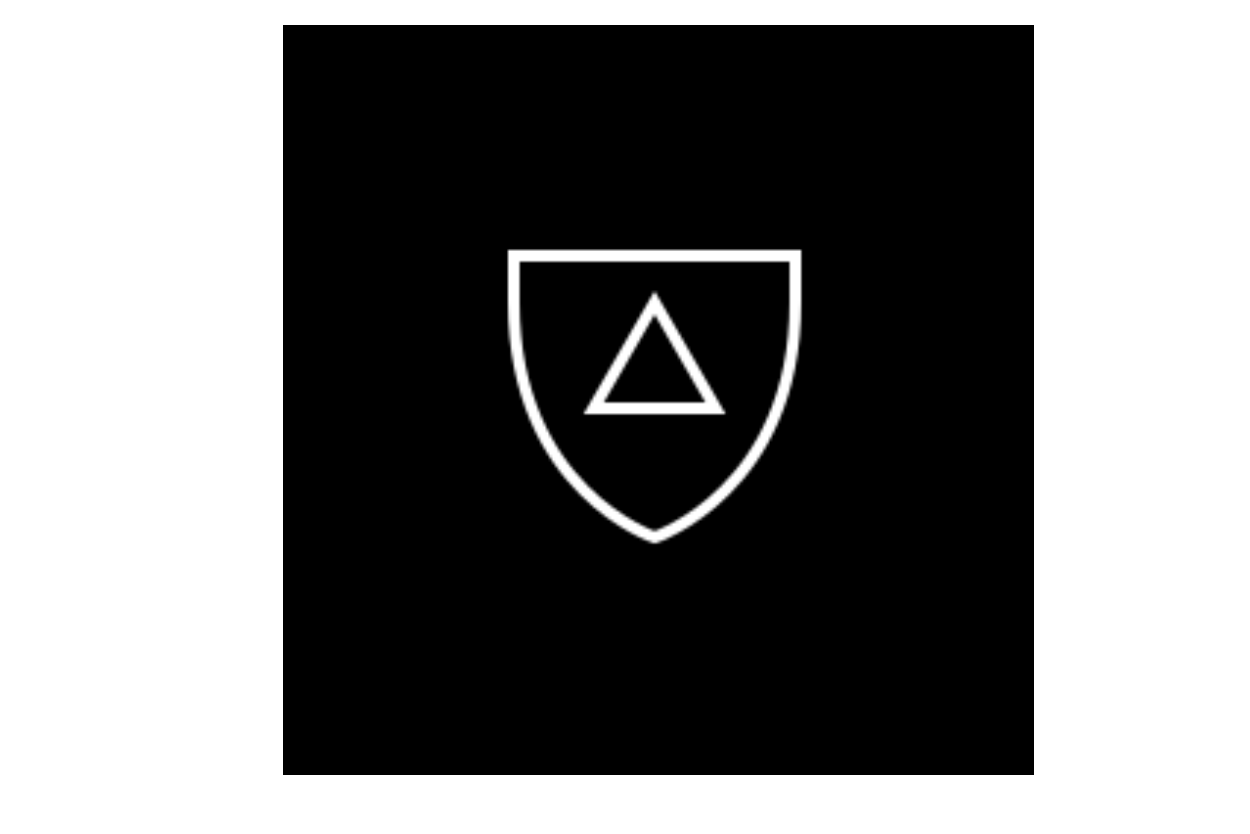}
    \end{tabular}}%
    \caption{
        Templates for academic datasets
        with domain $[-\frac{1}{\sqrt 2}, \frac{1}{\sqrt 2}]^2$
        represented by $256{\times}256$ pixels.
        }%
    \label{fig:academic_dataset}%
\end{figure}

\subsubsection{Classification under Affine Transformations}
\label{sec:class-aff-trans}

The first numerical example deals with the ideal, unperturbed setting,
where we aim to classify 10 affinely transformed versions
of all twelve symbols. 
The theory behind \mNRCDT\ in Theorem~\ref{thm:sep-max-nrcdt} predicts
that every class is transformed to a single point in \mNRCDT{} space.
In order to observe this behaviour numerically,
the underlying Radon transform and CDT
have to be discretized fine enough.
In this and all experiment regarding the academic datasets,
we choose 850 equispaced radii in $[-1,1]$
and 128 equispaced angles in $[0, 2\pi)$
for the Radon transform
and 64 equispaced interpolation points in $(0,1)$
for the CDT.
For illustration,
Figure~\ref{fig:academic_NRCDT_affine} shows
the \mNRCDT\ for the affine classes
with respect to templates~5 and~12.
The remaining numerical errors originate from
bi-quadratic interpolations
underlying the affine image transformations.
Figure~\ref{fig:academic_NRCDT_affine} also shows the \aNRCDT\ of both classes.
Note that the \aNRCDT\ does not transform an affine class to a single point
but to a small ball
around the template
whose radius depends on the eigenvalues of the affine transformations,
see Theorem~\ref{thm:sep-mean-nrcdt}.
The quality of both transformations is comparable
but visually the \mNRCDT\ yields a larger distance 
between the classes.

\begin{SCfigure}[3][t]
    \includegraphics[width=0.29\linewidth, clip=true, trim=10pt 10pt 20pt 20pt]{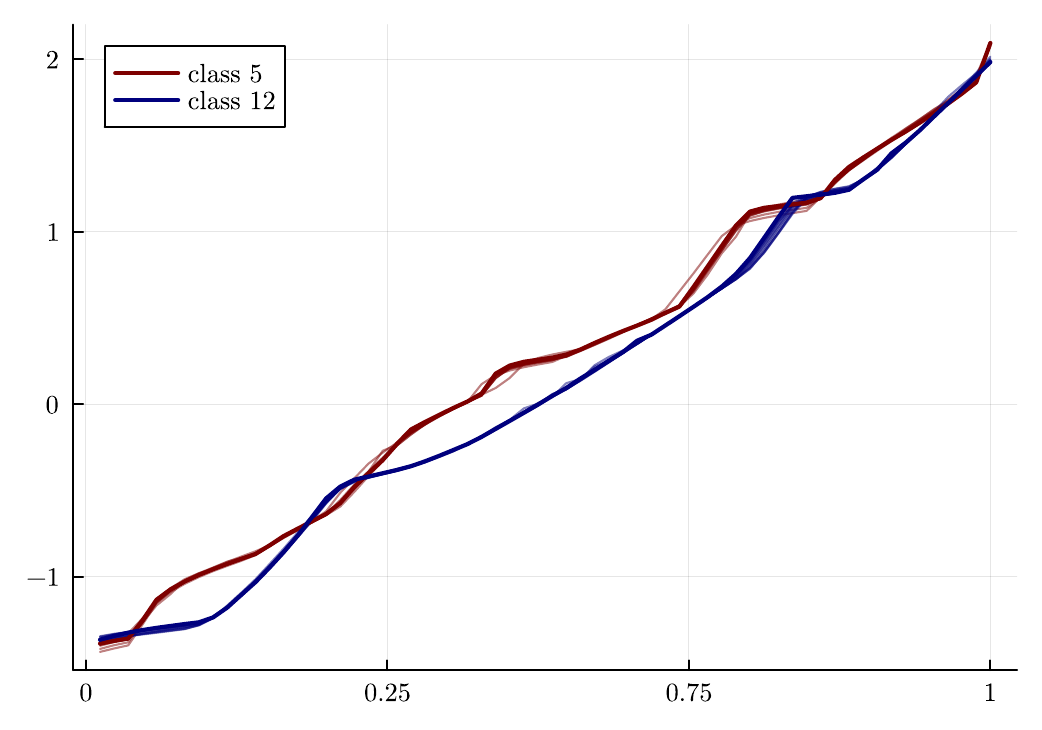}
    \includegraphics[width=0.29\linewidth, clip=true, trim=10pt 10pt 20pt 20pt]{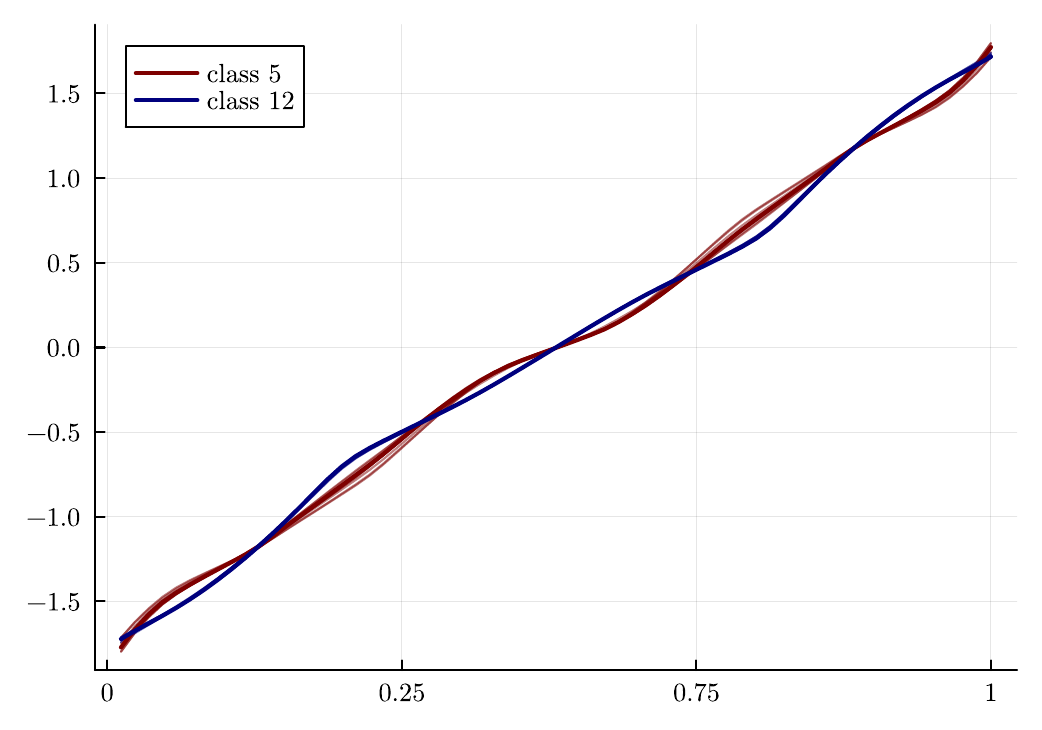}
    \caption{
    \mNRCDT{} (left) and \aNRCDT{} (right)
    of affine classes 5 and 12.
    The classes are generated using
    horizontal/vertical scaling by a factor in $[0.5,1.25]$,
    rotation by an angle in $[0^\circ, 360^\circ]$,
    shearing of horizontal/vertical axis by an angle in $[-45^\circ,45^\circ]$,
    and horizontal/vertical shifting by a pixel number in $[-20,20]$.
    }
    \label{fig:academic_NRCDT_affine}
\end{SCfigure}

\begin{table}[t]
    \caption{NT classification accuracies for
        academic datasets with 10 samples per class
        and different parameter ranges for random affine transformations.
        Similar to Figure~\ref{fig:academic_NRCDT_affine},
        rotation angles are in $[0^\circ,360^\circ]$,
        pixel shifts in $[-20,20]$.
        The best result per dataset and angle number is highlighted.}
    \resizebox{\linewidth}{!}{%
    \begin{tabular}{*{25}{l}}
    \toprule
    angles  
    & \multicolumn{6}{l}{scaling in $[0.5, 1.25]$, shearing in $[-45^\circ,45^\circ]$} 
    & \multicolumn{6}{l}{scaling in $[0.75, 1.25]$, shearing in $[-35^\circ,35^\circ]$} 
    & \multicolumn{6}{l}{scaling in $[0.75, 1.0]$, shearing in $[-15^\circ,15^\circ]$}
    & \multicolumn{6}{l}{no scaling and shearing}
    \\
    \cmidrule(lr){2-7}
    \cmidrule(lr){8-13}
    \cmidrule(lr){14-19}
    \cmidrule(lr){20-25}
    & \multicolumn{2}{l}{R-CDT} & \multicolumn{2}{l}{\mNRCDT} & \multicolumn{2}{l}{\aNRCDT} 
    & \multicolumn{2}{l}{R-CDT} & \multicolumn{2}{l}{\mNRCDT} & \multicolumn{2}{l}{\aNRCDT}
    & \multicolumn{2}{l}{R-CDT} & \multicolumn{2}{l}{\mNRCDT} & \multicolumn{2}{l}{\aNRCDT} 
    & \multicolumn{2}{l}{R-CDT} & \multicolumn{2}{l}{\mNRCDT} & \multicolumn{2}{l}{\aNRCDT}\\
    & $\|\cdot\|_\infty$ & $\|\cdot\|_2$ & $\|\cdot\|_\infty$ & $\|\cdot\|_2$ 
    & $\|\cdot\|_\infty$ & $\|\cdot\|_2$ & $\|\cdot\|_\infty$ & $\|\cdot\|_2$
    & $\|\cdot\|_\infty$ & $\|\cdot\|_2$ & $\|\cdot\|_\infty$ & $\|\cdot\|_2$
    & $\|\cdot\|_\infty$ & $\|\cdot\|_2$ & $\|\cdot\|_\infty$ & $\|\cdot\|_2$ 
    & $\|\cdot\|_\infty$ & $\|\cdot\|_2$ & $\|\cdot\|_\infty$ & $\|\cdot\|_2$
    & $\|\cdot\|_\infty$ & $\|\cdot\|_2$ & $\|\cdot\|_\infty$ & $\|\cdot\|_2$
    \\
    \midrule
    1 
    & $0.1083$ & $0.2333$ & $0.1750$ & \bm{$0.3416$} & $0.1750$ & \bm{$0.3416$}
    & $0.1166$ & $0.2500$ & $0.1500$ & \bm{$0.3083$} & $0.1500$ & \bm{$0.3083$} 
    & $0.1666$ & $0.2666$ & $0.1916$ & \bm{$0.3083$} & $0.1916$ & \bm{$0.3083$} 
    & $0.2000$ & $0.2583$ & $0.1583$ & \bm{$0.3166$} & $0.1583$ & \bm{$0.3166$}
    \\
    2  
    & $0.1083$ & $0.2333$ & $0.1750$ & \bm{$0.3416$} & $0.3000$ & \bm{$0.3416$} 
    & $0.1166$ & $0.2500$ & $0.1666$ & \bm{$0.3083$} & $0.2333$ & $0.2916$ 
    & $0.1666$ & $0.2666$ & $0.1916$ & \bm{$0.3083$} & $0.2166$ & $0.3000$ 
    & $0.2000$ & $0.2583$ & $0.1500$ & \bm{$0.3333$} & $0.2500$ & $0.3083$
    \\
    4 
    & $0.1333$ & $0.2500$ & $0.2166$ & \bm{$0.5333$} & $0.2416$ & $0.3583$ 
    & $0.1583$ & $0.2250$ & $0.2833$ & \bm{$0.5416$} & $0.2500$ & $0.3833$ 
    & $0.2250$ & $0.2083$ & $0.3500$ & \bm{$0.5666$} & $0.1916$ & $0.3833$ 
    & $0.2166$ & $0.2083$ & $0.3250$ & \bm{$0.5333$} & $0.1833$ & $0.3583$
    \\
    8
    & $0.0916$ & $0.2000$ & $0.4416$ & \bm{$0.6333$} & $0.4333$ & $0.4750$
    & $0.1333$ & $0.2083$ & $0.3916$ & \bm{$0.6083$} & $0.4250$ & $0.4583$ 
    & $0.1750$ & $0.2250$ & $0.4166$ & \bm{$0.6250$} & $0.4333$ & $0.5000$
    & $0.2166$ & $0.2250$ & $0.4583$ & \bm{$0.6500$} & $0.4833$ & $0.4750$
    \\
    16
    & $0.1416$ & $0.1916$ & $0.6333$ & \bm{$0.8250$} & $0.7250$ & $0.7750$
    & $0.1583$ & $0.2250$ & $0.6916$ & $0.8583$ & $0.8333$ & \bm{$0.8750$} 
    & $0.1666$ & $0.2250$ & $0.7916$ & \bm{$0.9083$} & $0.8583$ & \bm{$0.9083$}
    & $0.2500$ & $0.2250$ & $0.7833$ & \bm{$0.9166$} & $0.9250$ & $0.9000$
    \\
    32
    & $0.1500$ & $0.1916$ & $0.8833$ & \bm{$0.9916$} & $0.8416$ & $0.8750$
    & $0.2000$ & $0.2333$ & $0.9083$ & \bm{$1.0000$} & $0.9583$ & $0.9416$ 
    & $0.1750$ & $0.2083$ & $0.9500$ & \bm{$1.0000$} & \bm{$1.0000$} & \bm{$1.0000$}
    & $0.2500$ & $0.1266$ & $0.9500$ & $\bf{1.0000}$ & $\bf{1.0000}$ & $\bf{1.0000}$
    \\
    64 
    & $0.1666$ & $0.1916$ & $0.9500$ & $\bf{1.0000}$ & $0.8333$ & $0.8583$
    & $0.2000$ & $0.2250$ & $0.9750$ & \bm{$1.0000$} & $0.9583$ & $0.9416$ 
    & $0.1666$ & $0.2083$ & $0.9750$ & \bm{$1.0000$} & \bm{$1.0000$} & \bm{$1.0000$}
    & $0.2416$ & $0.2083$ & $0.9500$ & $\bf{1.0000}$ & $\bf{1.0000}$ & $\bf{1.0000}$
    \\
    128 
    & $0.1666$ & $0.1916$ & $\bf{1.0000}$ & $\bf{1.0000}$ & $0.8500$ & $0.8583$
    & $0.2000$ & $0.2250$ & \bm{$1.0000$} & \bm{$1.0000$} & $0.9583$ & $0.9416$ 
    & $0.1666$ & $0.2083$ & \bm{$1.0000$} & \bm{$1.0000$} & \bm{$1.0000$} & $0.9916$
    & $0.2250$ & $0.2083$ & $\bf{1.0000}$ & $\bf{1.0000}$ & $\bf{1.0000}$ & $\bf{1.0000}$
    \\
    256 
    & $0.1666$ & $0.1916$ & $0.9666$ & $\bf{1.0000}$ & $0.8500$ & $0.8666$
    & $0.1916$ & $0.2166$ & \bm{$1.0000$} & \bm{$1.0000$} & $0.9583$ & $0.9416$ 
    & $0.1833$ & $0.2083$ & $0.9750$ & \bm{$1.0000$} & \bm{$1.0000$} & $0.9833$
    & $0.2166$ & $0.2083$ & $0.9583$ & $0.9916$ & $\bf{1.0000}$ & $\bf{1.0000}$
    \\
    \midrule
    Eucl.
    & $0.0833$ & $0.0916$ & & & &
    & $0.0833$ & $0.0833$ & & & &
    & $0.0833$ & $0.0666$ & & & &
    & $0.0833$ & $0.0666$ & & & &
    \\
    \bottomrule
    \end{tabular}}%
    \label{tab:academic_NT_affine}
\end{table}

To classify a given datum,
we assign the label of the closest template
in the considered feature spaces.
Henceforth,
we refer to this approach as
\emph{nearest template} (NT) \emph{classification}.
Since the quality of the \aNRCDT\ mainly depends
of the size of the anisotropic scaling and shearing,
we repeat the experiment for different parameter ranges.
Our classification results are reported
in Table~\ref{tab:academic_NT_affine},
where we compare the NT performance of the \mNRCDT{} and \aNRCDT{}
with the R-CDT representation from~\cite{Kolouri2016}
and the Euclidean representation as baseline.
In feature space,
we use the Euclidean $\|\cdot\|_2$ 
and Chebyshev $\|\cdot\|_\infty$ norm
to assign the labels.
Moreover, we vary the number of
equispaced angles for the underlying Radon transform.
The \mNRCDT\ and \aNRCDT\ feature representations clearly outperform
the R-CDT and the Euclidean baseline
for classification under affine transformations.
The results of the \mNRCDT\ surpass the accuracies of the \aNRCDT\ 
especially in the presence of large anisotropic scaling and shearing,
which is covered by the developed theory.
Notice that
already small numbers of Radon angles yield high accuracies.
Finally, the Euclidean norm outperforms the Chebyshev norm;
for this reason,
we restrict ourselves henceforth to the Euclidean norm.

\subsubsection{Classification under Non-affine Deformations}

The theory behind Theorem~\ref{thm:sep-pert-max-nrcdt}
and~\ref{thm:sep-pert-mean-nrcdt} guarantees 
the separability of affine classes
in \mNRCDT{} and \aNRCDT{} space
even for imperfect affine transformations.
In the next experiments,
we study the robustness of the proposed methods
against non-affine deformations 
and additive impulsive noise.
Both error sources are illustrated in Figure~\ref{fig:academic_distortions}
together with optimal transport plans 
from the underlying true template.
Already light impulsive noise, 
which is referred to as salt noise,
has a similar effect on the Wasserstein-2 distance
as strong non-affine perturbations.
Therefore,
we expect that
\mNRCDT\ and \aNRCDT\ can manage non-affine distortions better than impulsive noise.

\begin{figure}
  \resizebox{\linewidth}{!}{%
    \tiny%
    \begin{tabular}{c c c}
        non-affine deformation & 
        salt noise & non-affine deformation \& salt noise \\
        \includegraphics[height=0.16\linewidth, clip=true, trim=130pt 10pt 100pt 0pt]{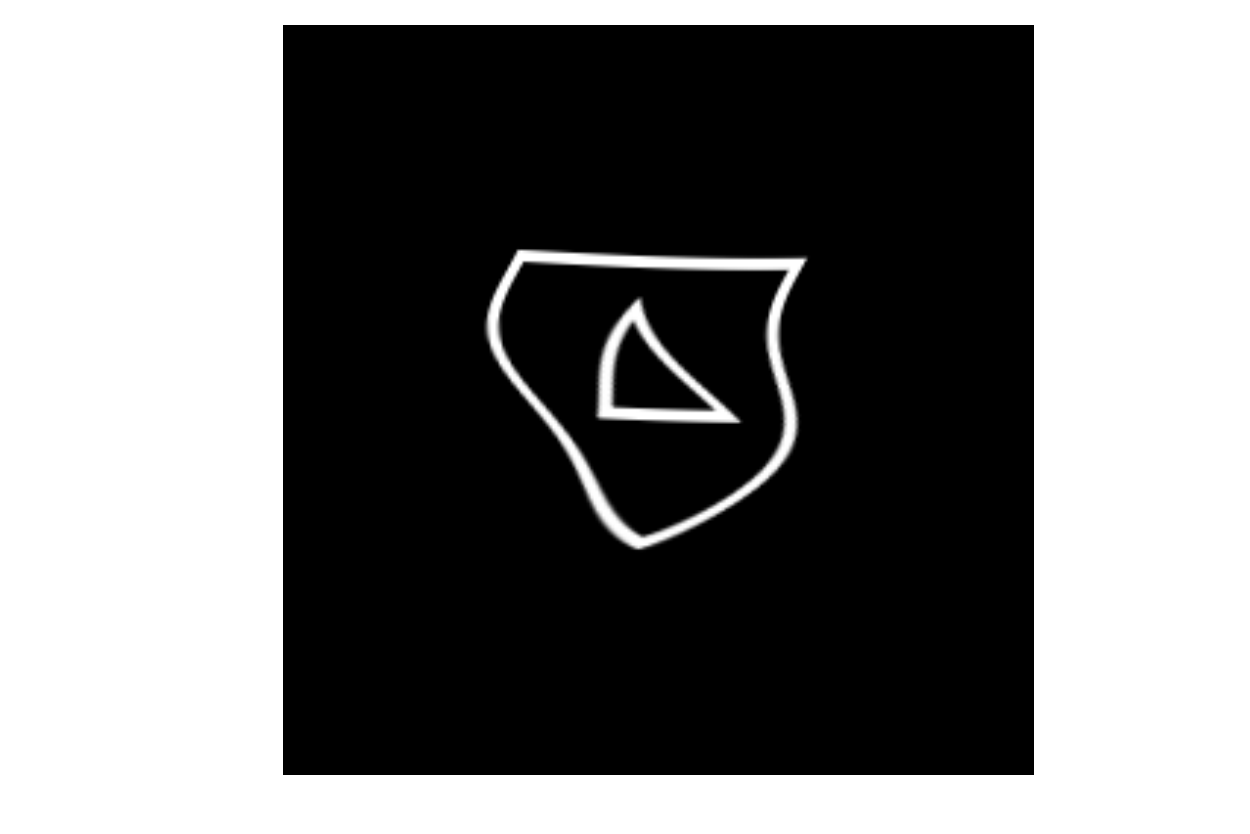}
        \includegraphics[height=0.16\linewidth, clip=true, trim=-60pt -250pt -60pt -240pt]{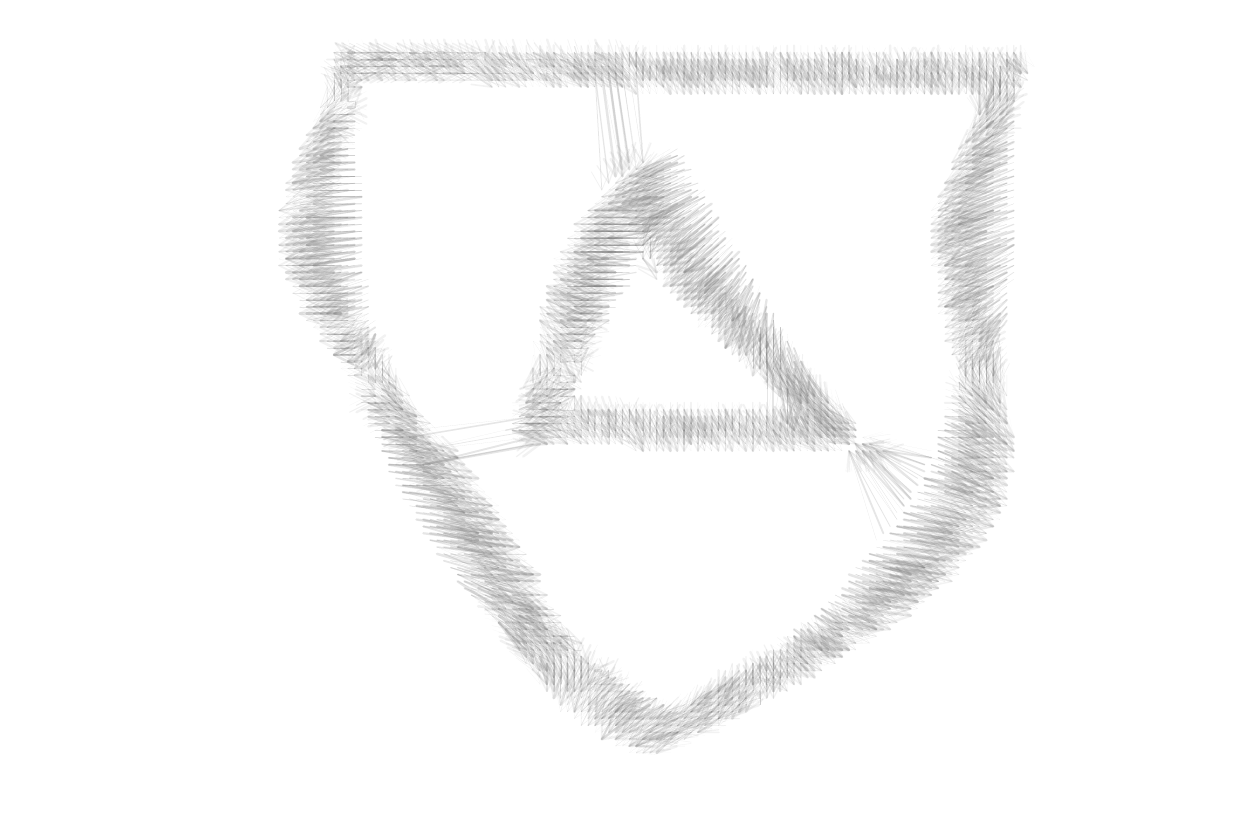}
        & \includegraphics[height=0.16\linewidth, clip=true, trim=130pt 10pt 100pt 0pt]{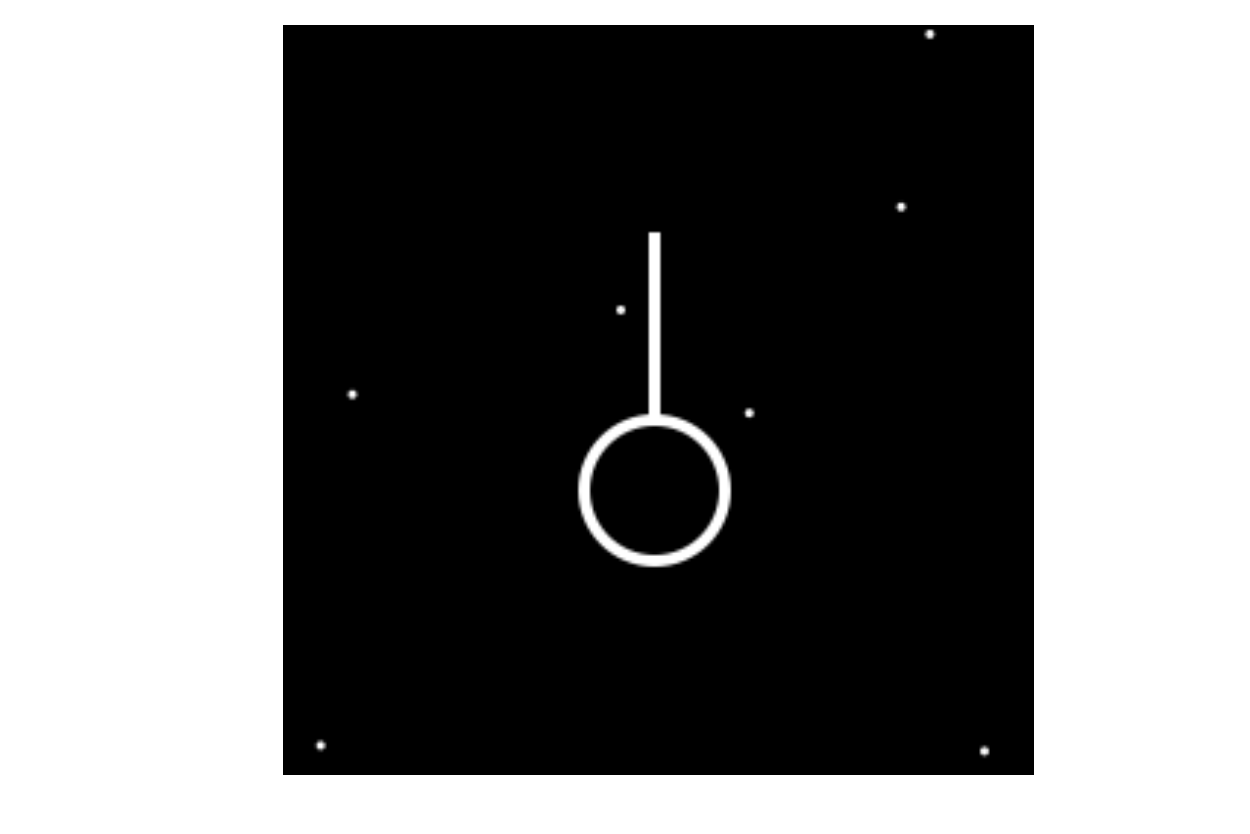}
        \includegraphics[height=0.16\linewidth, clip=true, trim=150pt 15pt 120pt 5pt]{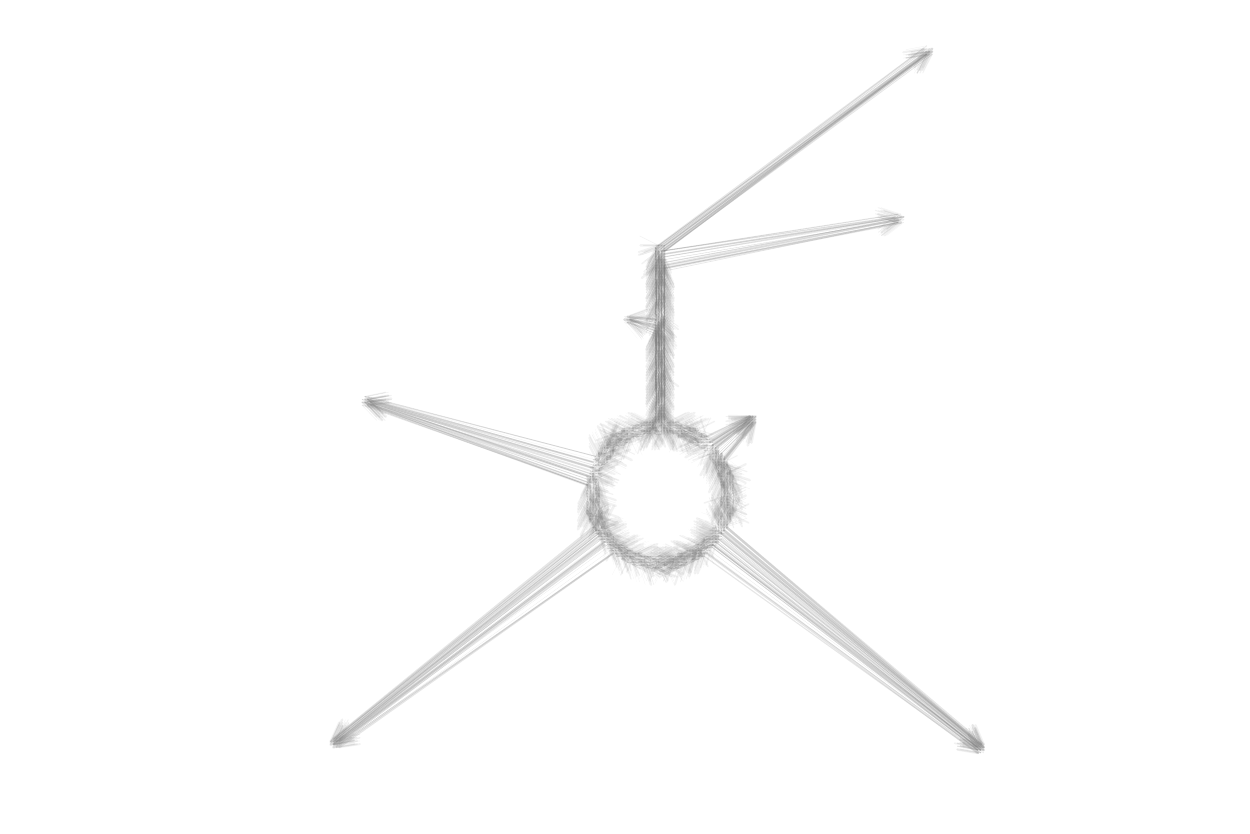}
        & \includegraphics[height=0.16\linewidth, clip=true, trim=130pt 10pt 100pt 0pt]{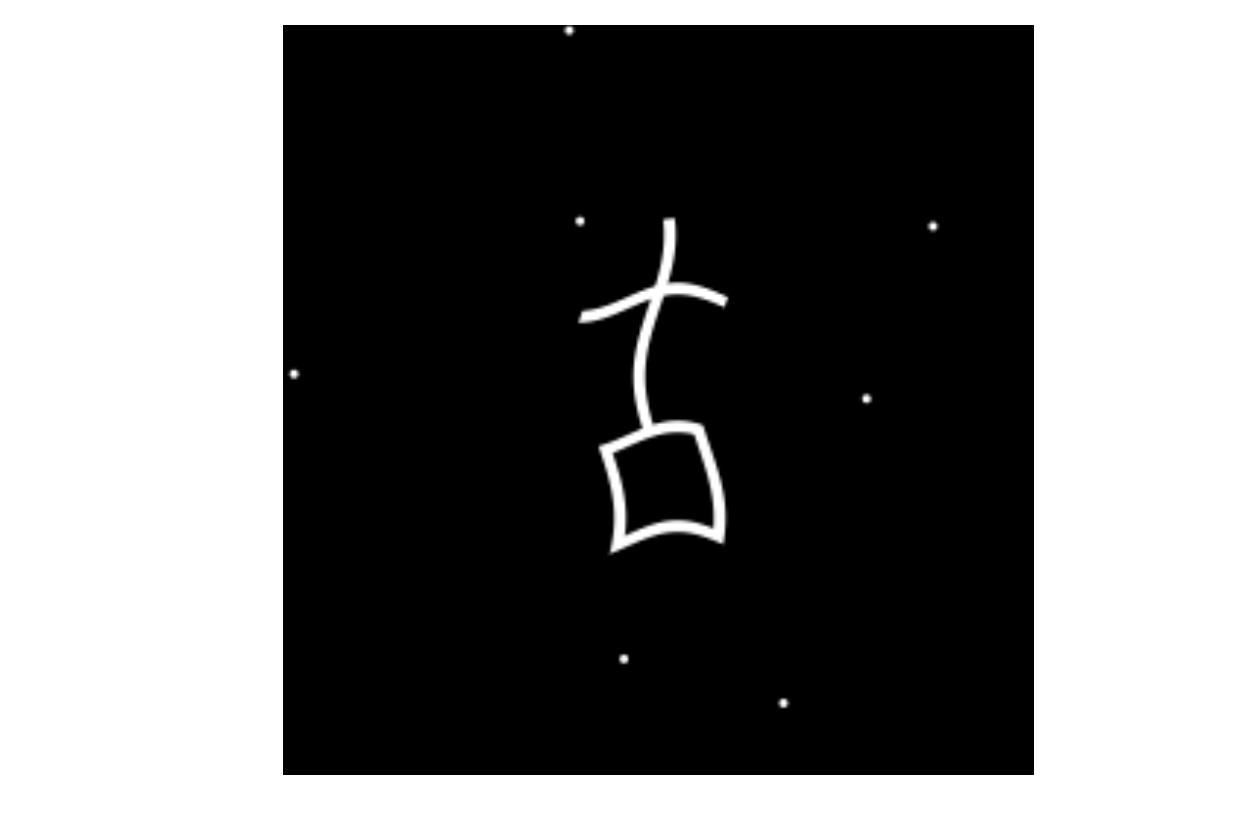}
        \includegraphics[height=0.16\linewidth, clip=true, trim=80pt -10pt 50pt 5pt]{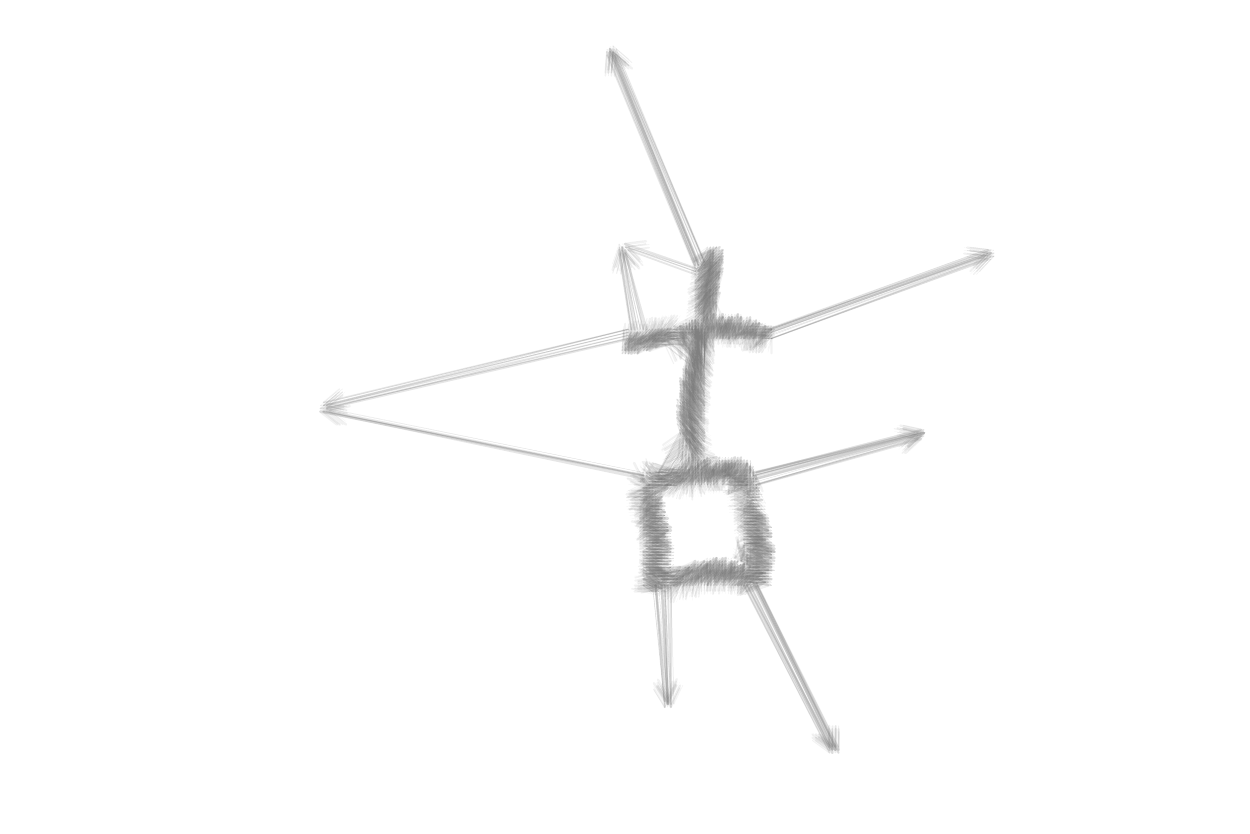} \\
        $W_2(\rho_{\text{org}}, \rho_{\text{dis}}) = 1.9860 \cdot 10^{-3}$ &
        $W_2(\rho_{\text{org}}, \rho_{\text{dis}}) = 3.1034 \cdot 10^{-3}$ &
        $W_2(\rho_{\text{org}}, \rho_{\text{dis}}) = 5.77112\cdot10^{-3}$ \\
    \end{tabular}}
    \caption{ 
    Instances of corrupted data regarding non-affine deformations and
    impulsive/salt noise considered in the robustness analysis.
    The accompanying vector fields illustrate the optimal Wasserstein-2 transport
    between the corrupted datum $\rho_\text{dis}$ and the true template $\rho_\text{org}$.
    }
    \label{fig:academic_distortions}
\end{figure}

\paragraph{Non-affine Deformations}

\begin{SCfigure}[3][t]
    \includegraphics[width=0.29\linewidth, clip=true, trim=10pt 10pt 20pt 20pt]{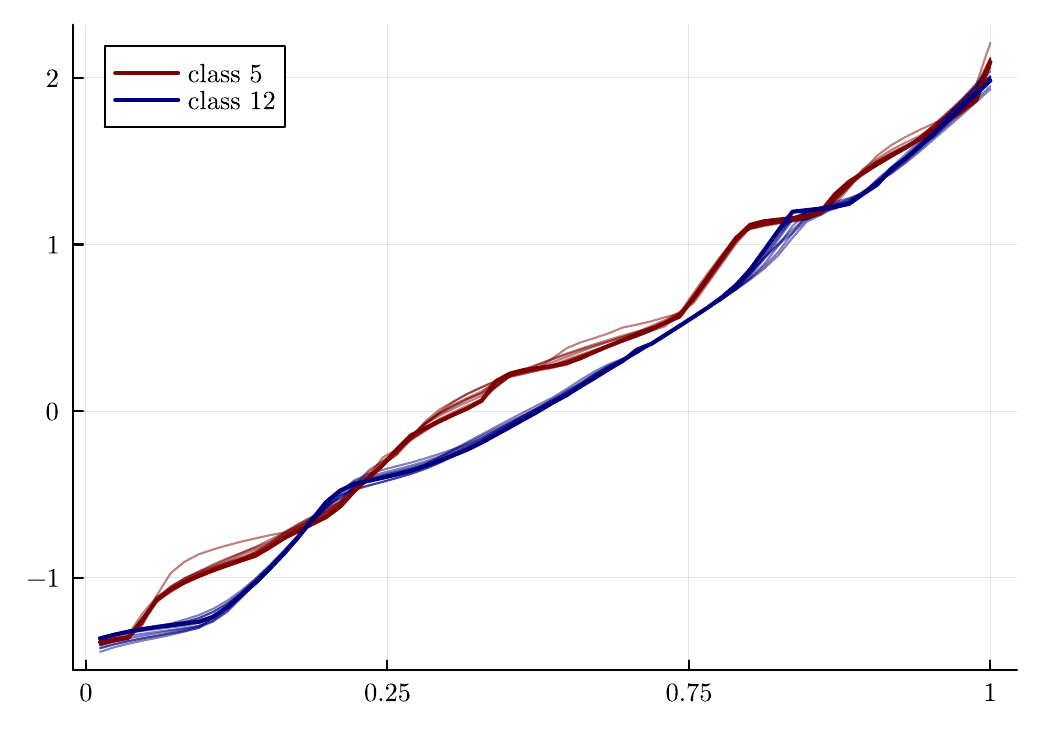}
    \includegraphics[width=0.29\linewidth, clip=true, trim=10pt 10pt 20pt 20pt]{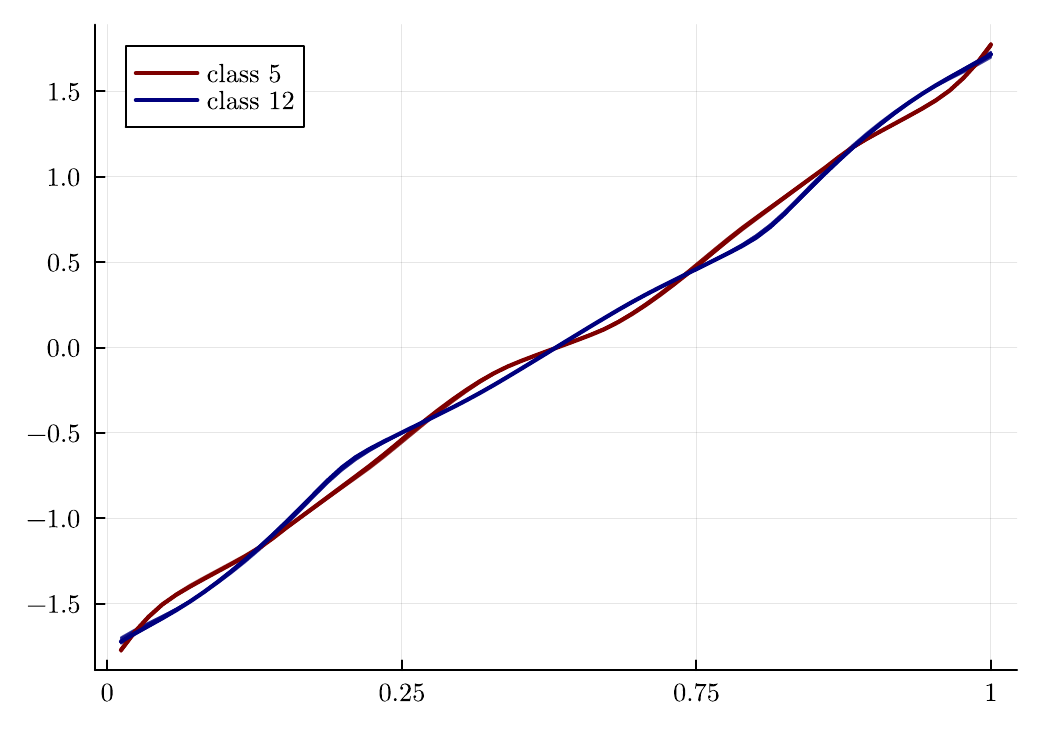}
    \caption{
    Visualization of
    \mNRCDT{} (left) and \aNRCDT{} (right)
    for classes 5 and 12 of academic dataset,
    each of size $10$ and generated by,
    first,
    random non-affine deformations
    induced by sine/cosine functions
    with amplitudes in $[2.5,7.5]$
    and frequencies in $[0.5,2.0]$,
    and, second,
    random affine transformation as in
    Figure~\ref{fig:academic_NRCDT_affine}.
    }
    \label{fig:academic_NRCDT_non-affine}
\end{SCfigure}

\begin{table}[t!]
    \caption{
    NT classification accuracies for the academic dataset
    with $10$ samples per class and
    different parameter ranges for
    random non-affine distortions.
    Additionally,
    random affine transformations
    are applied with
    scaling in $[0.75, 1.0]$,
    shearing in $[-5^\circ,5^\circ]$,
    rotation angles in $[0^\circ,360^\circ]$ and
    pixel shifts in $[-20,20]$.
    The best result per dataset and angle is highlighted.
    }
    \resizebox{\linewidth}{!}{%
    \begin{tabular}{*{19}{l}}
    \toprule
    angles
    & \multicolumn{3}{l}{no non-affine distortion} 
    & \multicolumn{3}{l}{freq.\ in $[0.5, 2.0]$, amp.\ in $[2.5, 7.5]$}
    & \multicolumn{3}{l}{freq.\ in $[0.5, 2.0]$, amp.\ in $[8.0, 13.0]$}
    & \multicolumn{3}{l}{freq.\ in $[0.5, 4.0]$, amp.\ in $[0.5, 2.0]$}
    & \multicolumn{3}{l}{freq.\ in $[0.5, 4.0]$, amp.\ in $[0.5, 7.5]$} 
    & \multicolumn{3}{l}{freq.\ in $[0.5, 4.0]$, amp.\ in $[2.5, 7.5]$}
    \\
    \cmidrule(lr){2-4}
    \cmidrule(lr){5-7}
    \cmidrule(lr){8-10}
    \cmidrule(lr){11-13}
    \cmidrule(lr){14-16}
    \cmidrule(lr){17-19} 
    & R-CDT & \mNRCDT & \aNRCDT
    & R-CDT & \mNRCDT & \aNRCDT 
    & R-CDT & \mNRCDT & \aNRCDT
    & R-CDT & \mNRCDT & \aNRCDT
    & R-CDT & \mNRCDT & \aNRCDT
    & R-CDT & \mNRCDT & \aNRCDT
    \\
    \midrule
    1 
    & $0.2583$ & \bm{$0.3083$} & \bm{$0.3083$} 
    & $0.2250$ & \bm{$0.2666$} & \bm{$0.2666$} 
    & $0.2250$ & \bm{$0.2500$} & \bm{$0.2500$} 
    & $0.2416$ & \bm{$0.2666$} & \bm{$0.2666$} 
    & $0.2166$ & \bm{$0.2250$} & \bm{$0.2250$} 
    & $0.2000$ & \bm{$0.2166$} & \bm{$0.2166$} 
    \\
    2  
    & $0.2883$ & \bm{$0.3083$} & $0.2916$ 
    & $0.2250$ & \bm{$0.2750$} & \bm{$0.2750$} 
    & $0.2250$ & \bm{$0.2416$} & $0.2333$ 
    & $0.2416$ & \bm{$0.2583$} & $0.2500$ 
    & $0.2166$ & $0.2166$ & \bm{$0.2333$} 
    & $0.2000$ & \bm{$0.2250$} & $0.2166$ 
    \\
    4 
    & $0.1833$ & \bm{$0.5250$} & $0.3666$ 
    & $0.1750$ & \bm{$0.4333$} & $0.2500$ 
    & $0.2083$ & \bm{$0.4000$} & $0.2083$ 
    & $0.1916$ & \bm{$0.4583$} & $0.3166$ 
    & $0.1666$ & \bm{$0.3916$} & $0.2333$ 
    & $0.1750$ & \bm{$0.3750$} & $0.2083$ 
    \\
    8
    & $0.2083$ & \bm{$0.6500$} & $0.4916$ 
    & $0.1583$ & \bm{$0.6000$} & $0.4166$ 
    & $0.1333$ & \bm{$0.5583$} & $0.4333$ 
    & $0.1833$ & \bm{$0.6500$} & $0.5166$ 
    & $0.1966$ & \bm{$0.5833$} & $0.4416$ 
    & $0.1916$ & \bm{$0.5833$} & $0.4250$ 
    \\
    16
    & $0.2000$ & \bm{$0.9166$} & $0.8750$ 
    & $0.1583$ & $0.8250$ & \bm{$0.8666$} 
    & $0.1250$ & $0.7083$ & \bm{$0.7500$} 
    & $0.1583$ & \bm{$0.9083$} & $0.8833$ 
    & $0.1916$ & $0.7916$ & \bm{$0.8083$} 
    & $0.2083$ & \bm{$0.7583$} & $0.7250$ 
    \\
    32
    & $0.2000$ & \bm{$1.0000$} & \bm{$1.0000$} 
    & $0.1666$ & $0.9916$ & \bm{$1.0000$} 
    & $0.1250$ & $0.8083$ & \bm{$0.9083$} 
    & $0.1666$ & \bm{$1.0000$} & \bm{$1.0000$} 
    & $0.1916$ & $0.9500$ & \bm{$0.9666$} 
    & $0.2000$ & $0.9083$ & \bm{$0.9250$} 
    \\
    64 
    & $0.1916$ & \bm{$1.0000$} & \bm{$1.0000$} 
    & $0.1666$ & $0.9833$ & \bm{$1.0000$} 
    & $0.1250$ & $0.8000$ & \bm{$0.9166$} 
    & $0.1583$ & \bm{$1.0000$}& \bm{$1.0000$} 
    & $0.1916$ & $0.9333$ & \bm{$0.9666$} 
    & $0.2000$ & $0.9083$ & \bm{$0.9166$} 
    \\
    128 
    & $0.1916$ & \bm{$1.0000$} & \bm{$1.0000$} 
    & $0.1666$ & $0.9916$ & \bm{$1.0000$} 
    & $0.1250$ & $0.8166$ & \bm{$0.9250$} 
    & $0.1666$ & \bm{$1.0000$} & \bm{$1.0000$} 
    & $0.1916$ & $0.9333$ & \bm{$0.9750$} 
    & $0.2000$ & $0.9000$ & \bm{$0.9250$} 
    \\
    256 
    & $0.1916$ & $0.9916$ & \bm{$1.0000$} 
    & $0.1666$ & $0.9916$ & \bm{$1.0000$} 
    & $0.1580$ & $0.8416$ & \bm{$1.0000$} 
    & $0.1666$ & \bm{$1.0000$} & \bm{$1.0000$} 
    & $0.1916$ & $0.9416$ & \bm{$0.9750$} 
    & $0.2000$ & $0.9000$ & \bm{$0.9250$} 
    \\
    \midrule
    Eucl.\
    & $0.0833$ & &
    & $0.0583$ & &
    & $0.0666$ & &
    & $0.0750$ & &
    & $0.0750$ & &
    & $0.0750$ & & 
    \\
    \bottomrule
    \end{tabular}}
    \label{tab:academic_NT_non-affine}
\end{table}

To generate non-affine deformations
of an ($N{\times}N$)-pixel image,
we assign to the $(j,k)$th pixel
the bi-quadratically interpolated gray value
at the perturbed location
\begin{equation*}
    \bigl(j + a_1 \sin\bigl(\tfrac{2\pi f_1}{N} \, k\bigr), k + a_2 \cos\bigl(\tfrac{2\pi f_2}{N} \, j\bigr)\bigr)
\end{equation*}
with fixed random frequencies $f_1, f_2$ and amplitudes $a_1, a_2$.
Figuratively,
this deformation yields local bendings of the template symbols;
see Figure~\ref{fig:academic_distortions} (left).
The datasets for the following experiments consist of
10 non-affinely deformed samples for each of the twelve templates
in Figure~\ref{fig:academic_dataset},
followed by the application of a random affine transformation.
The impact on the \mNRCDT{}s and \aNRCDT{}s is illustrated
in Figure~\ref{fig:academic_NRCDT_non-affine}.
While the non-affine deformations only have
a very small impact on the \aNRCDT{} representation,
we clearly observe within-class variations for the \mNRCDT{}.
For classification,
we again use the NT approach
and repeat the entire experiment
for different academic datasets
whose amplitudes and frequencies 
of the random non-affine deformations
lie in varying parameter ranges.
Our results are reported
in Table~\ref{tab:academic_NT_non-affine}.
We observe that
our \mNRCDT{} and \aNRCDT{} representations
clearly outperform R-CDT and Euclidean representations.
In particular,
for a small amount of non-affine distortions
both---\mNRCDT{} and \aNRCDT{}---yield
perfect classification
for sufficiently many angles.
With increasing amount of distortions,
we see that \aNRCDT{} performs better
than \mNRCDT{},
as expected by our observations
based on Figure~\ref{fig:academic_distortions}.

\paragraph{Salt Noise}

\begin{figure}[t]
    \centering
    \includegraphics[width=\linewidth, clip=true, trim=135pt 360pt 105pt 20pt]{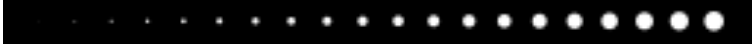}
    \resizebox{\linewidth}{!}{%
    \tiny%
    \begin{tabularx}{\linewidth}{@{\hspace{0.0cm}} *{21}{Y} @{\hspace{0.65cm}}}
        $\ell$ & 1 & 2 & 3 & 4 & 5 & 6 & 7 & 8 & 9 & 10 & 11 & 12 & 13 & 14 & 15 & 16 & 17 & 18 & 19 & 20
    \end{tabularx}}
    \vspace{-15pt}
    \caption{Scala of exemplarily rendered salt noise of strength $\ell \in \{1,\ldots,20\}$
        used in the experiments.
        The salt noise illustrated in Figure~\ref{fig:academic_distortions} has strength $\ell = 9$.
        }
    \label{fig:salt-noise}
\end{figure}

\begin{SCfigure}[3][t]
    \includegraphics[width=0.29\linewidth, clip=true, trim=10pt 10pt 20pt 20pt]{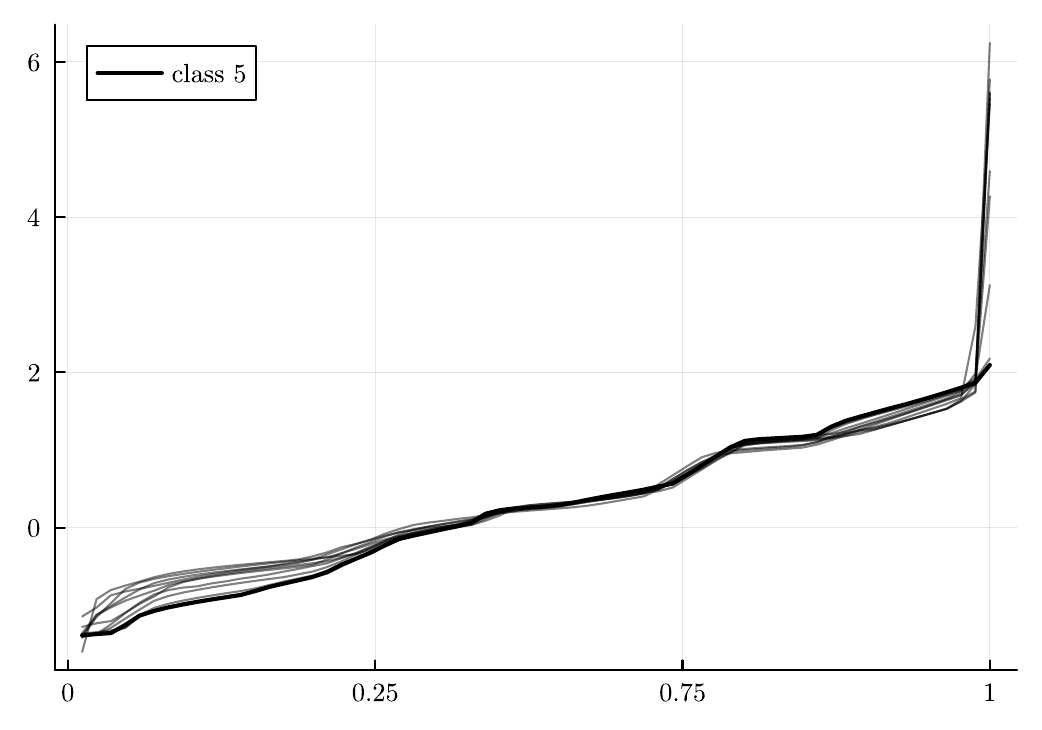}
    \includegraphics[width=0.29\linewidth, clip=true, trim=10pt 10pt 20pt 20pt]{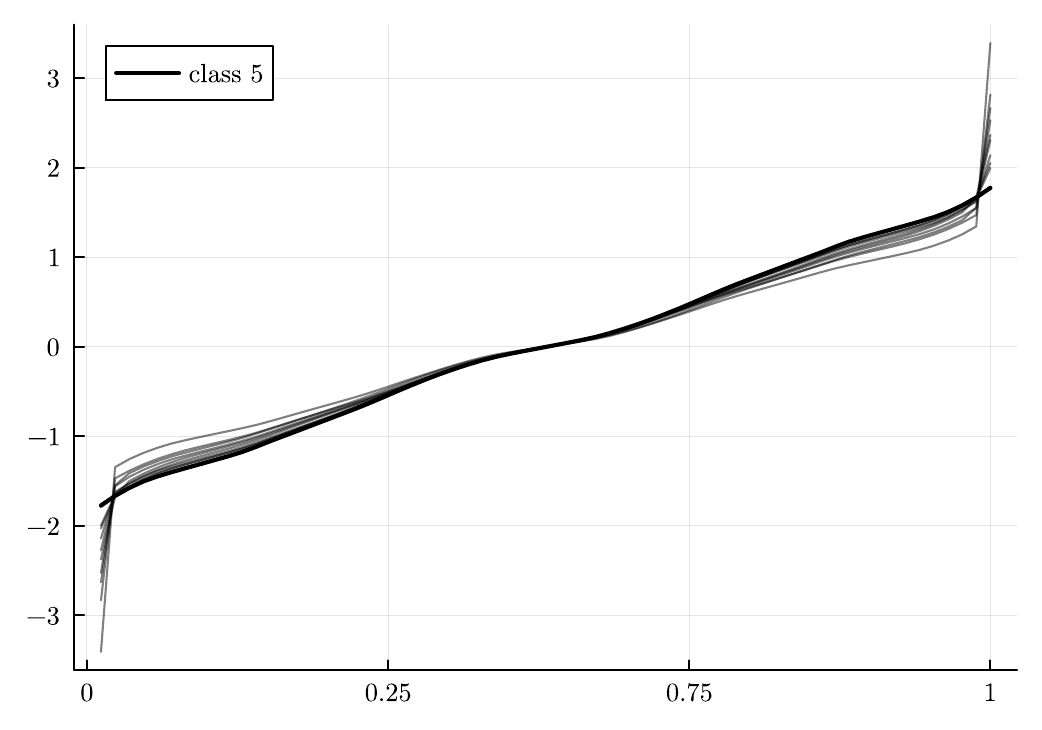}
    \caption{
    Visualization of \mNRCDT{} and \aNRCDT{} for
    class 5 of academic dataset
    of size $10$,
    generated by,
    first,
    random affine transformations with
    scaling in $[0.75,1.0]$,
    shear in $[-5^\circ,5^\circ]$,
    rotation in $[0^\circ,360^\circ]$,
    shift in $[-20,20]$
    and, second,
    adding salt noise
    of strength 9
    according to Figure~\ref{fig:salt-noise}
    at $4$ to $7$ locations.
    }
    \label{fig:academic_NRCDT_noise}
\end{SCfigure}

\begin{figure}[t!]
    \resizebox{\linewidth}{!}{%
    \scriptsize
    \begin{tabular}{c @{\hspace{3pt}} c @{\hspace{10pt}} c @{\hspace{3pt}} c @{\hspace{10pt}} c @{\hspace{3pt}} c @{\hspace{3pt}} c}
        \multicolumn{2}{c}{16 equispaced Radon angles}
        & \multicolumn{2}{c}{64 equispaced Radon angles}
        & \multicolumn{2}{c}{256 equispaced Radon angles}
        \\
        \includegraphics[height=0.155\linewidth, clip=true, trim=10pt 10pt 37pt 5pt]{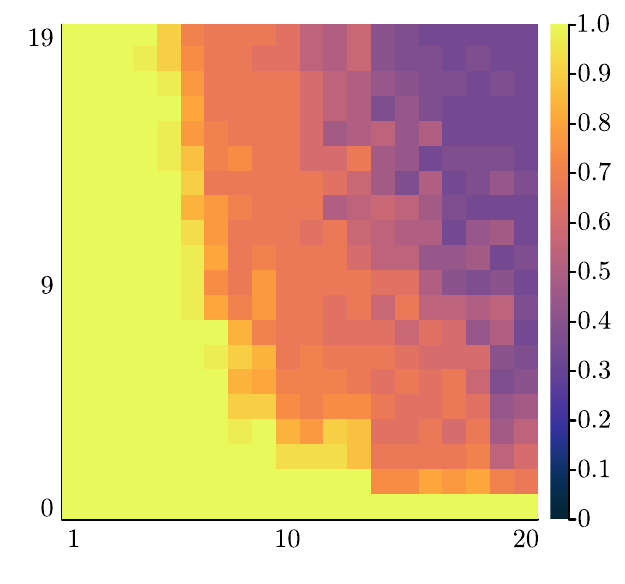}
        & \includegraphics[height=0.155\linewidth, clip=true, trim=10pt 10pt 37pt 5pt]{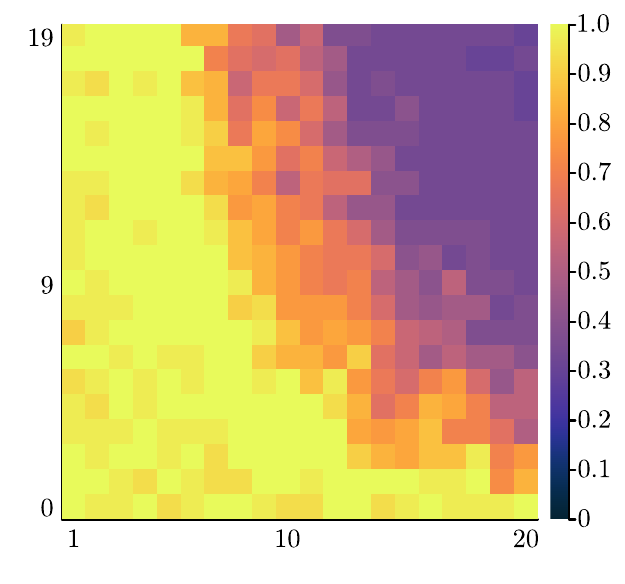}
        &\includegraphics[height=0.155\linewidth, clip=true, trim=10pt 10pt 37pt 5pt]{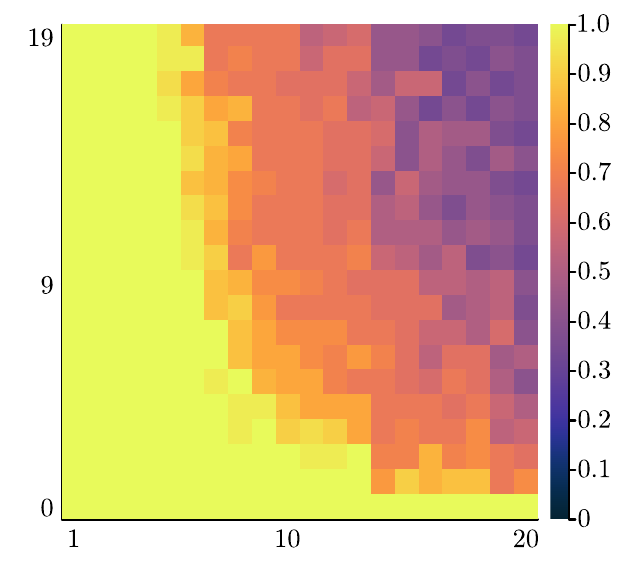}
        &\includegraphics[height=0.155\linewidth, clip=true, trim=10pt 10pt 37pt 5pt]{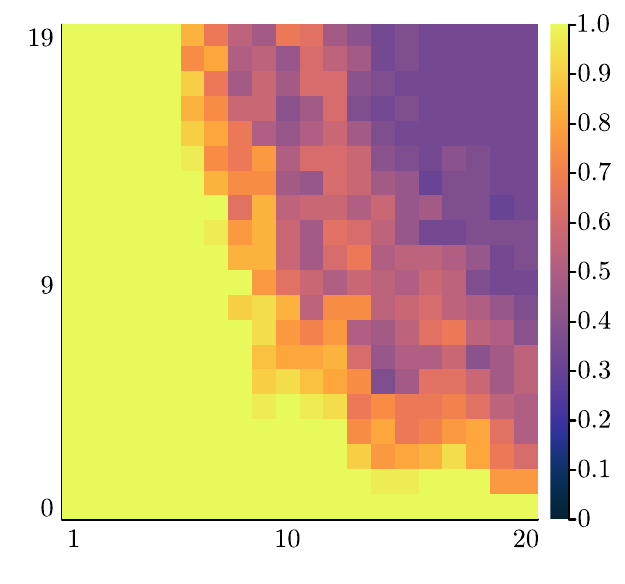}
        &\includegraphics[height=0.155\linewidth, clip=true, trim=10pt 10pt 37pt 5pt]{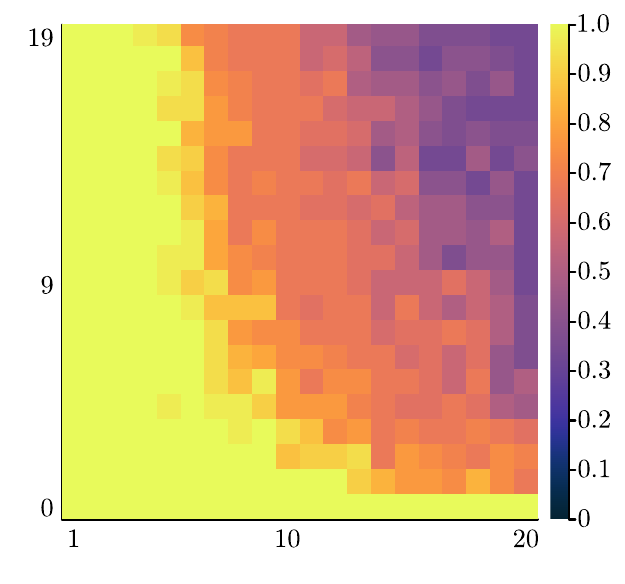}     
        &\includegraphics[height=0.155\linewidth, clip=true, trim=10pt 10pt 37pt 5pt]{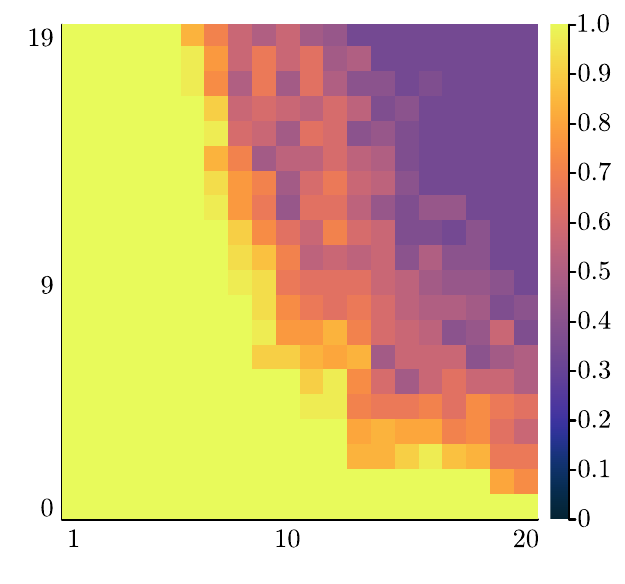} 
        &\includegraphics[height=0.155\linewidth, clip=true, trim=0pt 0pt 20pt 5pt]{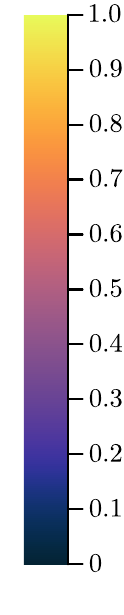}\\
        \mNRCDT
        & \aNRCDT
        & \mNRCDT
        & \aNRCDT
        & \mNRCDT
        & \aNRCDT
    \end{tabular}}%
    \caption{
        Phase transition of NT classification accuracies 
        for 3-class academic datasets with 10 samples per class
        and different salt noise 
        (vertical: component numbers, 
        horizontal: noise strengths, cf.\ Figure~\ref{fig:salt-noise}).
        The random affine transformations consist of rotations with angle in $[0^\circ,360^\circ]$
        and pixel shifts in $[-20,20]$.
        The experiment is repeated for varying numbers of Radon angles.
    }
    \label{fig:salt_added}
\end{figure}

In style of salt-and-pepper noise,
we use the term \emph{salt noise} for 
image distortions caused 
by adding white discs with fixed radius to the image.
Figure~\ref{fig:salt-noise} shows a gamut
of rendered salt noise 
used in our experiments.
The impact of salt noise to the \mNRCDT\ and \aNRCDT\ 
is illustrated in Figure~\ref{fig:academic_NRCDT_noise}
and mainly consists in heavy disturbances 
at the end points of the domain $(0,1)$,
which also affect the \mNRCDT{s} and \aNRCDT{s} as a whole.
The 3-class academic datasets in this experiment are generated as follows:
first, 
the template images 1, 5, and 12 in Figure~\ref{fig:academic_dataset}
are randomly affinely transformed 
(without anisotropic scaling and shearing to avoid additional disturbances of \aNRCDT{s});
second,
the obtained images are corrupted by
salt noise.
The final datasets consist of 10 samples per class
and are classified using the NT approach.
Repeating the experiment for datasets
with different noise strength and component numbers,
we obtain the phase transitions in Figure~\ref{fig:salt_added}.
The \aNRCDT\ requires more angles 
but then performs better than \mNRCDT,
whose phase transition is less sharp.

\begin{figure}[t]
    \resizebox{\linewidth}{!}{%
    \scriptsize
    \begin{tabular}{c @{\hspace{3pt}} c @{\hspace{10pt}} c @{\hspace{3pt}} c @{\hspace{10pt}} c @{\hspace{3pt}} c @{\hspace{3pt}} c}
        \multicolumn{2}{c}{16 equispaced Radon angles}
        & \multicolumn{2}{c}{64 equispaced Radon angles}
        & \multicolumn{2}{c}{256 equispaced Radon angles}
        \\
        \includegraphics[height=0.147\linewidth, clip=true, trim=10pt 10pt 0pt 5pt]{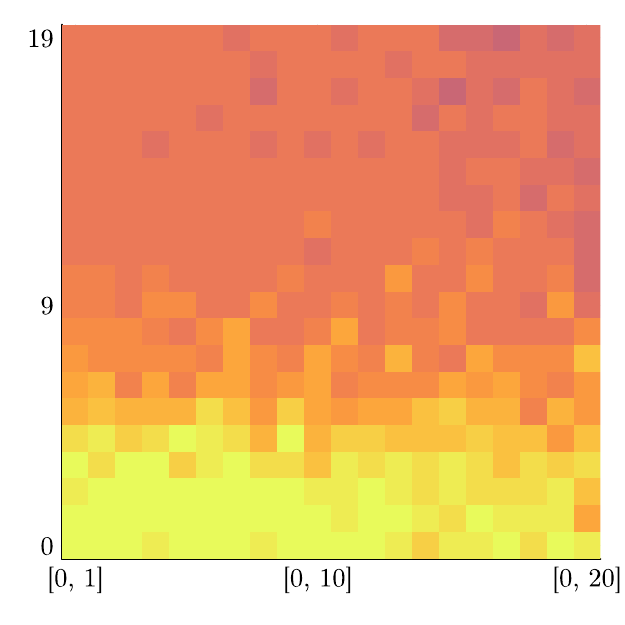}
        & \includegraphics[height=0.147\linewidth, clip=true, trim=10pt 10pt 0pt 5pt]{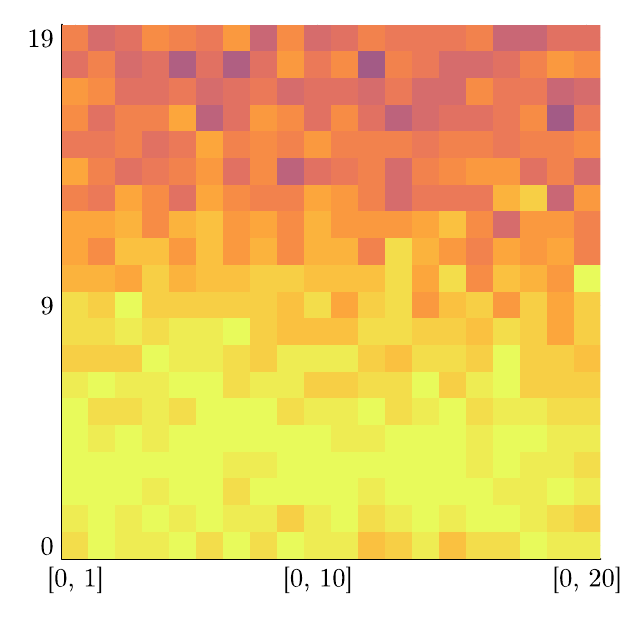}
        &\includegraphics[height=0.147\linewidth, clip=true, trim=10pt 10pt 0pt 5pt]{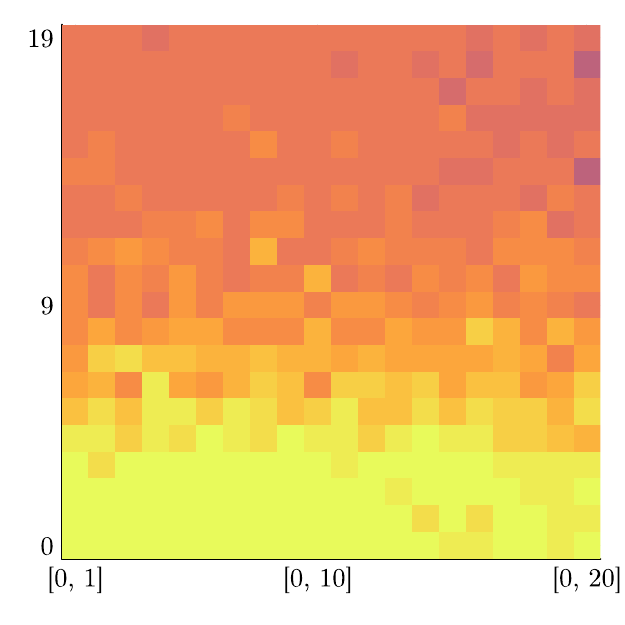}
        &\includegraphics[height=0.147\linewidth, clip=true, trim=10pt 10pt 0pt 5pt]{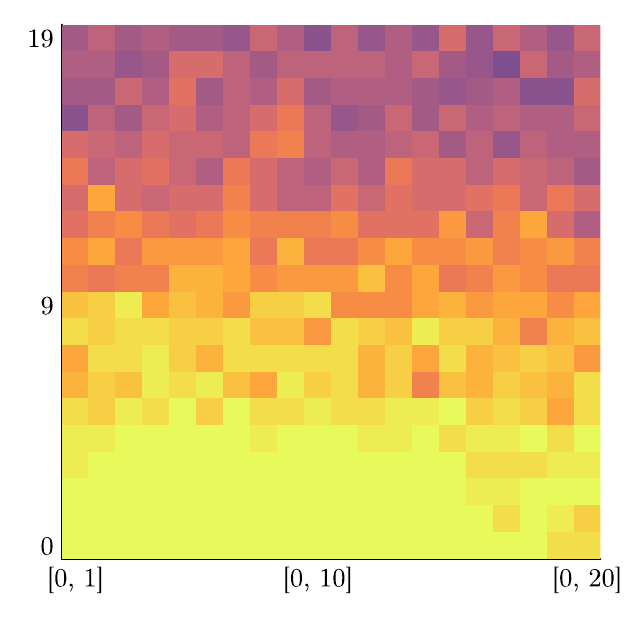}
        &\includegraphics[height=0.147\linewidth, clip=true, trim=10pt 10pt 0pt 5pt]{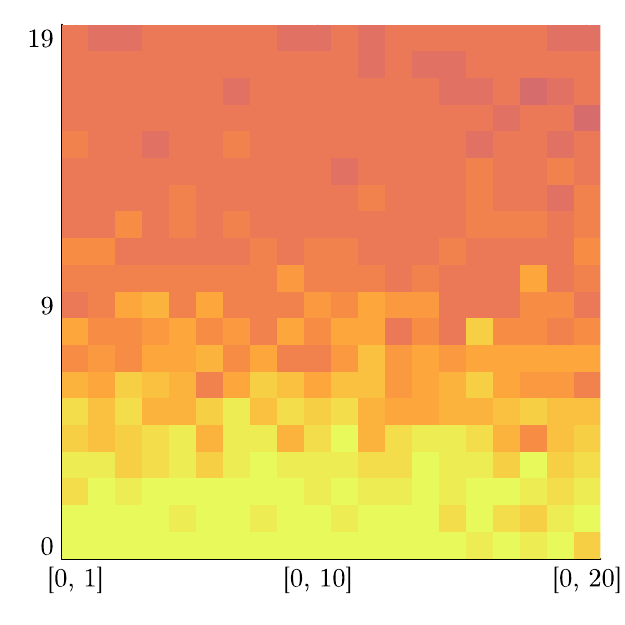}     
        &\includegraphics[height=0.147\linewidth, clip=true, trim=10pt 10pt 0pt 5pt]{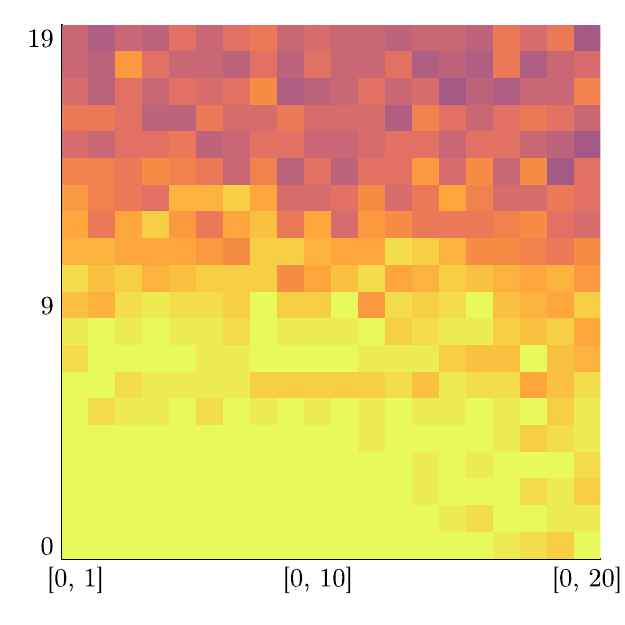} 
        &\includegraphics[height=0.147\linewidth, clip=true, trim=10pt 0pt 20pt 5pt]{Images/colorbar_thermal.pdf}\\
        \mNRCDT
        & \aNRCDT
        & \mNRCDT
        & \aNRCDT
        & \mNRCDT
        & \aNRCDT
    \end{tabular}}%
    \caption{
        Phase transition of NT classification accuracies 
        for 3-class academic datasets with 10 samples per class,
        non-affine deformations with frequencies in $[0,2]$
        and variable amplitude ranges (horizontal),
        as well as salt noise of strength 9 according to Figure~\ref{fig:salt-noise}
        and variable location numbers (vertical).
        The random affine transformations consist of rotations with angles in $[0^\circ,360^\circ]$
        and pixel shifts in $[-20,20]$.
        The experiment is repeated for varying numbers of Radon angles.
    }
    \label{fig:salt_elastic_added}
\end{figure}

\paragraph{Non-affine Deformations and Salt Noise}

In the final academic experiment,
we study the combined impact 
of non-affine deformations 
and salt noise.
Similar to before,
we consider 3-class academic datasets 
based on the symbols 1, 5, and 12 from Figure~\ref{fig:academic_dataset}.
The datasets themselves,
each consisting of 10 samples per class,
are generated by,
first, distorting via a non-affine deformation,
second, applying an affine transformation
(again without anisotropic scaling and shearing),
and third, corrupting with salt noise.
For classification,
we employ the NT approach.
The resulting phase transitions are reported
in Figure~\ref{fig:salt_elastic_added},
where the range of the random amplitudes of the non-affine deformation
and the number of locations corrupted by salt noise are varied.
We observe that
the salt noise has
more impact on the
classification success
and
that the area of near perfect classification
is larger for \aNRCDT{} than \mNRCDT{},
indicating more robustness against distortions.

\begin{table}[]
    \caption{
        NN classification accuracies 
        (mean plus/minus standard deviation)
        for academic datasets 
        with all 12 symbols from Figure~\ref{fig:academic_dataset}
        and 100 samples per class.
        In 2.\ and 4.,
        random non-affine deformations
        with frequencies in $[0.5, 2.0]$
        and amplitudes in $[2.5,7.5]$
        are applied.
        Affine transformations consist of
        rotations in $[0^\circ, 360^\circ]$,
        shifts in $[-20, 20]$
        and,
        in 1.\ and 2.,
        scaling in $[0.5, 1.25]$
        and shearing in $[-45^\circ,45^\circ]$
        as well as, 
        in 3.\ and 4.,
        scaling in $[0.75,1.0]$
        and shearing in $[-5^\circ, 5^\circ]$.
        In 3.\ and 4.,
        salt noise is of strength 9, cf.\ Figure~\ref{fig:salt-noise},
        with random location numbers in $[4,7]$.
        The best result per dataset and training number is highlighted.
    }
    \resizebox{\linewidth}{!}{%
    \begin{tabular}{l l l l l l l l l}
        \toprule
        dataset 
        & \multicolumn{4}{l}{5 training samples} 
        & \multicolumn{4}{l}{10 training samples} \\
        \cmidrule(r){2-5}
        \cmidrule(l){6-9}
        & Euclidean & R-CDT & \mNRCDT & \aNRCDT & Euclidean & R-CDT & \mNRCDT & \aNRCDT \\
        \midrule
        1. affine 
        & $0.0915\pm0.0052$ & $0.1217\pm0.0094$ & \bm{$0.9999\pm0.0001$} & $0.9010\pm0.0247$
        & $0.0958\pm0.0083$ & $0.1320\pm0.0091$ & \bm{$1.0000\pm0.0000$} & $0.9595\pm0.0122$ \\
        2. non-affine, affine
        & $0.0903\pm0.0081$ & $0.1212\pm0.0110$ & \bm{$0.9899\pm0.0048$} & $0.9055\pm0.0264$
        & $0.0922\pm0.0080$ & $0.1343\pm0.0111$ & \bm{$0.9962\pm0.0031$} & $0.9623\pm0.0128$ \\
        3. affine, salt
        & $0.1003\pm0.0072$ & $0.1707\pm0.0111$ & $0.5669\pm0.0226$ & \bm{$0.6378\pm0.0300$}
        & $0.1103\pm0.0085$ & $0.2003\pm0.0085$ & $0.6542\pm0.0209$ & \bm{$0.7236\pm0.0157$} \\
        4. non-affine, affine, salt
        & $0.1042\pm0.0075$ & $0.1666\pm0.0132$ & $0.5412\pm0.0294$ & \bm{$0.6325\pm0.0271$}
        & $0.1108\pm0.0068$ & $0.1904\pm0.0080$ & $0.6296\pm0.0223$ & \bm{$0.7149\pm0.0223$} \\
        \bottomrule
    \end{tabular}}%
    \label{tab:k-NN_academic}
\end{table}

\begin{figure}
    \resizebox{\linewidth}{!}{%
    \scriptsize
        \begin{tabular}{c @{\hspace{3pt}} c @{\hspace{10pt}} c @{\hspace{3pt}} c @{\hspace{3pt}} c}
            \multicolumn{2}{c}{3. experiment: affine, salt}
            & \multicolumn{2}{c}{4. experiment: non-affine, affine, salt} \\
            \includegraphics[width=0.23\linewidth, clip=true, trim=10pt 10pt 70pt 0pt]{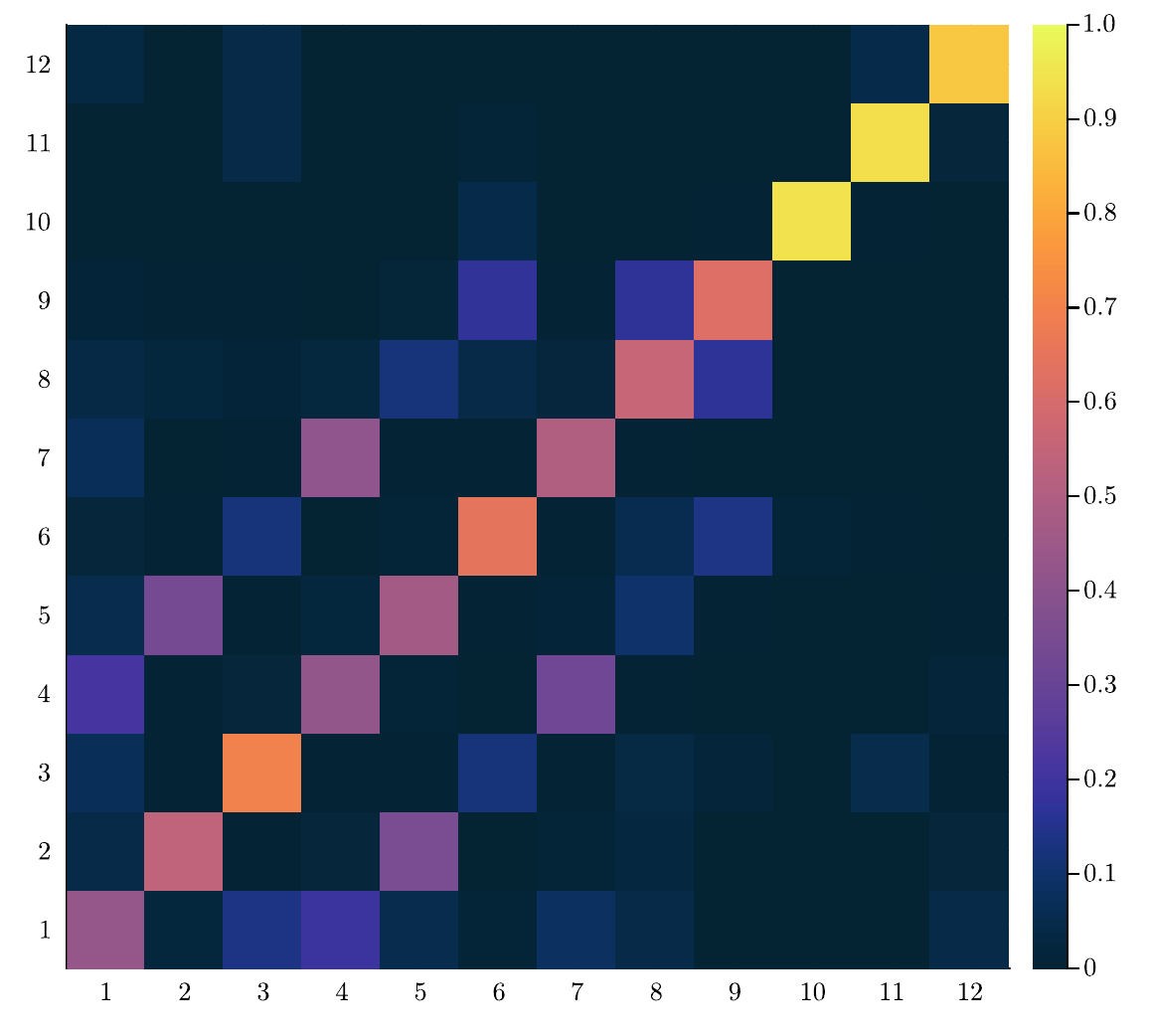}
            &\includegraphics[width=0.23\linewidth, clip=true, trim=10pt 10pt 70pt 0pt]{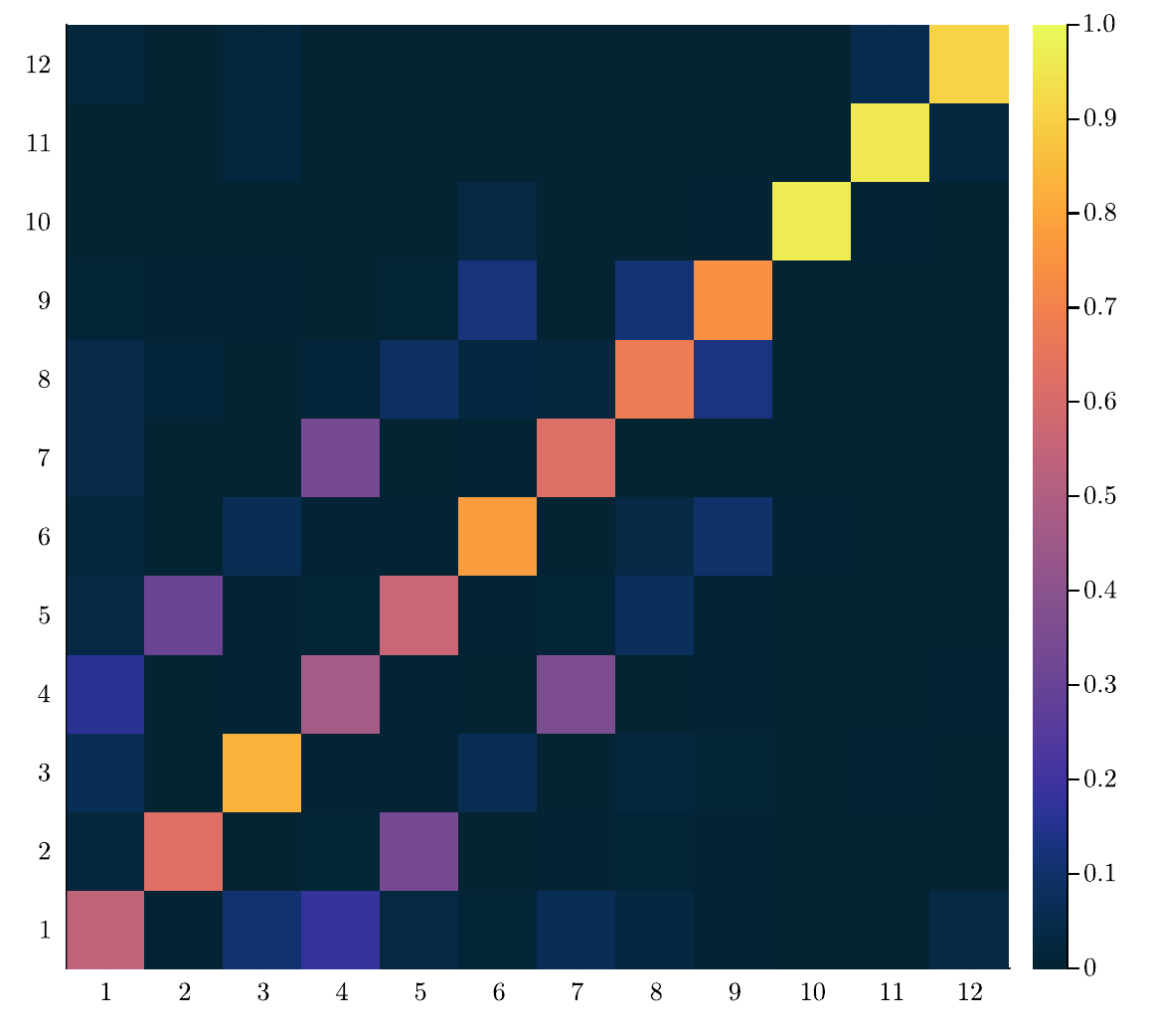}
            &\includegraphics[width=0.23\linewidth, clip=true, trim=10pt 10pt 70pt 0pt]{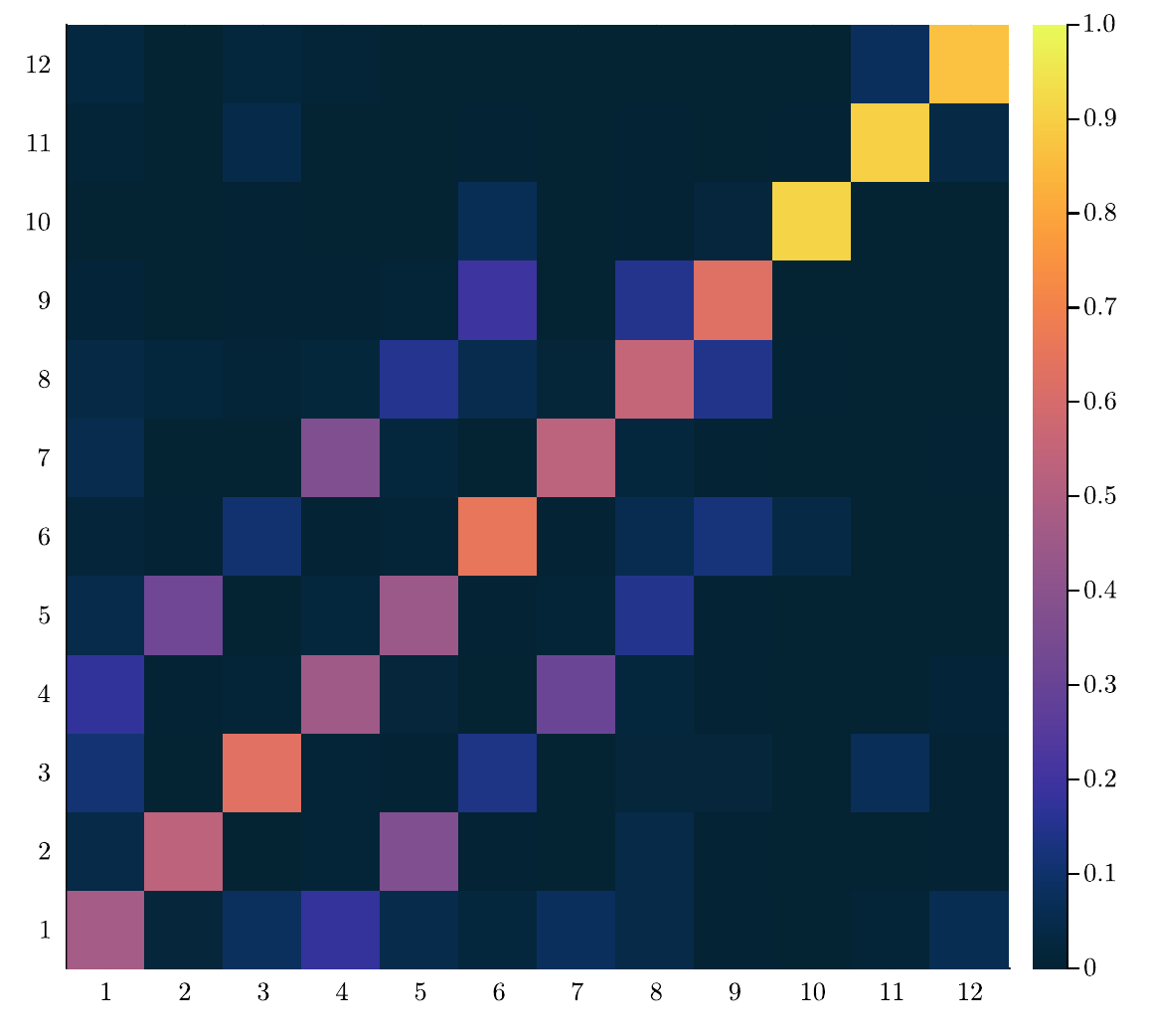} 
            &\includegraphics[width=0.23\linewidth, clip=true, trim=10pt 10pt 70pt 0pt]{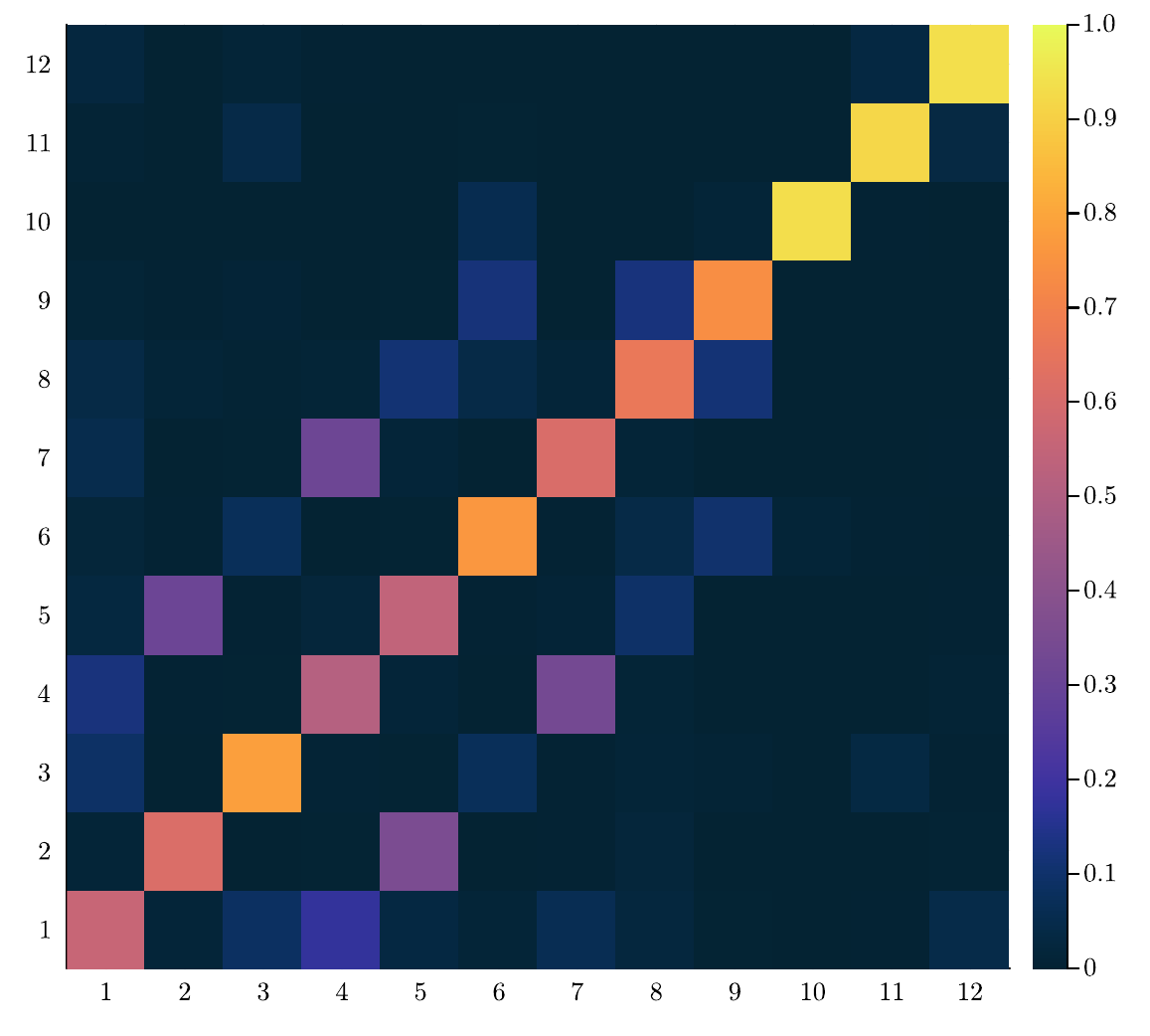}
            &\includegraphics[height=0.237\linewidth, clip=true, trim=10pt 10pt 20pt 5pt]{Images/colorbar_thermal.pdf}\\
            5 training samples 
            & 10 training samples
            & 5 training samples
            & 10 training samples
        \end{tabular}}%
    \caption{
    Confusion maps for \aNRCDT{} in the 3.\ and 4.\ experiment
    in Table~\ref{tab:k-NN_academic}.
    True labels are vertical 
    and the classified labels horizontal.
    }
    \label{fig:conf_maps_k-NN_academic_entire}
\end{figure}

\subsubsection{Nearest Neighbour Classification}

In the final experiments,
regarding academic datasets,
the NT method
is replaced by the nearest neighbour (NN) method
to demonstrate that the classification accuracy 
can be improved,
especially under presence of non-affine distortions
and noise.
The parameter regarding the discretized Radon transform 
and CDT 
remain unchanged,
i.e.,~we use $128$ Radon angles,
$850$ radii,
and $64$ interpolation points.
The generated datasets
consists of $100$ samples 
per each of the 12 classes 
in Figure~\ref{fig:academic_dataset}.
Furthermore,
the employed NN classifiers 
rely on up to 10 random samples per class.
In Table~\ref{tab:k-NN_academic},
the mean NN accuracies
for different distortion settings
are recorded,
where each experiment is repeated 20 times 
with different training samples.
In the ideal setting
and under non-affine deformations, 
rows 1 and 2,
\mNRCDT\ yields perfect classifications,
whereas \aNRCDT\ performs slightly worse.
The parameter ranges 
of the anisotropic scaling and shearing 
in these experiments 
are relatively large,
such that the decrease of the performance 
of \aNRCDT\ is expected.
Under the presence of salt noise,
rows 3 and 4,
\aNRCDT\ significantly outperforms \mNRCDT.
Considering the confusion maps 
of these experiments in Figure~\ref{fig:conf_maps_k-NN_academic_entire},
we notice that 
\aNRCDT\ can clearly distinguish 
the shield symbols 
and the crosses at the top 
but has problems to classify 
the basis (circle, square, triangle) of the symbols.

\subsection{Semi-synthetic Chinese Character Dataset}

To demonstrate the applicability of our approach to
multi-class problems with a large number of classes,
we consider 
the first 1000 classes
of the Chinese character dataset~\cite{Bliem2022}.
For each of these,
we select the first representative as template,
see Figure~\ref{fig:chinese_character} for the first 12 classes,
which is then randomly scaled, rotated, sheared and shifted
to form our semi-synthetic Chinese character datasets.

\begin{figure}[t]
    \resizebox{\linewidth}{!}{%
    \begin{tabular}{c c c c c c c c c c c c}
        \includegraphics[width=0.33\linewidth, clip=true, trim=120pt 20pt 120pt 20pt]{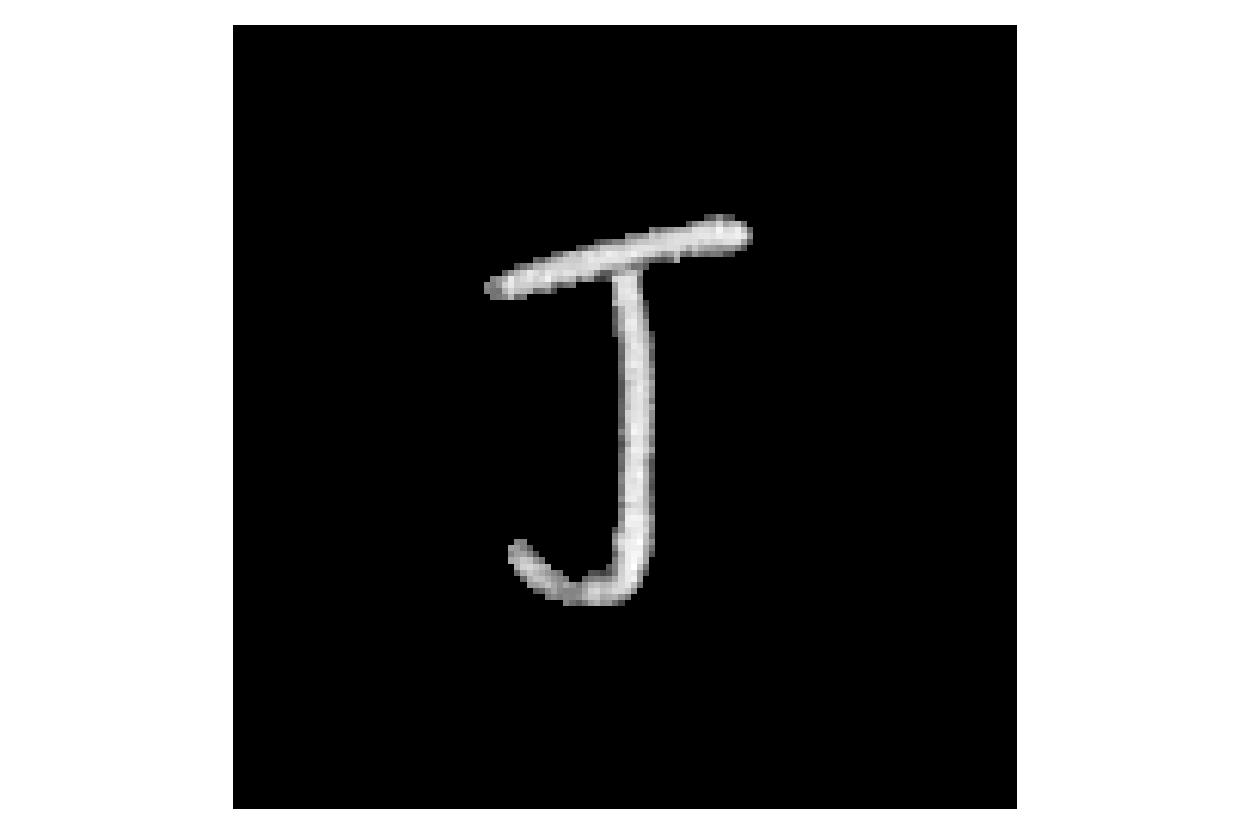}
        &\includegraphics[width=0.33\linewidth, clip=true, trim=120pt 20pt 120pt 20pt]{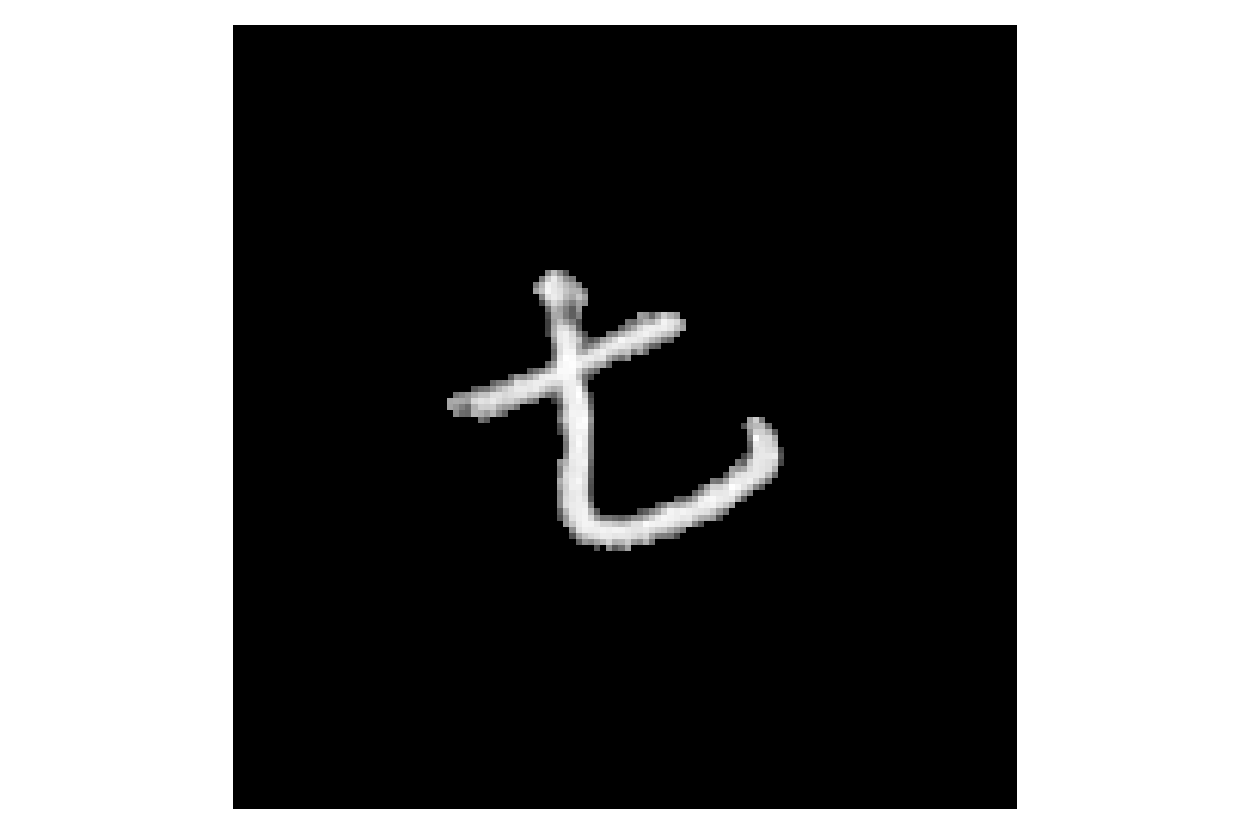}
        &\includegraphics[width=0.33\linewidth, clip=true, trim=120pt 20pt 120pt 20pt]{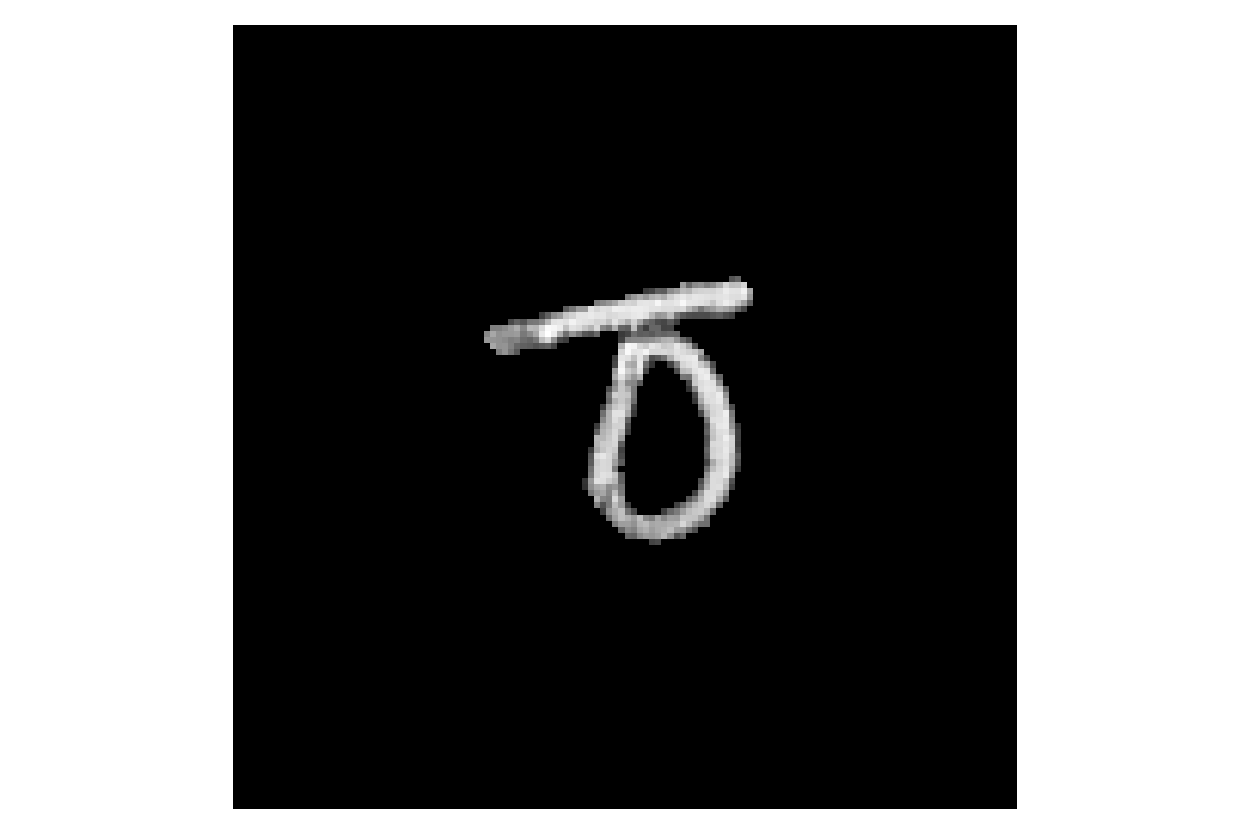}
        &\includegraphics[width=0.33\linewidth, clip=true, trim=120pt 20pt 120pt 20pt]{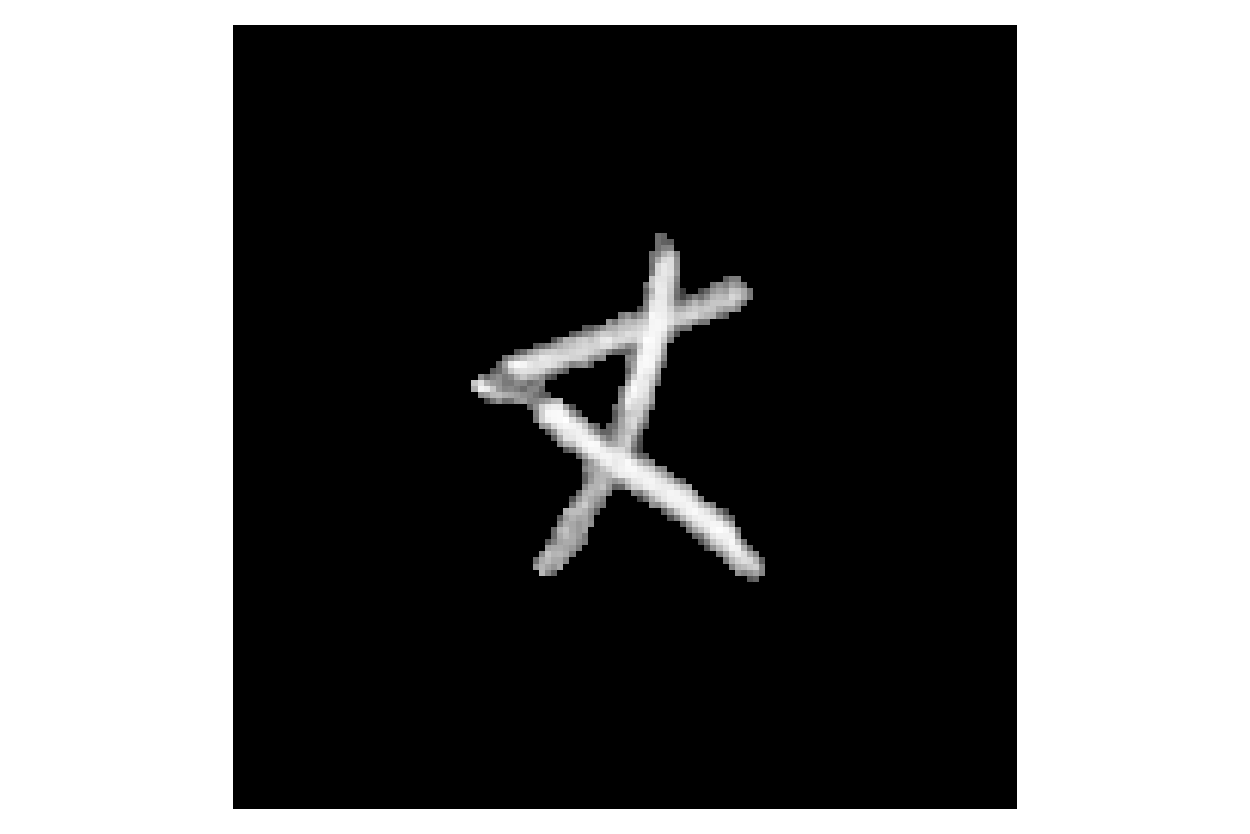}
        &\includegraphics[width=0.33\linewidth, clip=true, trim=120pt 20pt 120pt 20pt]{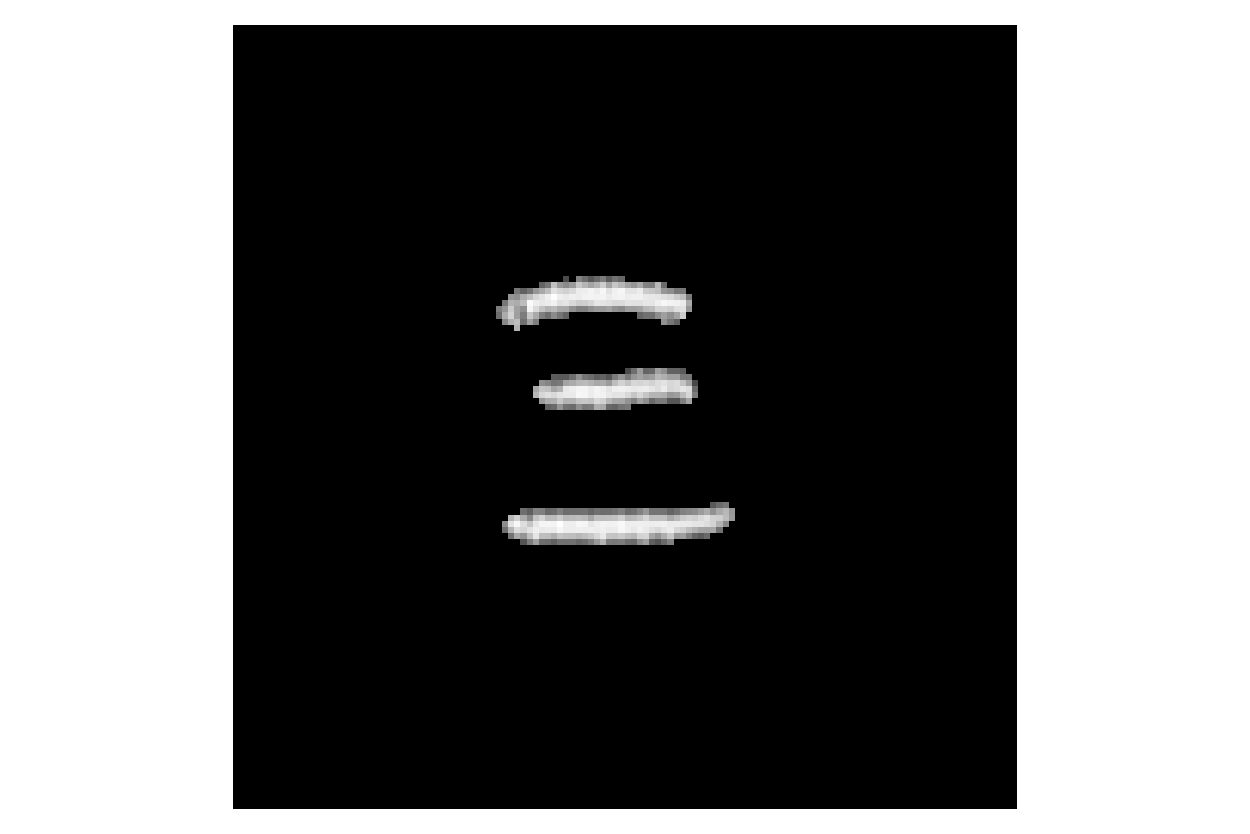}
        &\includegraphics[width=0.33\linewidth, clip=true, trim=120pt 20pt 120pt 20pt]{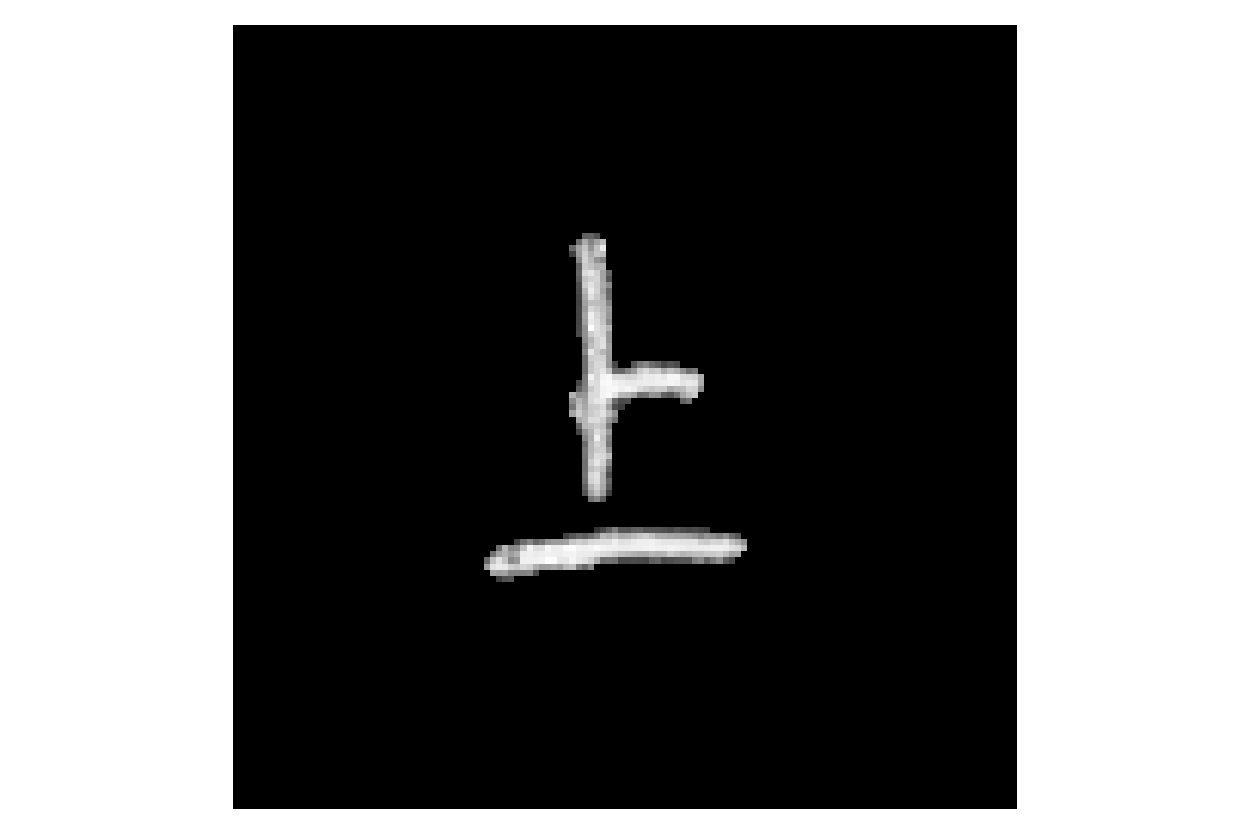}
        &\includegraphics[width=0.33\linewidth, clip=true, trim=120pt 20pt 120pt 20pt]{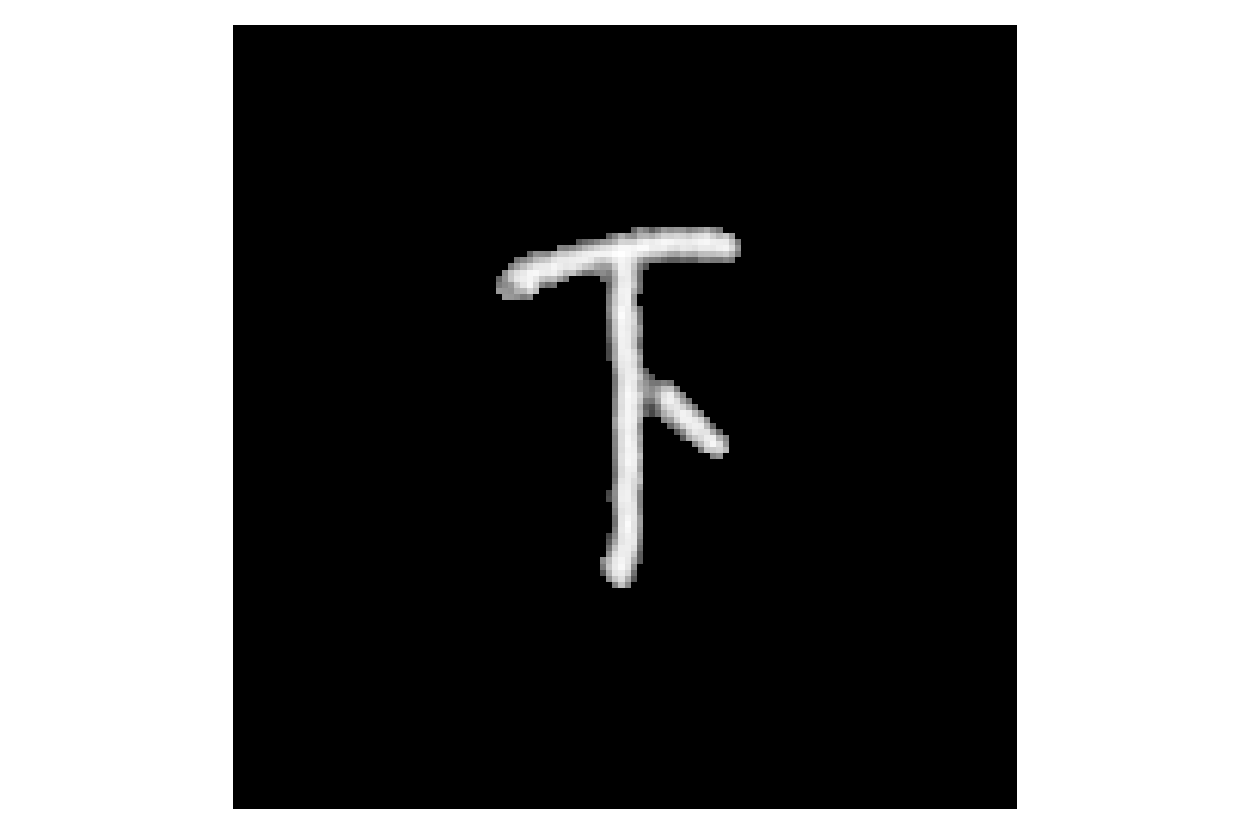}
        &\includegraphics[width=0.33\linewidth, clip=true, trim=120pt 20pt 120pt 20pt]{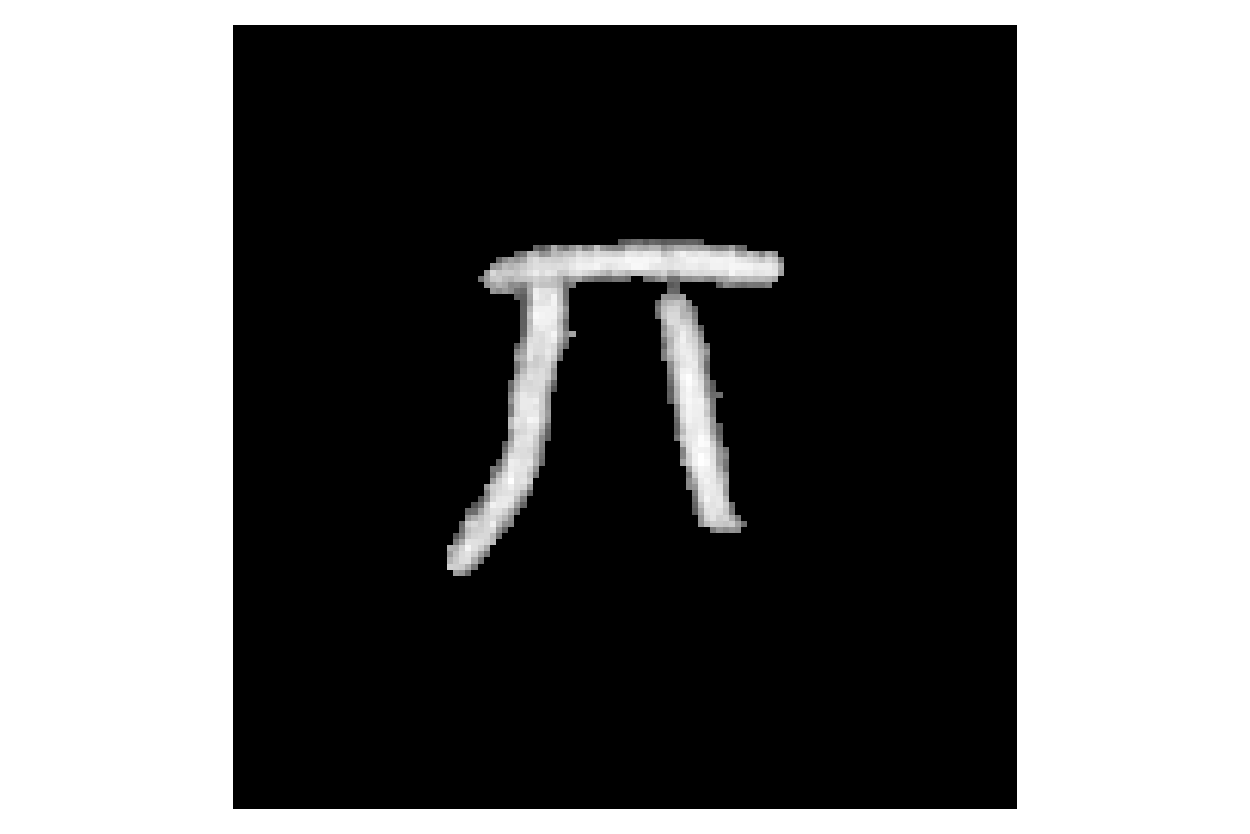}
        &\includegraphics[width=0.33\linewidth, clip=true, trim=120pt 20pt 120pt 20pt]{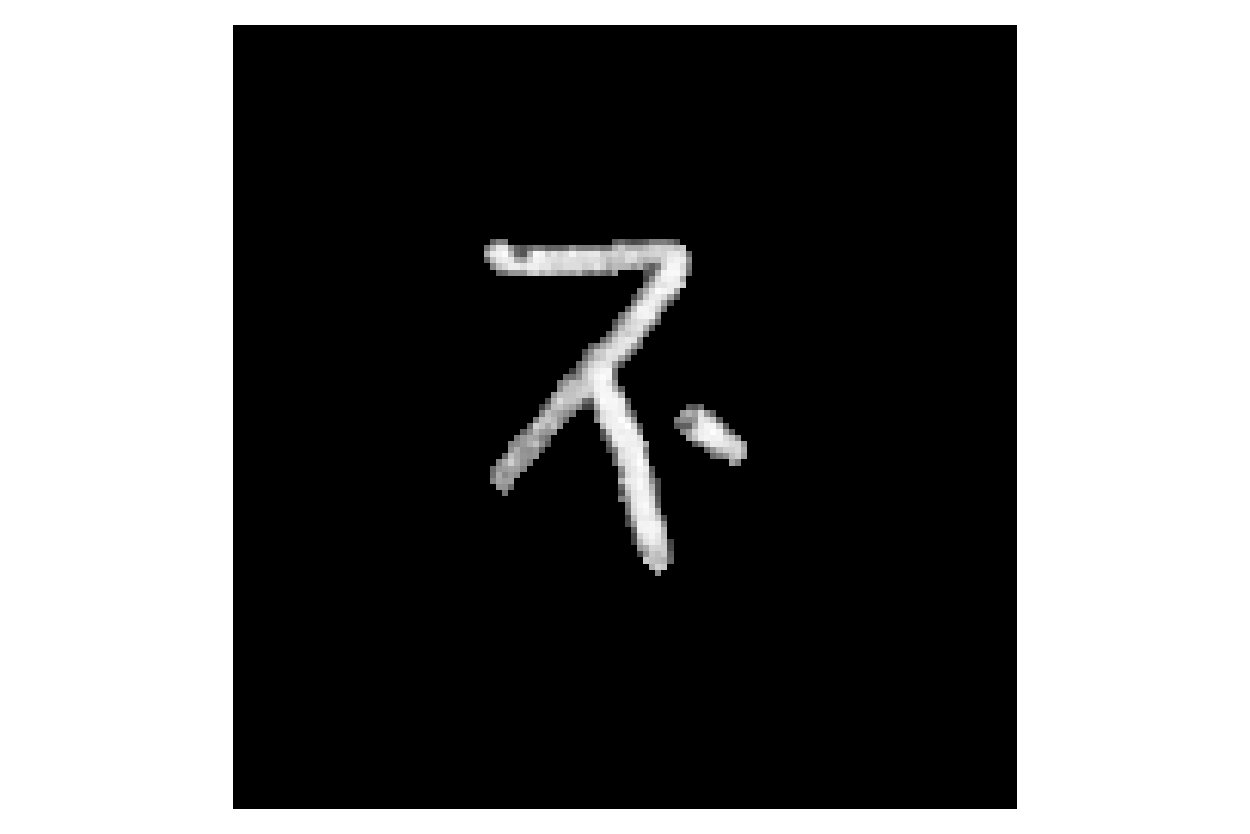}
        &\includegraphics[width=0.33\linewidth, clip=true, trim=120pt 20pt 120pt 20pt]{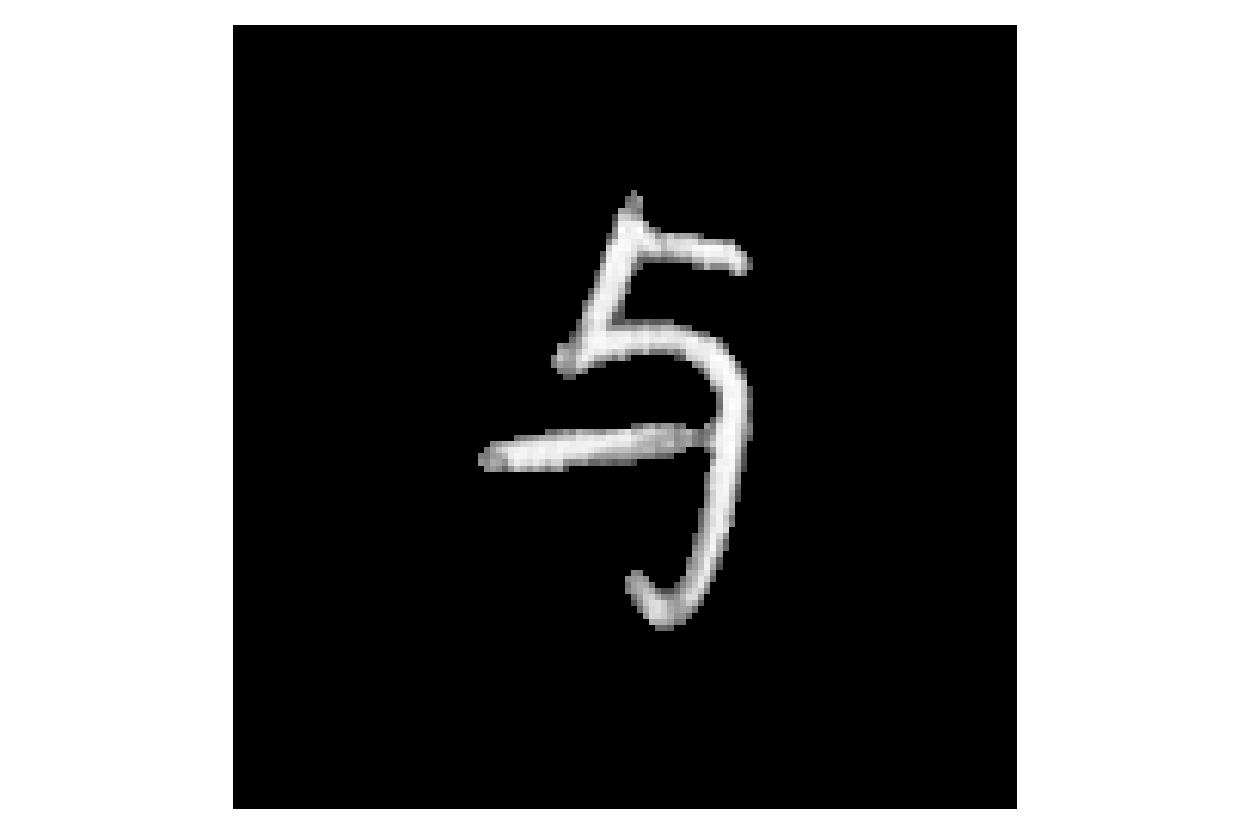}
        &\includegraphics[width=0.33\linewidth, clip=true, trim=120pt 20pt 120pt 20pt]{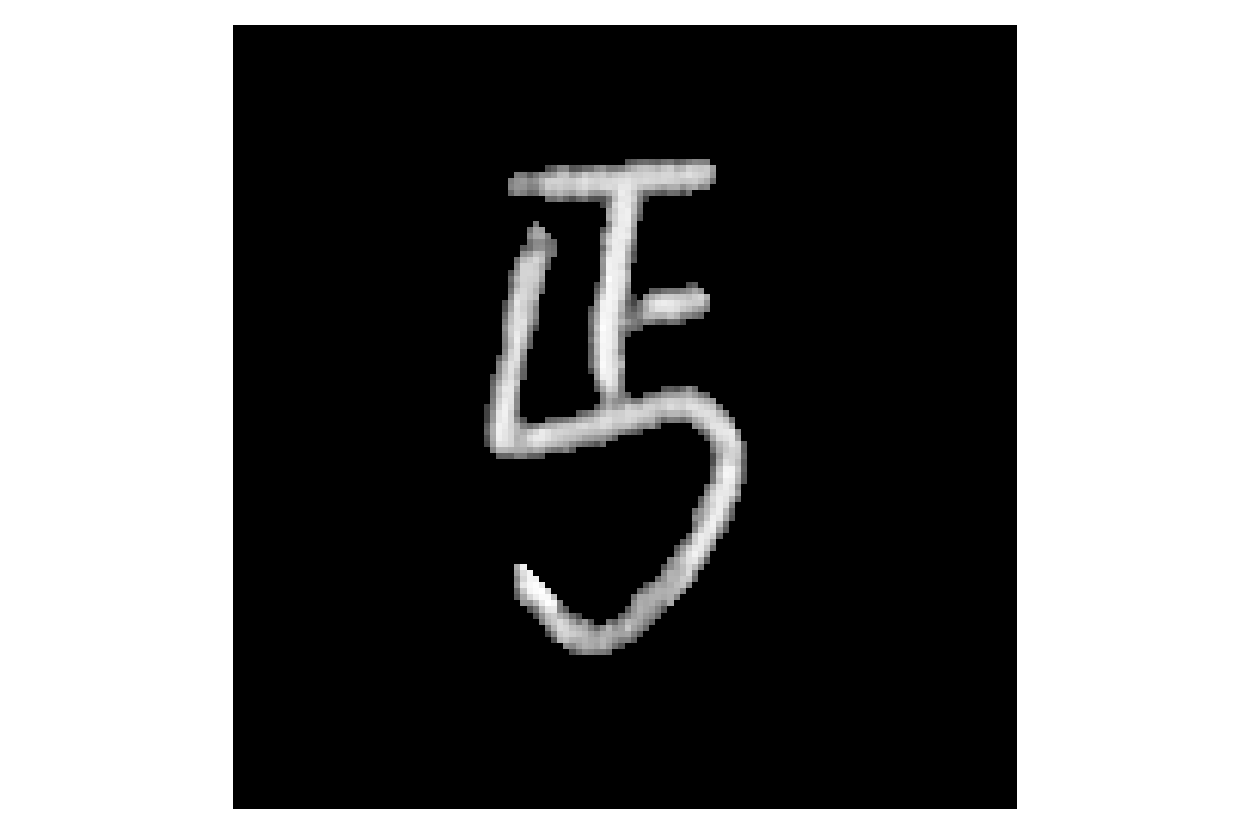}
        &\includegraphics[width=0.33\linewidth, clip=true, trim=120pt 20pt 120pt 20pt]{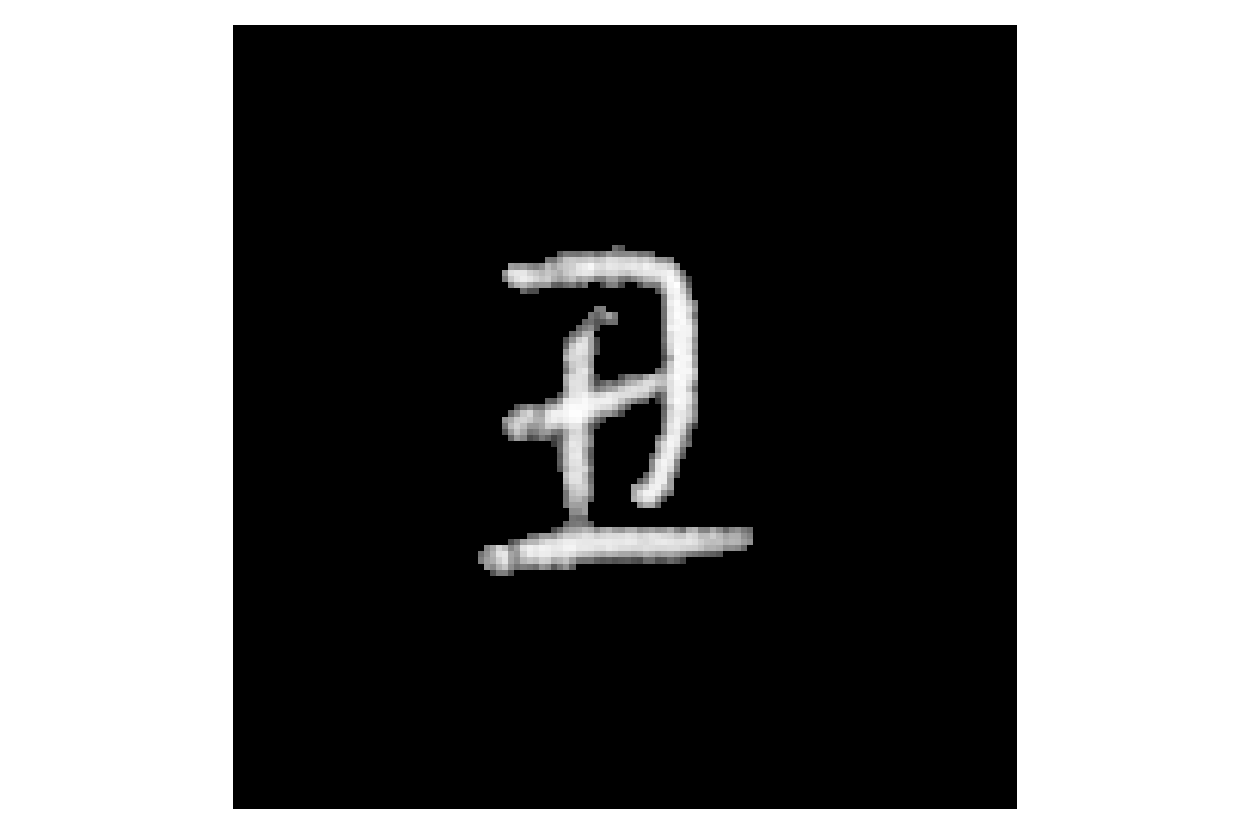}
    \end{tabular}}%
    \caption{
    Chinese character dataset.
    One member of each class is selected as template symbol and
    affinely transformed to create the dataset.
    Here,
    the templates
    of classes 1 to 12 are shown,
    each of size $128{\times}128$ pixels.
    }
    \label{fig:chinese_character}
\end{figure}

\begin{table}[t]
    \caption{
    NT and NN classification accuracies
    for the Chinese character dataset,
    where each class consists of 50 samples,
    generated by random affine transformations with
    scaling in $[0.5,1.0]$,
    shearing in $[-25^\circ, 25^\circ]$,
    rotation angles in $[0^\circ, 360^\circ]$
    and pixel shifts in $[-20,20]$.
    Separately for NT and NN,
    the best result is highlighted.
    }
    \label{tab:chinese}
    \resizebox{\linewidth}{!}{%
    \begin{tabular}{l l  l l l l l l l  l l}
        \toprule
        \multirow{3}{*}{\rotatebox[origin=c]{90}{\scriptsize \# classes}} & method 
        & \multicolumn{7}{l}{NT}
        & \multicolumn{2}{l}{NN} \\
        & & \multicolumn{7}{l}{\# angles} 
        & \multicolumn{2}{l}{\# training samples} \\
        & & 2 & 4 & 8 & 16 & 32 & 64 & 128 
        & 5 & 10
        \\
        \midrule
        \multirow{4}{*}{\rotatebox[origin=c]{90}{\scriptsize 100}}
        & Euclidean 
        & $0.0096$ & & & & & & 
        & $0.0202\pm0.0021$
        & $0.0242\pm0.0022$
        \\
        & R-CDT
        & $0.0160$ & $0.0180$ & $0.0172$ & $0.0164$ & $0.0168$& $0.0170$ & $0.0170$
        & $0.0188\pm0.0020$
        & $0.0217\pm0.0021$
        \\
        & \mNRCDT{}
        & \bm{$0.0924$} & \bm{$0.1038$} & \bm{$0.3124$} & \bm{$0.8422$} & \bm{$0.9836$} & \bm{$1.0000$} & \bm{$1.0000$}
        & \bm{$1.0000\pm0.0000$}
        & \bm{$1.0000\pm0.0000$}
        \\
        & \aNRCDT{}
        & $0.0810$ & $0.0172$ & $0.1860$ & $0.4950$ & $0.6486$ & $0.6814$ & $0.6852$ 
        & $0.8345\pm0.0097$
        & $0.9139\pm0.0062$
        \\
        \midrule
        \multirow{2}{*}{\rotatebox[origin=c]{90}{\scriptsize 1000}}
        & \mNRCDT{}
        & \bm{$0.0848$} & \bm{$0.0620$} & \bm{$0.2178$} & \bm{$0.7368$} & \bm{$0.9791$} & \bm{$0.9975$} & \bm{$0.9981$} 
        & \bm{$0.9987\pm0.0001$}
        & \bm{$0.9987\pm0.0001$}
        \\
        & \aNRCDT{}
        & $0.0629$ & $0.0362$ & $0.0751$ & $0.3475$ & $0.5554$ & $0.6030$ & $0.6104$  
        & $0.7651\pm0.0026$
        & $0.8631\pm0.0018$
        \\
        \bottomrule
    \end{tabular}}
\end{table}

To start with,
we restrict ourselves to
the leading $100$ classes
with $50$ samples per class
and compare \mNRCDT{} and \aNRCDT{}
with the Euclidean and R-CDT representations.
To this end,
we again use the NT approach
as well as the NN method
with $5$ or $10$ training samples,
respectively.
For the discretization of the
Radon transform and CDT 
we use $128$ Radon angles,
$850$ radii,
and $64$ interpolation points.
The classification accuracies are reported
in Table~\ref{tab:chinese} (top).
For both NT and NN,
we see that \mNRCDT{} and \aNRCDT{}
clearly outperform the Euclidean and R-CDT approaches,
which only perform at the level of random guessing.
As expected by our theory,
we observe perfect NT and NN classification
using \mNRCDT{} with sufficiently many angles.
For \aNRCDT{},
the performance of the NT classification is worse,
which is also expected due to the
application of rather severe scaling and shearing.
However,
the classification accuracy is significantly
improved by NN
and we achieve nearly perfect results
with increasing number of training samples.

Since the Euclidean and R-CDT approach cannot
successfully classify the first $100$ classes,
we only consider our \mNRCDT{} and \aNRCDT{} representations
when dealing with $1000$ Chinese characters.
The results are
shown in Table~\ref{tab:chinese} (bottom).
Again,
\mNRCDT{} yields nearly perfect results,
while the performance of \aNRCDT{}
is improved by NN with increasing
number of training samples.
Hence, all in all,
our numerical observations
reflect our theoretical results
also in the challenging case of
a tremendous number of classes.

\subsection{LinMNIST Dataset}

For a more realistic scenario,
we finally consider the LinMNIST dataset~\cite{Beckmann2024}
consisting of affinely transformed MNIST digits~\cite{Deng2012},
cf.\ Figure~\ref{fig:linmnist}.
More precisely,
this dataset is generated by
selecting the first $500$ samples
of each of the ten MNIST classes
and, thereon, applying a random
affine transformation.
In this way,
we combine our theoretically inspired setting
of affinely transformed classes
with the variety in real-world datasets.
To account for this,
we change the NT and NN approach to
$k$-NN classification
and vary the value of $k$
as well as the number of
training samples.
Again, we compare
\mNRCDT{} and \aNRCDT{}
with the Euclidean and R-CDT representations.
For the discretization of the
Radon transform and CDT 
we use $128$ Radon angles,
$300$ radii,
and $64$ interpolation points.

\begin{figure}[t]
    \resizebox{\linewidth}{!}{%
    \begin{tabular}{c c c c c c c c c c}
        \includegraphics[width=0.2\linewidth, clip=true, trim=40pt 40pt 20pt 20pt]{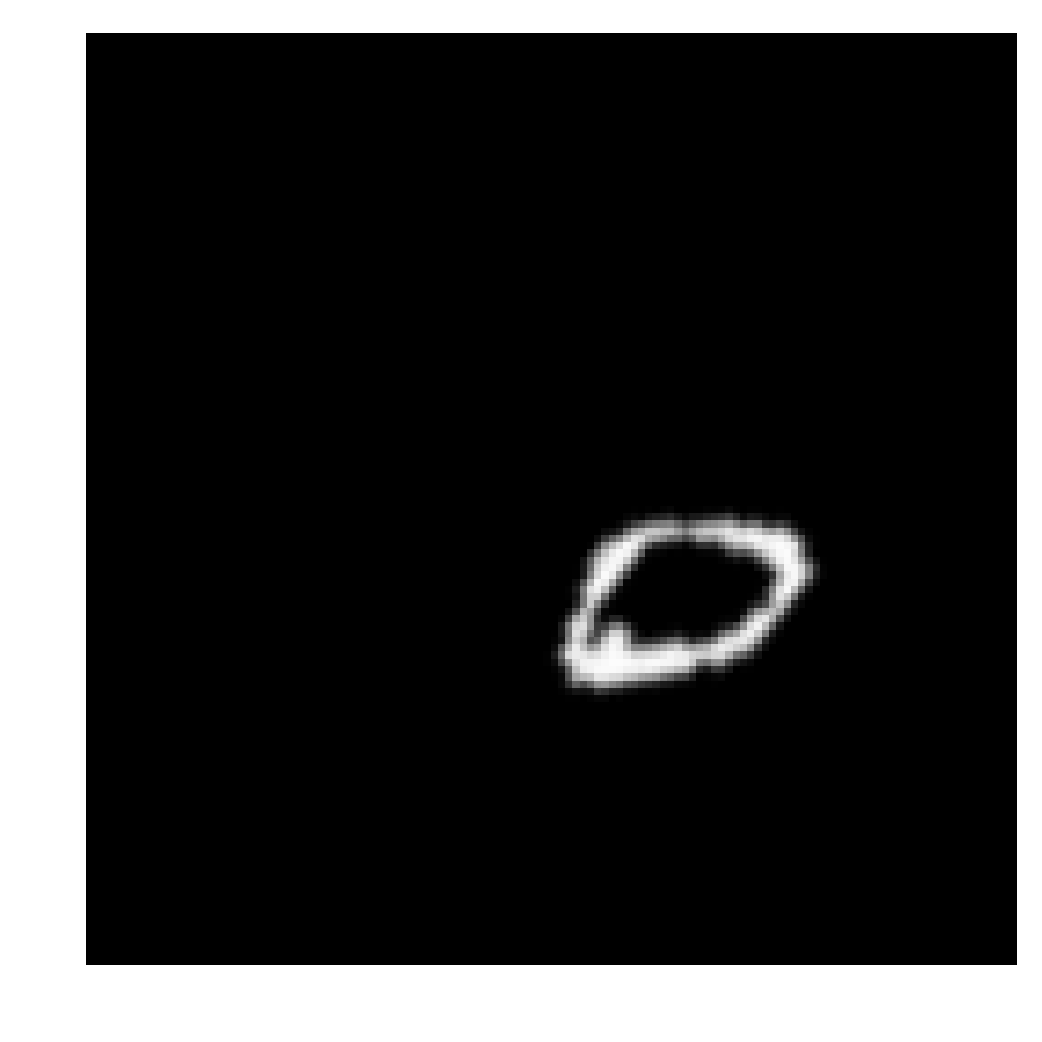}
        & \includegraphics[width=0.2\linewidth, clip=true, trim=40pt 40pt 20pt 20pt]{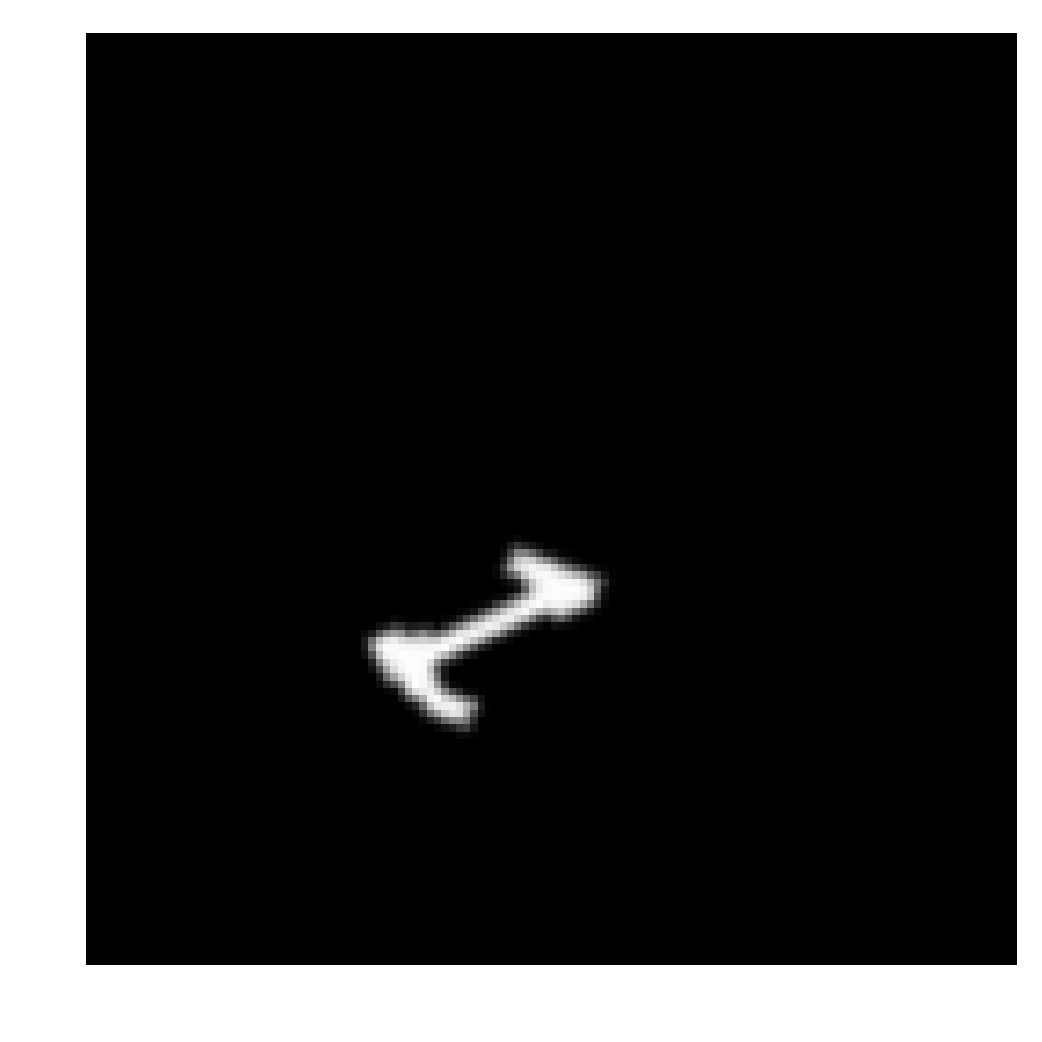}
        & \includegraphics[width=0.2\linewidth, clip=true, trim=40pt 40pt 20pt 20pt]{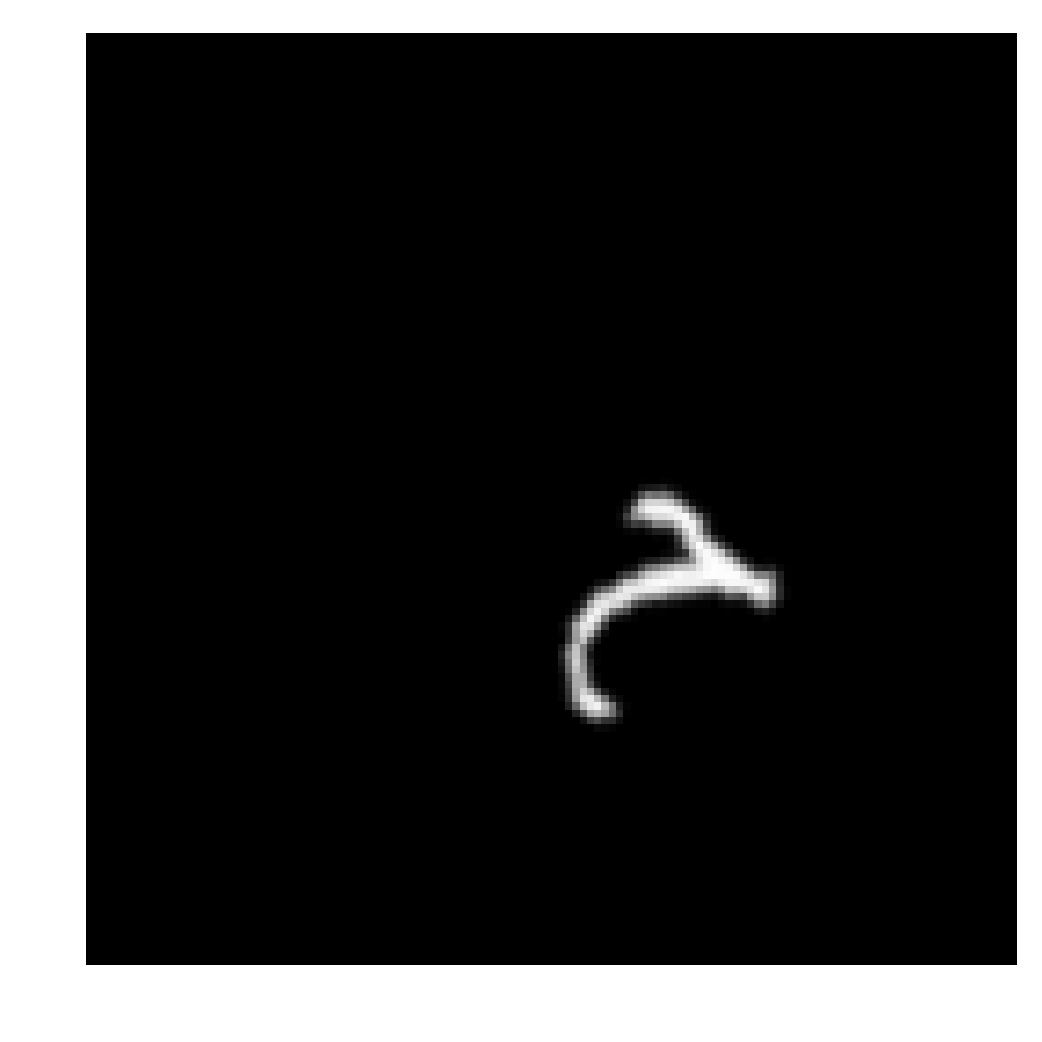}
        & \includegraphics[width=0.2\linewidth, clip=true, trim=40pt 40pt 20pt 20pt]{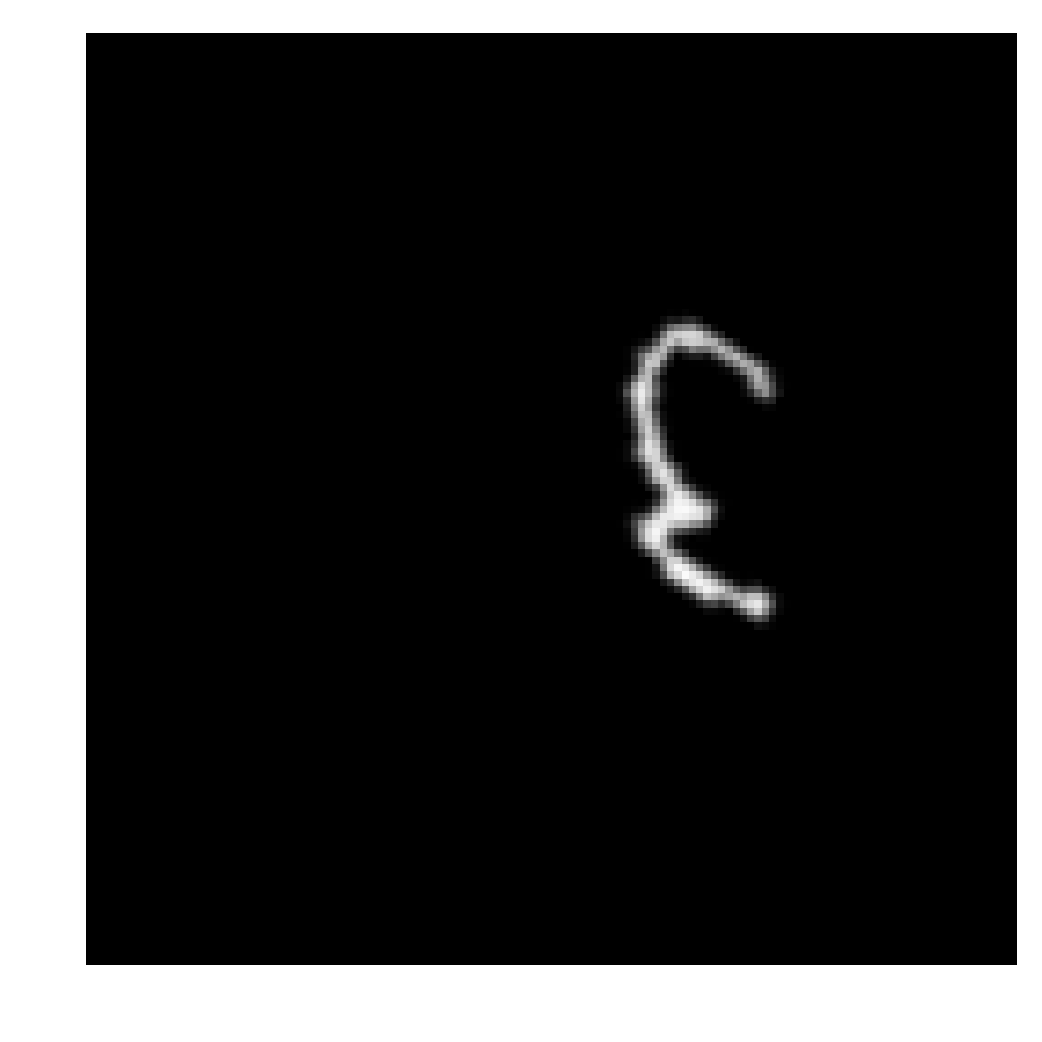}
        & \includegraphics[width=0.2\linewidth, clip=true, trim=40pt 40pt 20pt 20pt]{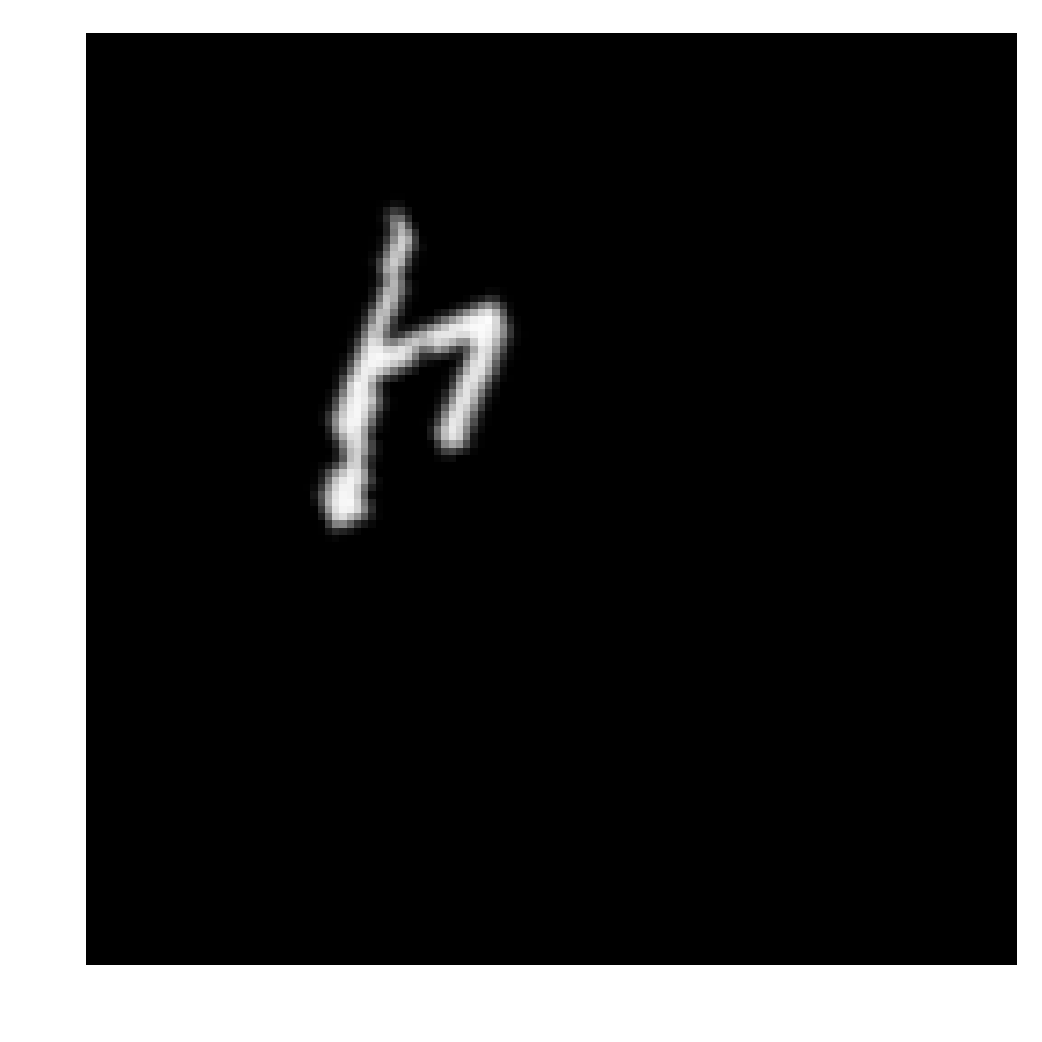}
        & \includegraphics[width=0.2\linewidth, clip=true, trim=40pt 40pt 20pt 20pt]{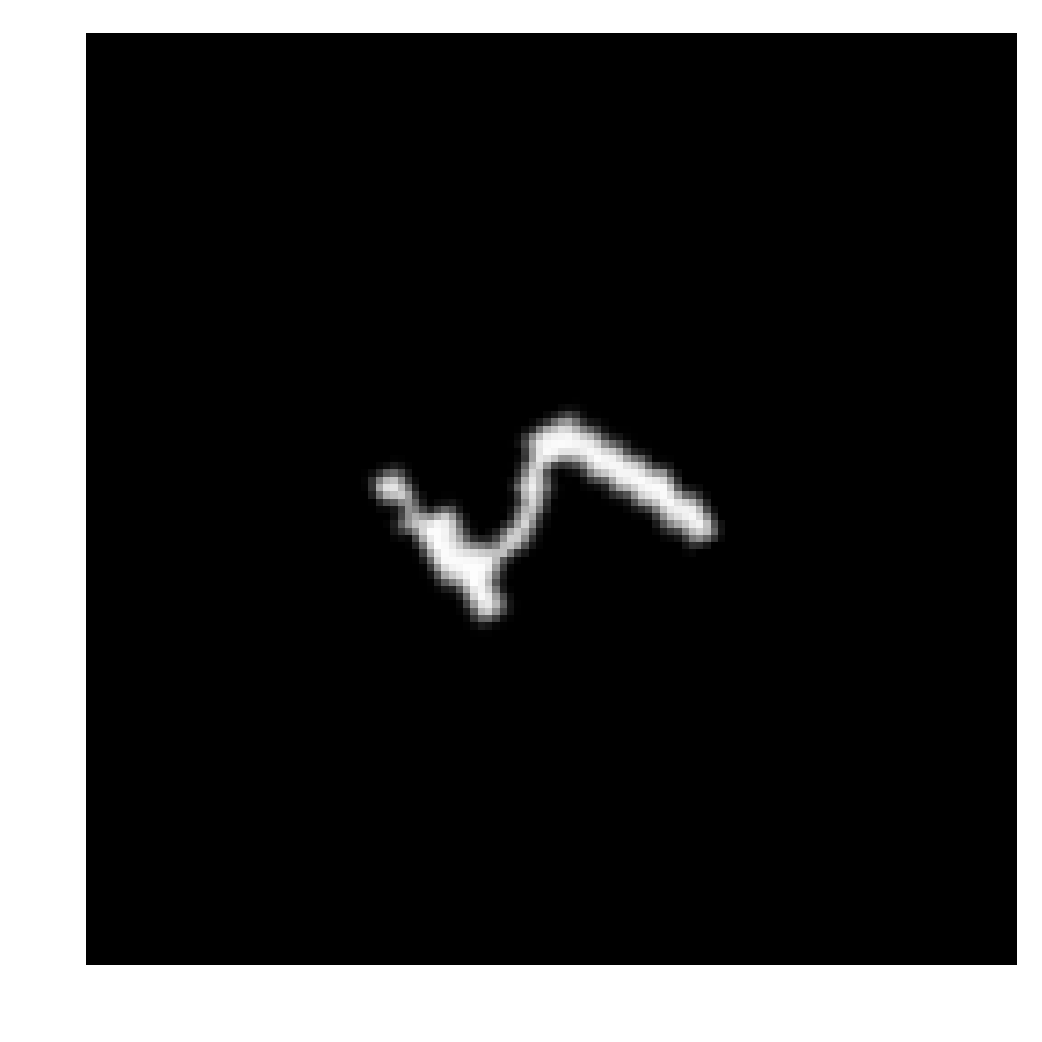}
        & \includegraphics[width=0.2\linewidth, clip=true, trim=40pt 40pt 20pt 20pt]{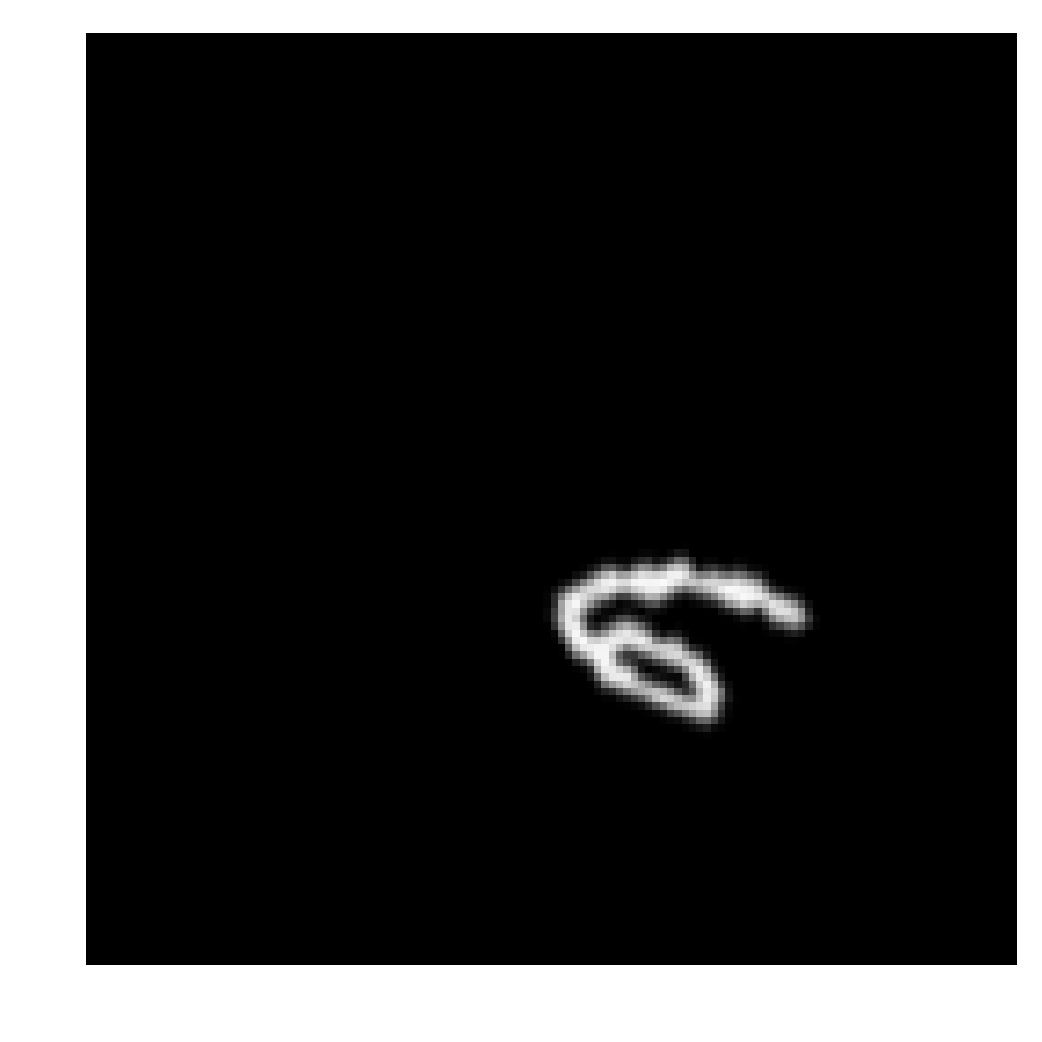}
        & \includegraphics[width=0.2\linewidth, clip=true, trim=40pt 40pt 20pt 20pt]{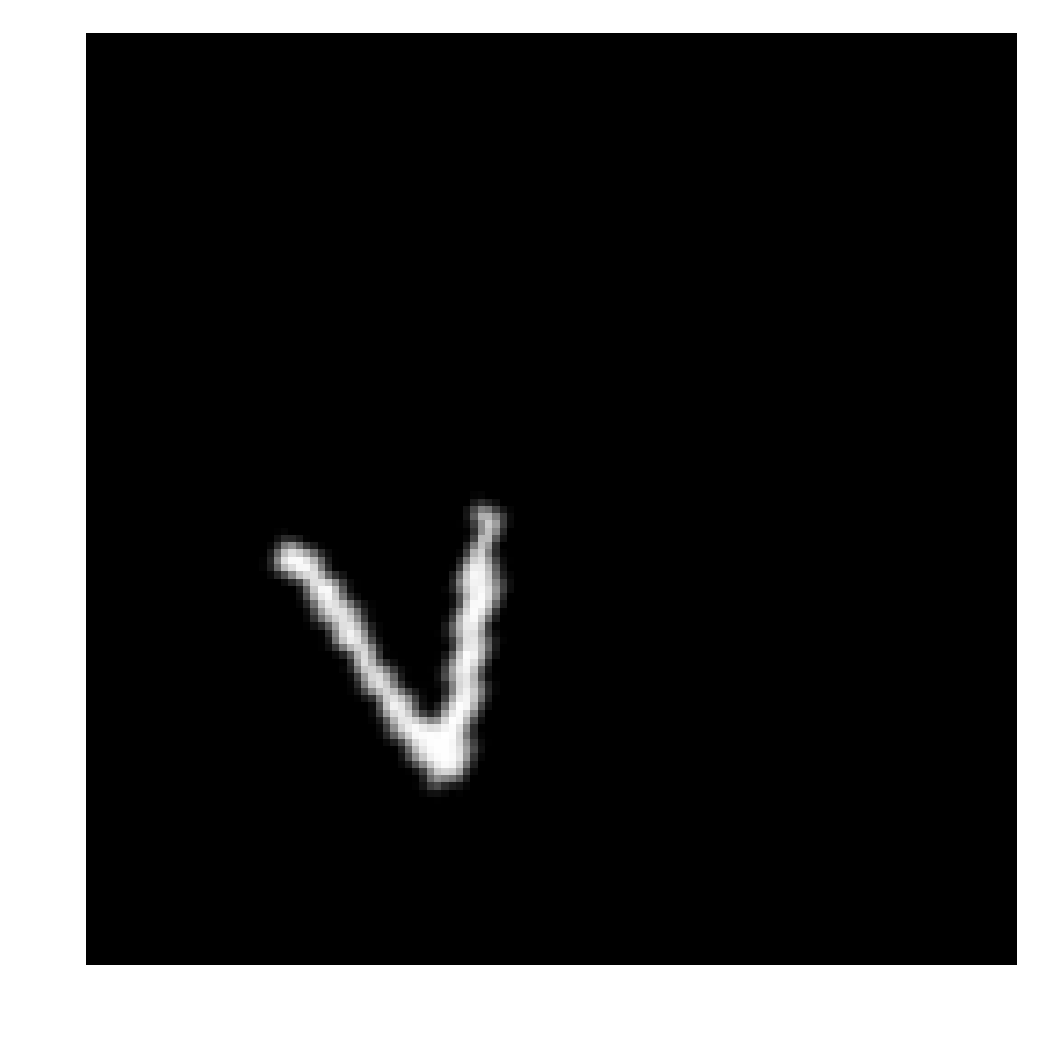}
        & \includegraphics[width=0.2\linewidth, clip=true, trim=40pt 40pt 20pt 20pt]{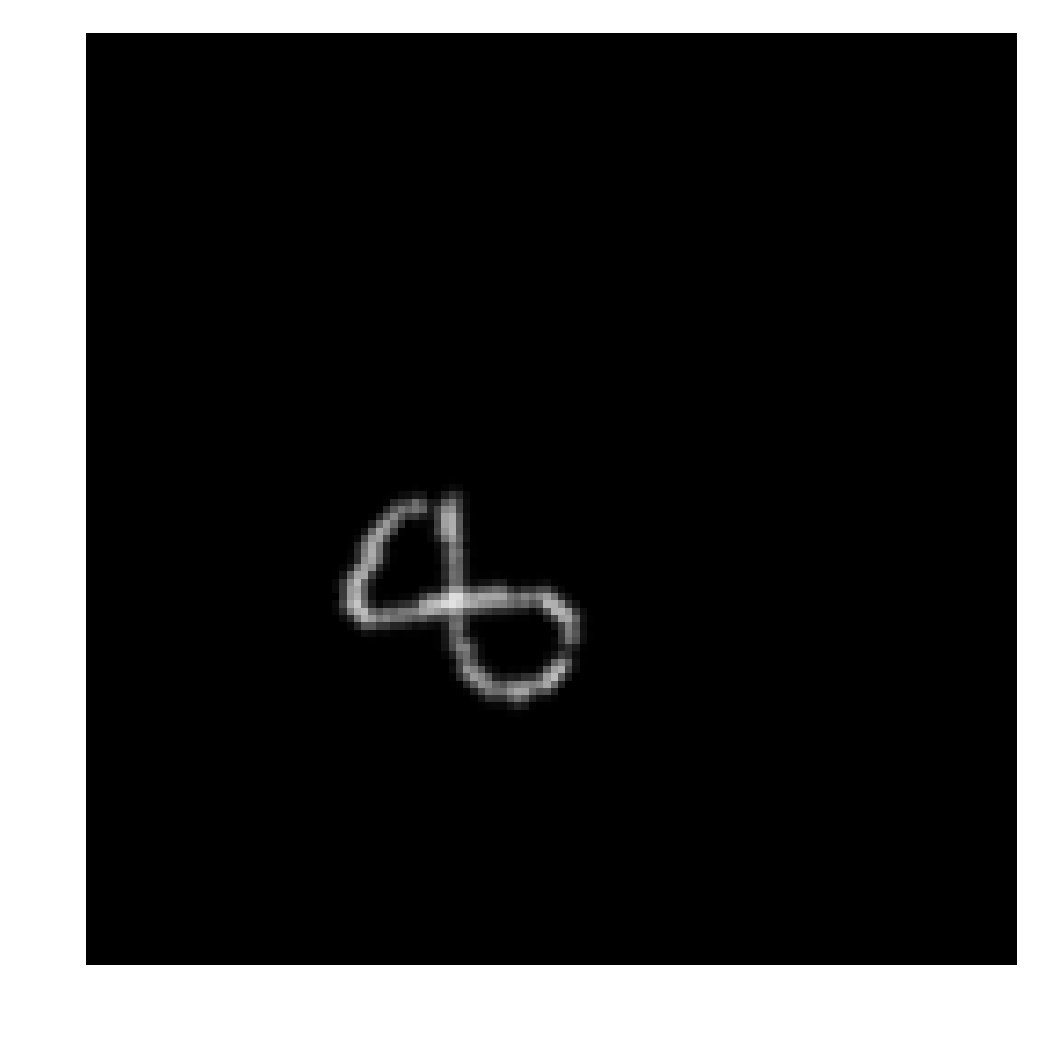}
        & \includegraphics[width=0.2\linewidth, clip=true, trim=40pt 40pt 20pt 20pt]{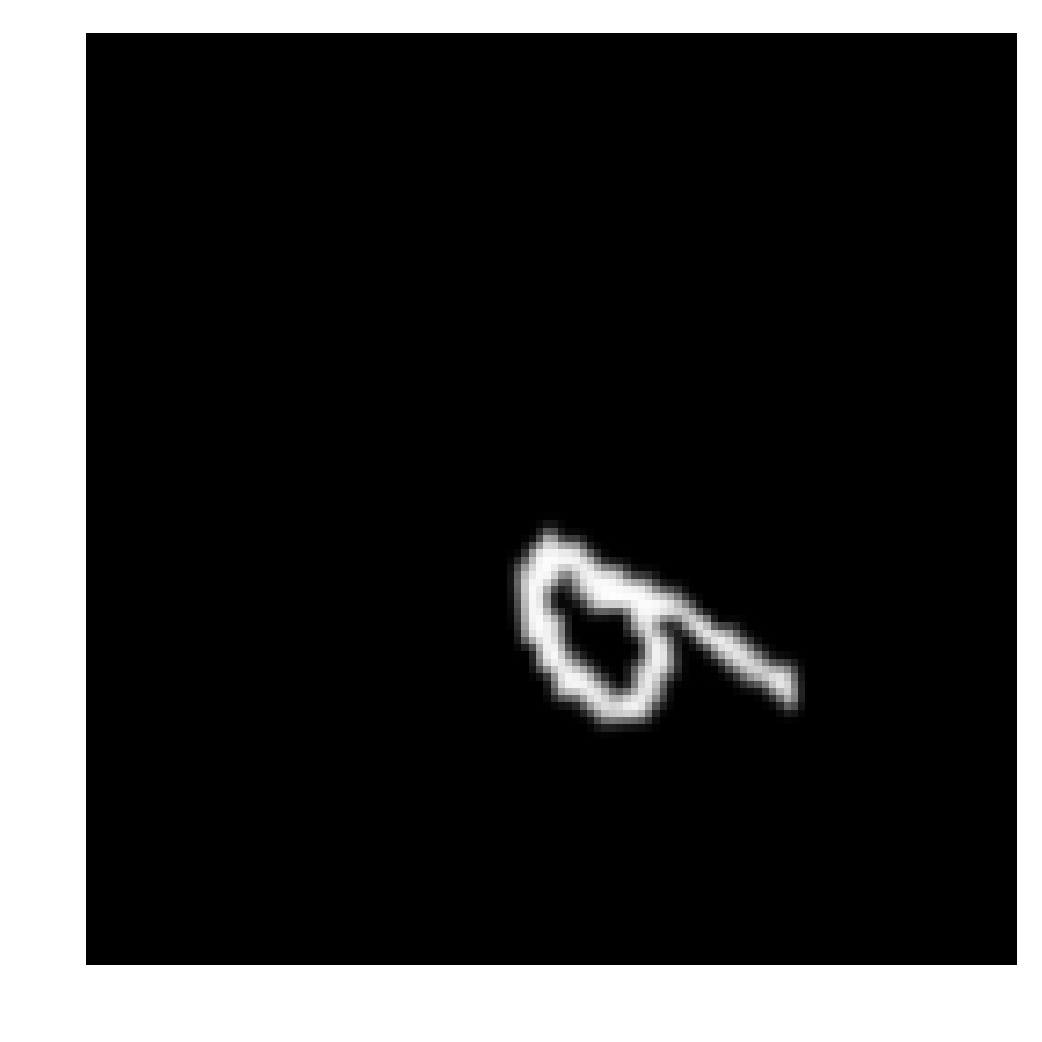} \\
        \includegraphics[width=0.2\linewidth, clip=true, trim=40pt 40pt 20pt 20pt]{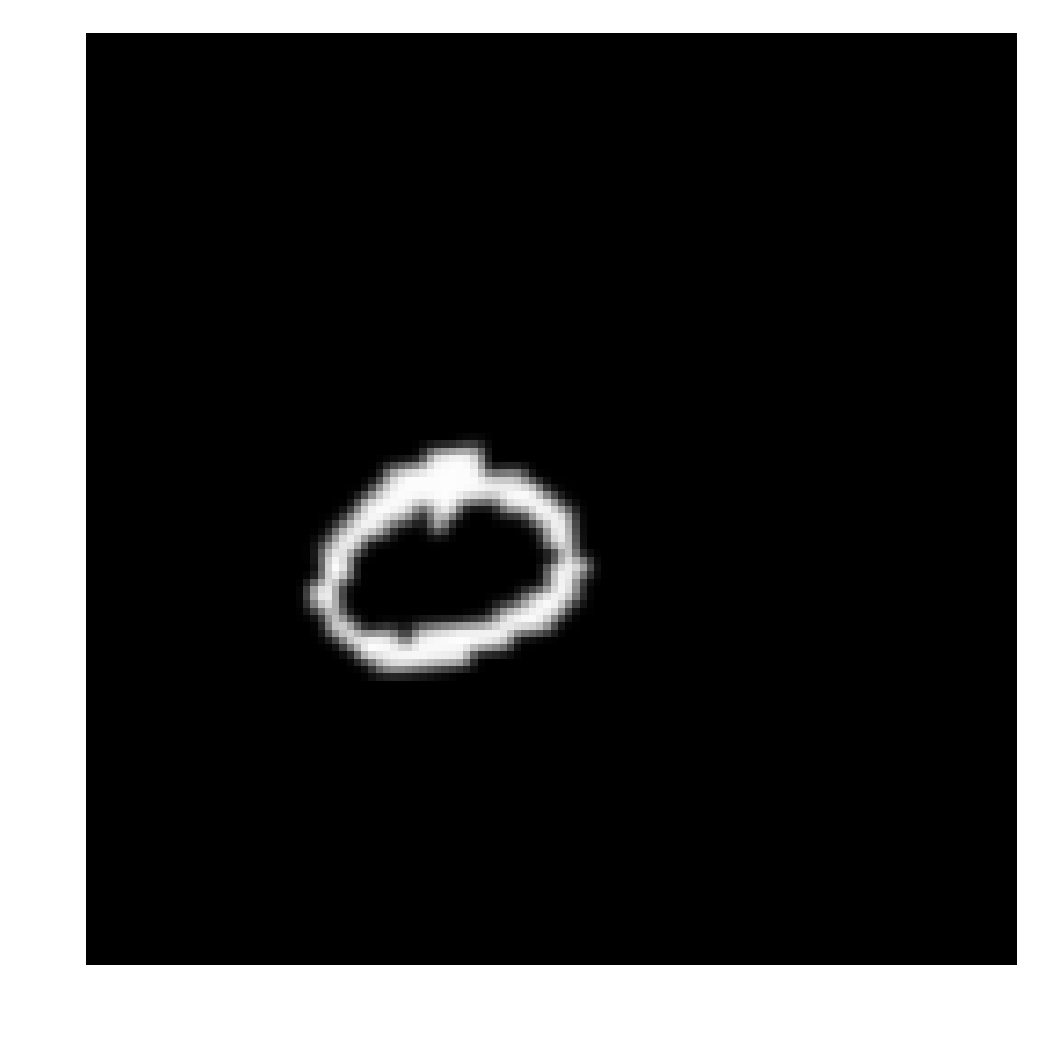}
        & \includegraphics[width=0.2\linewidth, clip=true, trim=40pt 40pt 20pt 20pt]{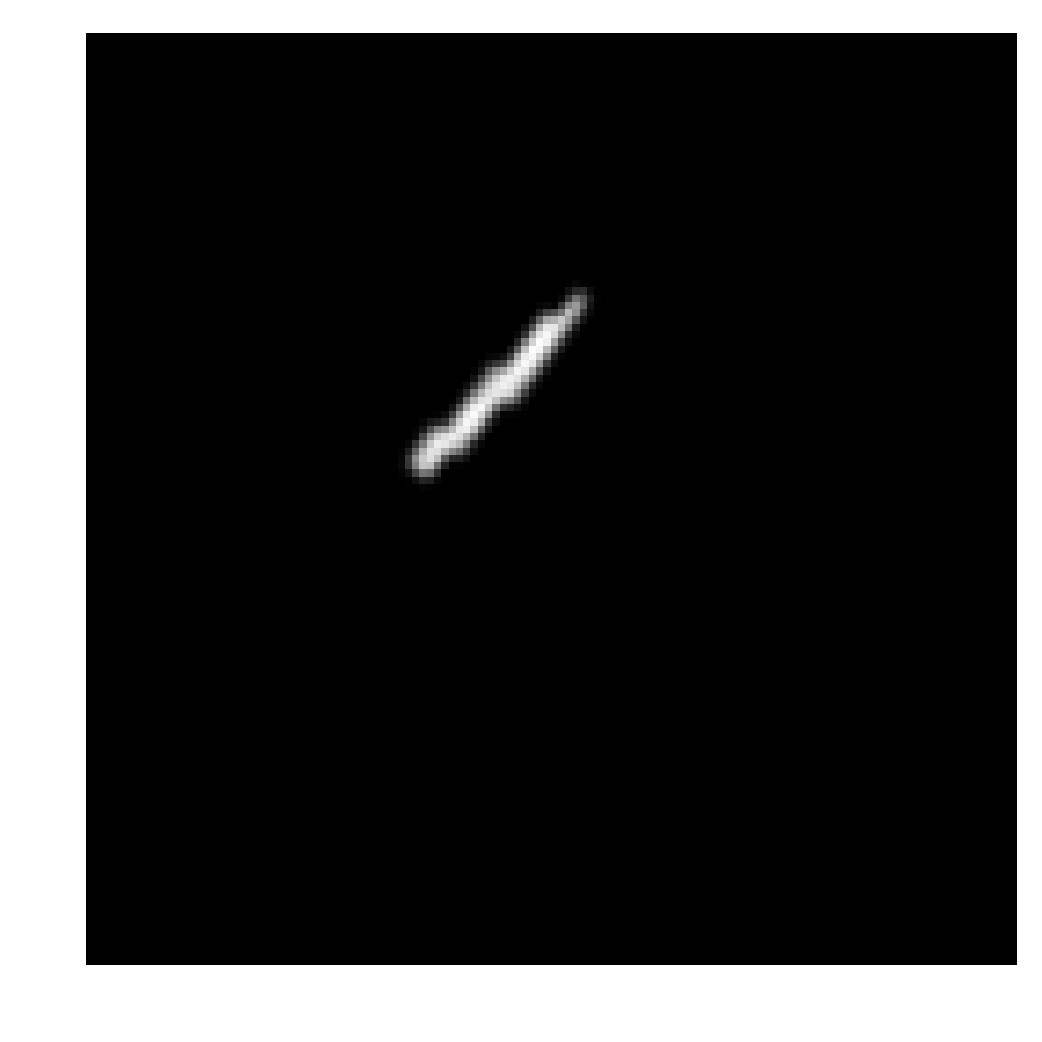}
        & \includegraphics[width=0.2\linewidth, clip=true, trim=40pt 40pt 20pt 20pt]{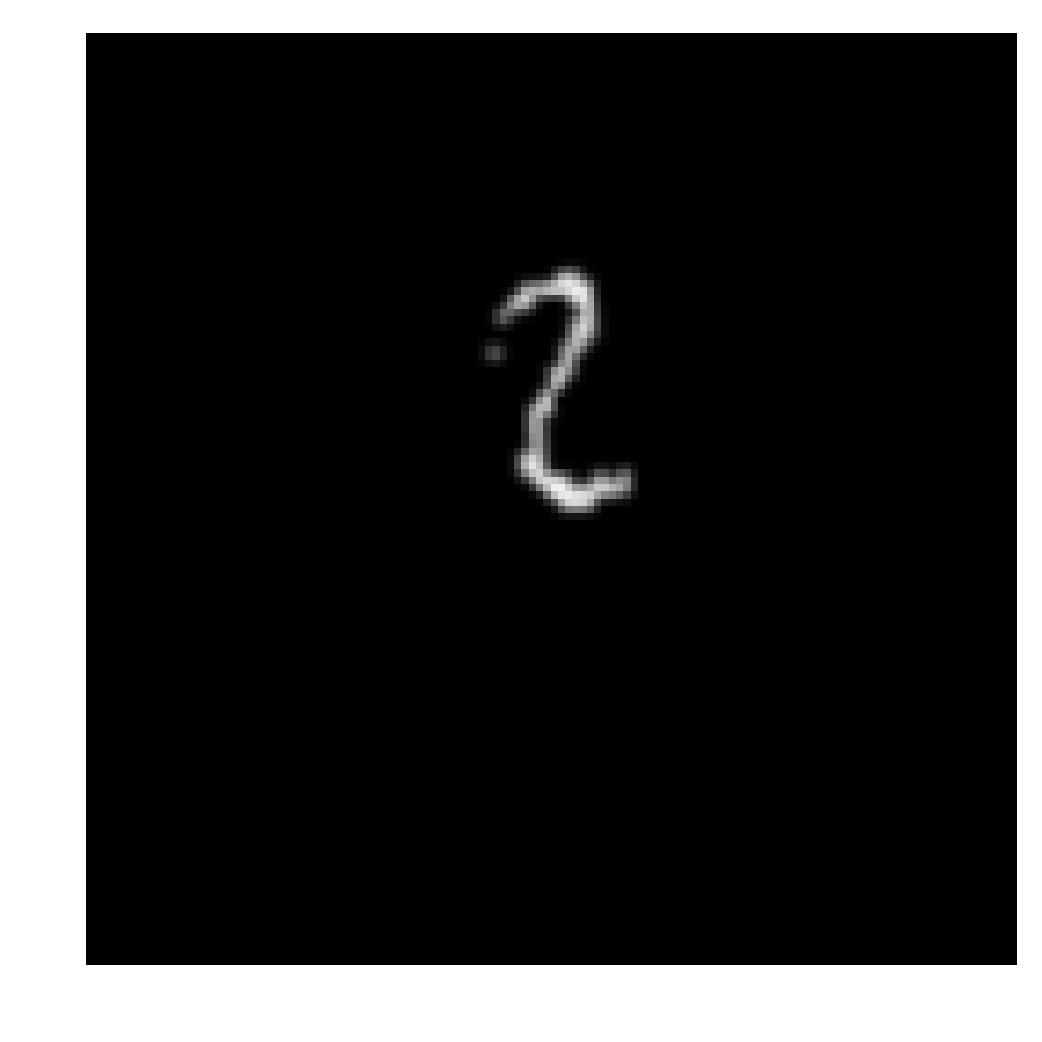}
        & \includegraphics[width=0.2\linewidth, clip=true, trim=40pt 40pt 20pt 20pt]{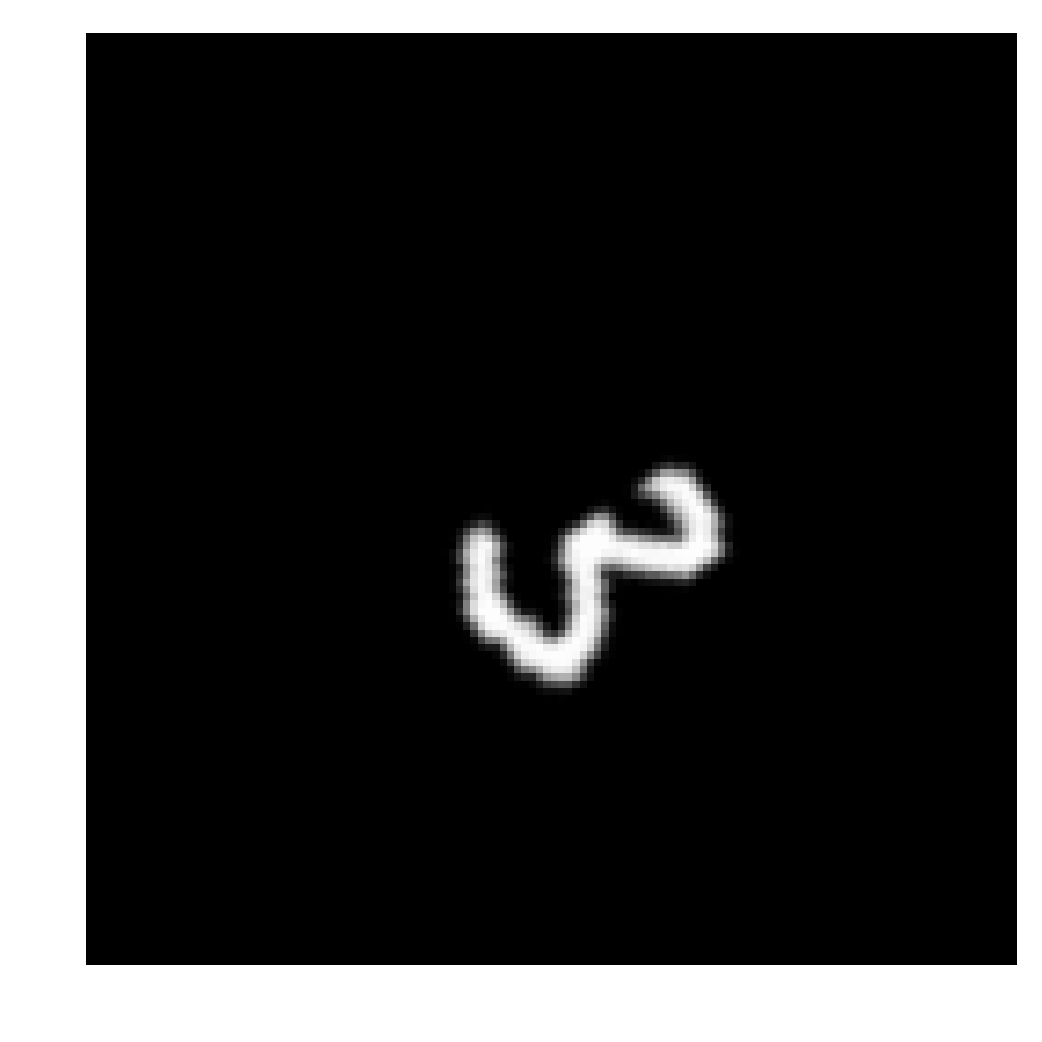}
        & \includegraphics[width=0.2\linewidth, clip=true, trim=40pt 40pt 20pt 20pt]{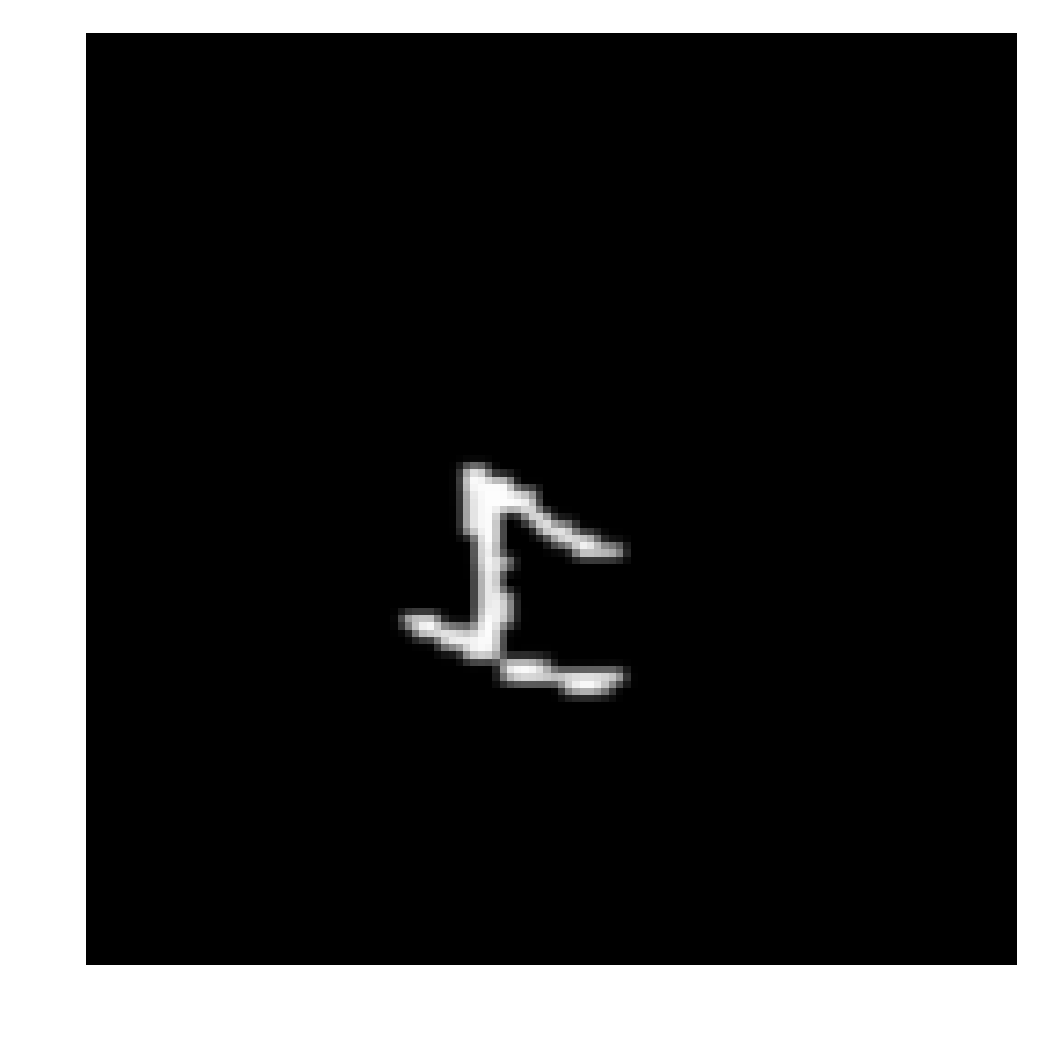}
        & \includegraphics[width=0.2\linewidth, clip=true, trim=40pt 40pt 20pt 20pt]{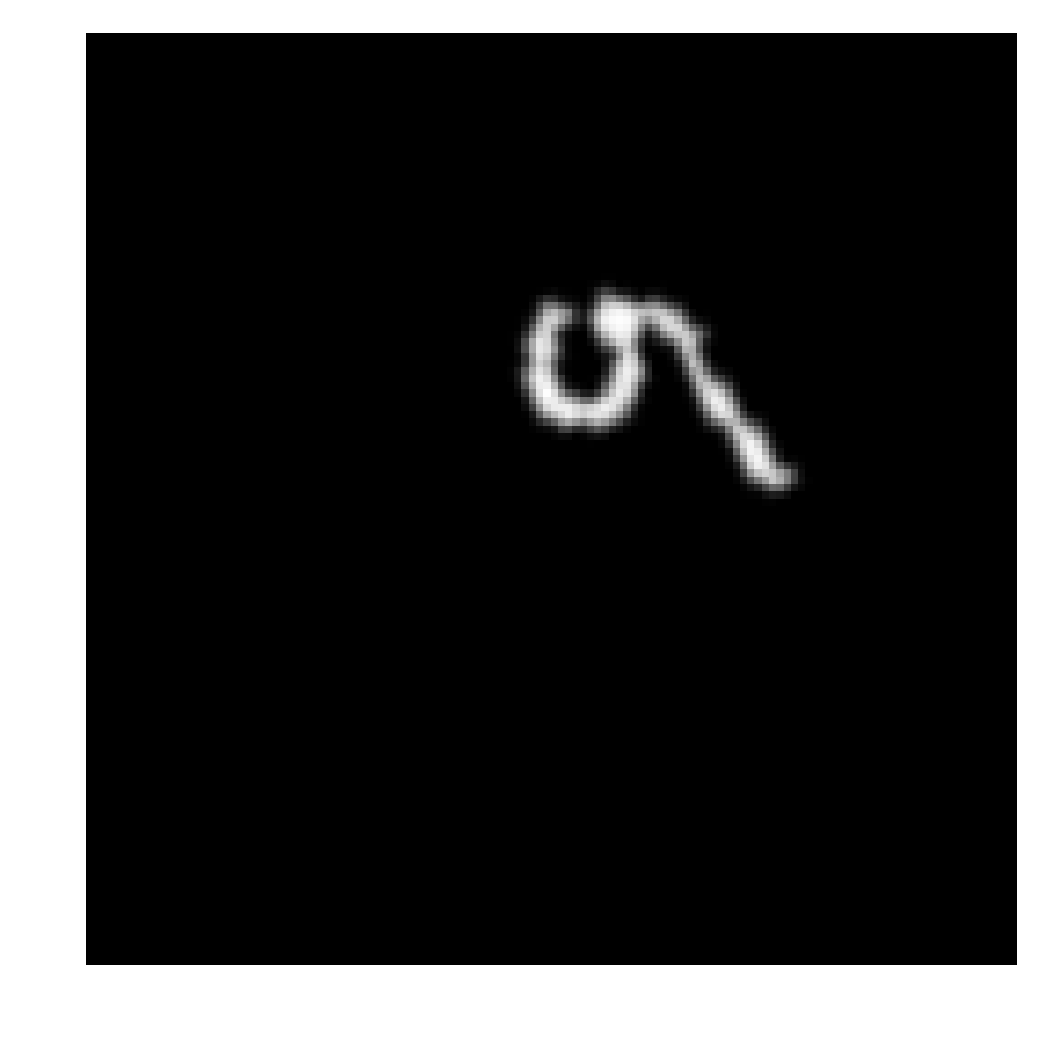}
        & \includegraphics[width=0.2\linewidth, clip=true, trim=40pt 40pt 20pt 20pt]{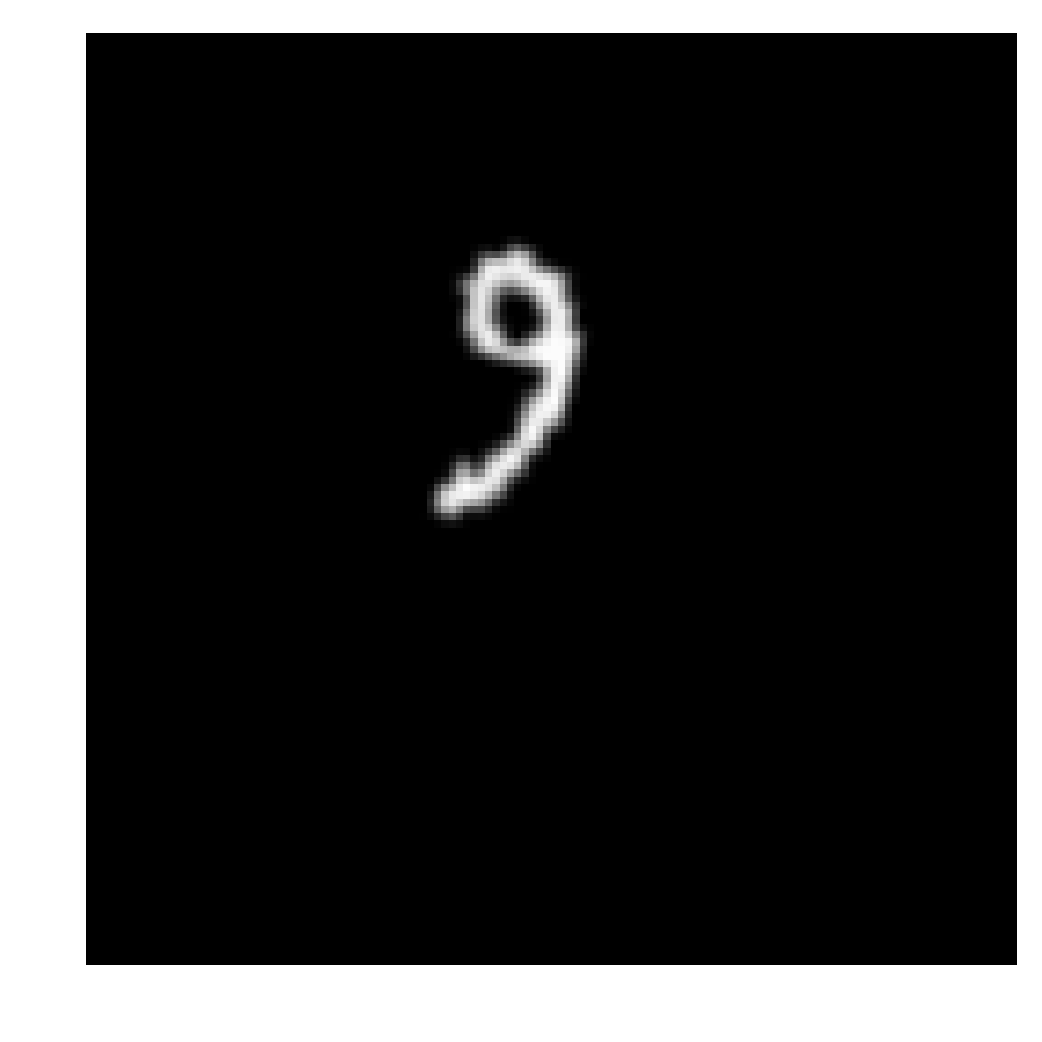}
        & \includegraphics[width=0.2\linewidth, clip=true, trim=40pt 40pt 20pt 20pt]{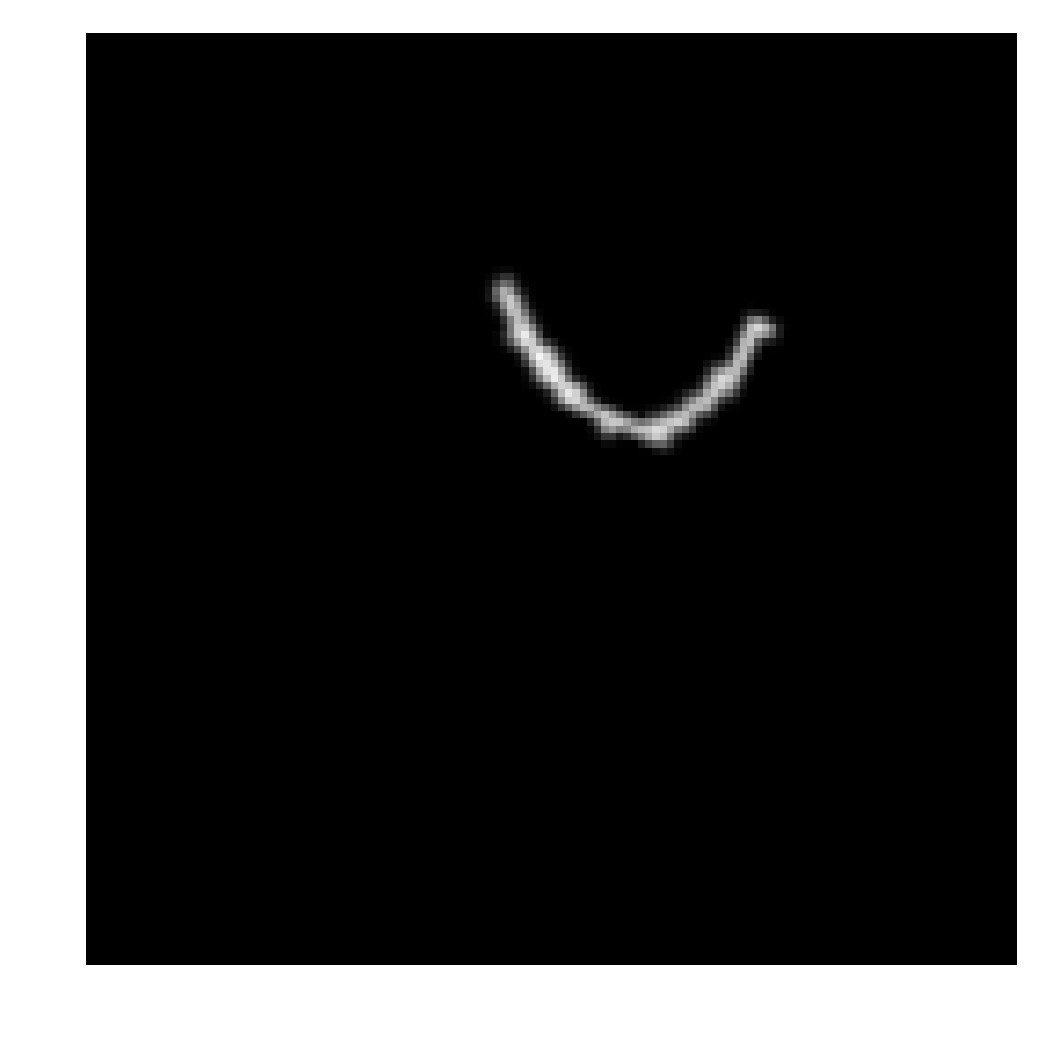}
        & \includegraphics[width=0.2\linewidth, clip=true, trim=40pt 40pt 20pt 20pt]{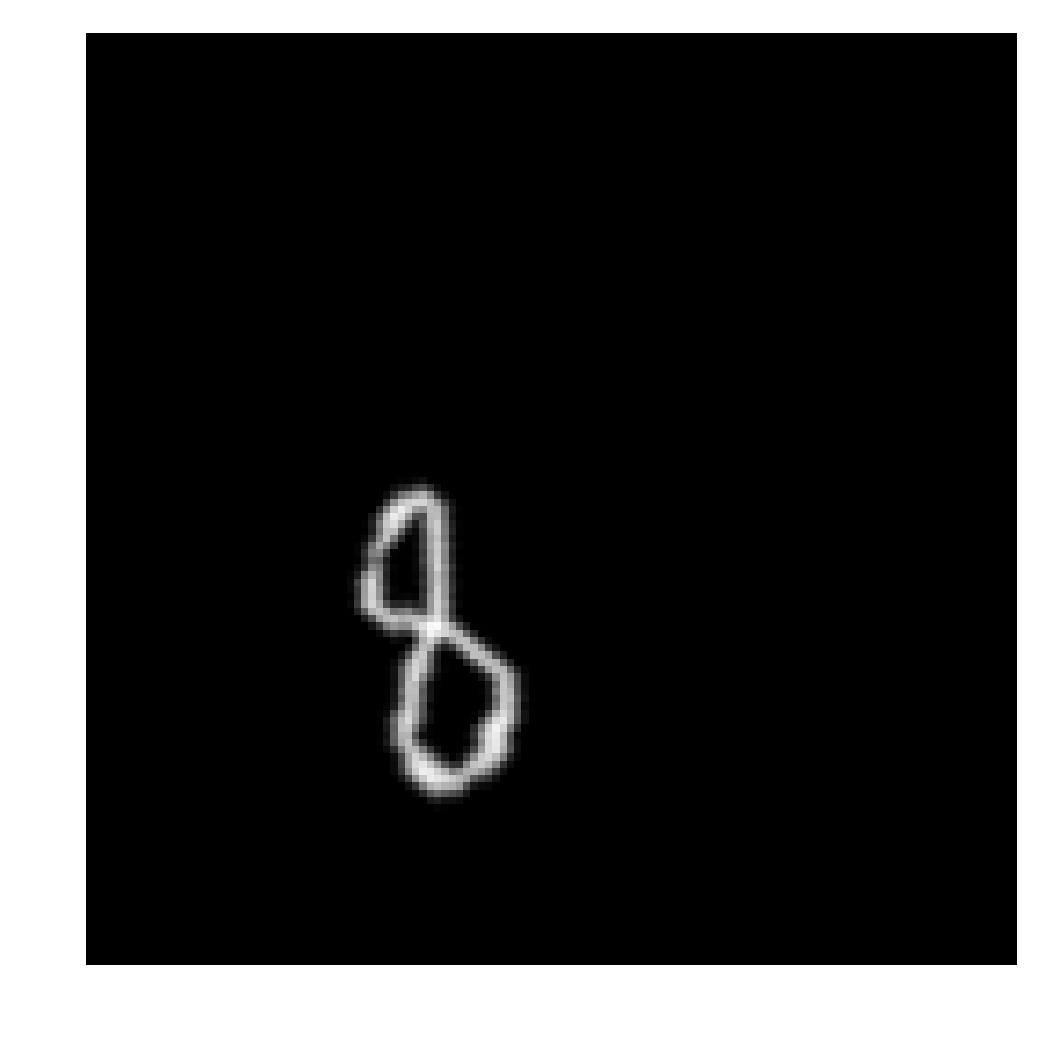}
        & \includegraphics[width=0.2\linewidth, clip=true, trim=40pt 40pt 20pt 20pt]{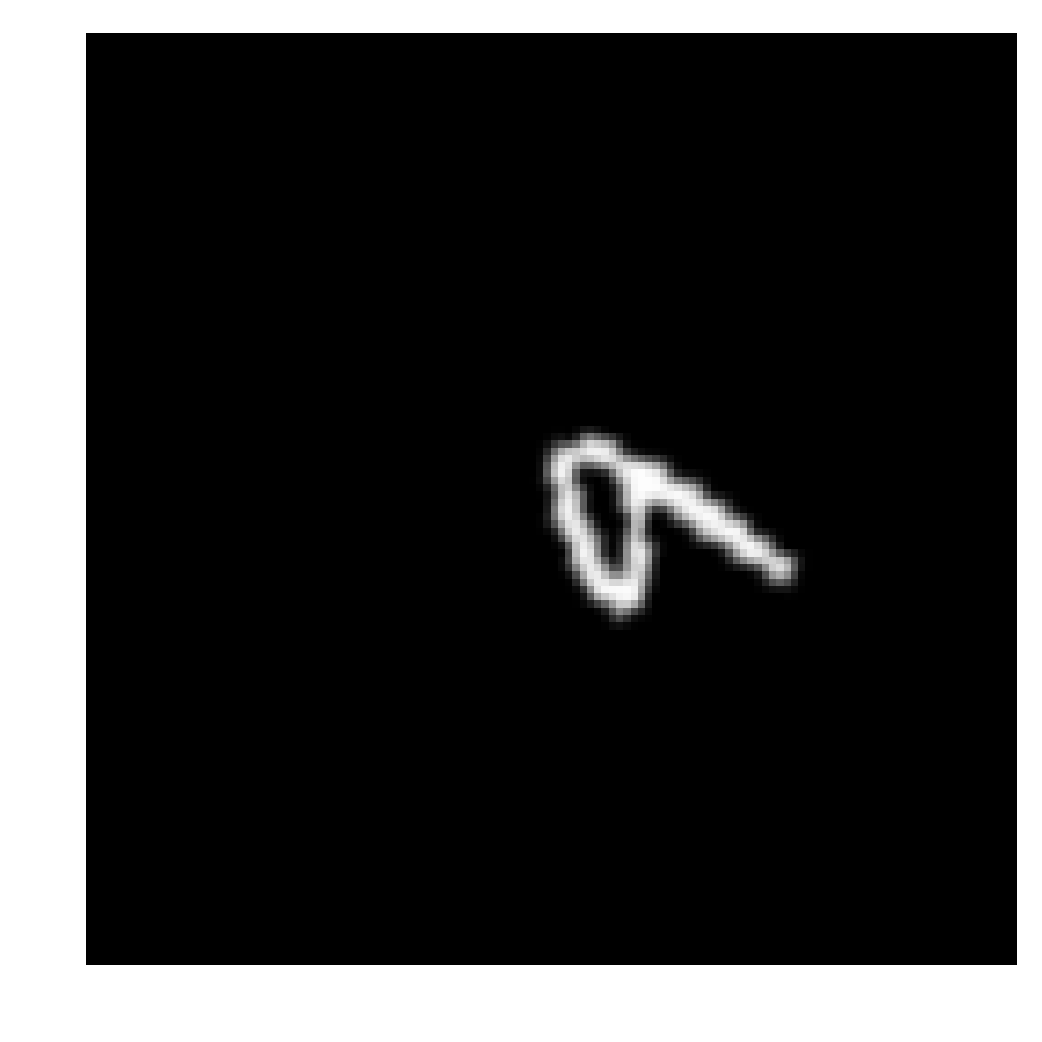}
    \end{tabular}}
    \caption{
    Two samples of each class (zero to nine)
    of the LinMNIST dataset, 
    generated by randomly rotating, 
    anisotropically scaling
    and shifting
    a random MNIST sample of the corresponding class.
    Here, 
    rotation angles are in $[0^{\circ},360^{\circ}]$,
    scalings in $[0.75,1.0]$
    and shifts in $[-20,20]$.
    The image resolution is
    $128 \times 128$ pixels.
    }
    \label{fig:linmnist}
\end{figure}

\begin{table}
    \caption{
    $k$-NN classification accuracies
    (mean plus/minus standard deviation)
    for the LinMNIST dataset with class size $500$
    generated by anisotropic scaling in $[0.75,1.0]$,
    rotation angles in $[0^\circ, 360^\circ]$
    and pixel shifts in $[-20,20]$.
    For each number of training samples,
    the best result is highlighted.
    }
    \resizebox{\linewidth}{!}{%
    \begin{tabular}{l l  l l l l}
        \toprule
        \# training samples & $k$ & Euclidean & R-CDT & \mNRCDT{} & \aNRCDT{} \\
        \midrule
        \multirow{3}{*}{11}
        & 1 
        & $0.1222\pm0.0078$ 
        & $0.1279\pm0.0067$ 
        & $0.5541\pm0.0134$
        & $0.3899\pm0.0140$ \\
        & 5 
        & $0.1168\pm0.0100$
        & $0.1105\pm0.0037$
        & \bm{$0.5591\pm0.0148$}
        & $0.4005\pm0.0135$ \\
        & 11
        & $0.1135\pm0.0097$ 
        & $0.1053\pm0.0038$
        & $0.5445\pm0.0184$
        & $0.4015\pm0.0131$ \\
        \midrule
        \multirow{3}{*}{25} 
        & 1 
        & $0.1387\pm0.0067$
        & $0.1434\pm0.0067$
        & $0.6010\pm0.0122$
        & $0.4208\pm0.0098$ \\
        & 5 
        & $0.1278\pm0.0113$
        & $0.1260\pm0.0060$
        & \bm{$0.6147\pm0.0094$}
        & $0.4402\pm0.0111$ \\
        & 11 
        & $0.1229\pm0.0121$
        & $0.1156\pm0.0047$
        & $0.6132\pm0.0100$
        & $0.4453\pm0.0107$ \\
        \midrule
        \multirow{3}{*}{50} 
        & 1 
        & $0.1597\pm0.0052$
        & $0.1619\pm0.0056$
        & $0.6283\pm0.0083$ 
        & $0.4467\pm0.0076$ \\
        & 5 
        & $0.1416\pm0.0069$
        & $0.1457\pm0.0047$
        & $0.6524\pm0.0084$ 
        & $0.4722\pm0.0067$ \\
        & 11 
        & $0.1318\pm0.0089$
        & $0.1336\pm0.0054$
        & \bm{$0.6524\pm0.0075$} 
        & $0.4776\pm0.0089$ \\
        \bottomrule
    \end{tabular}}
    \label{tab:linmnist_kNN}
\end{table}

The classification accuracies
are reported in Table~\ref{tab:linmnist_kNN}.
We observe that \mNRCDT{} clearly outperforms
the other approaches followed by \aNRCDT{}
and reaches a $k$-NN classification
accuracy of $65\%$
when using $k = 11$
and $50$ training samples.
In contrast to this,
the Euclidean and R-CDT representation
perform on the level of random guessing.
Inspecting the confusion maps
in Figure~\ref{fig:linmnist_conf_map}
reveals that both \mNRCDT{} and \aNRCDT{}
nearly perfectly classify classes $0$ and $1$.
The $4$s are better classified in \aNRCDT{} space.
For all other numbers,
the representation in \mNRCDT{} space
leads to better classifications.

\begin{figure}
    \resizebox{\linewidth}{!}{%
    \scriptsize
    \begin{tabular}{c @{\hspace{3pt}} c @{\hspace{10pt}} c @{\hspace{3pt}} c @{\hspace{10pt}} c @{\hspace{3pt}} c @{\hspace{3pt}} c}
        \multicolumn{2}{c}{11 training samples} 
        & \multicolumn{2}{c}{25 training samples}
        & \multicolumn{2}{c}{50 training samples} \\ 
        \mNRCDT{} & \aNRCDT{} & \mNRCDT{} & \aNRCDT{} & \mNRCDT{} & \aNRCDT{} \\
        \includegraphics[width=0.167\linewidth, clip=true, trim=10pt 10pt 70pt 0pt]{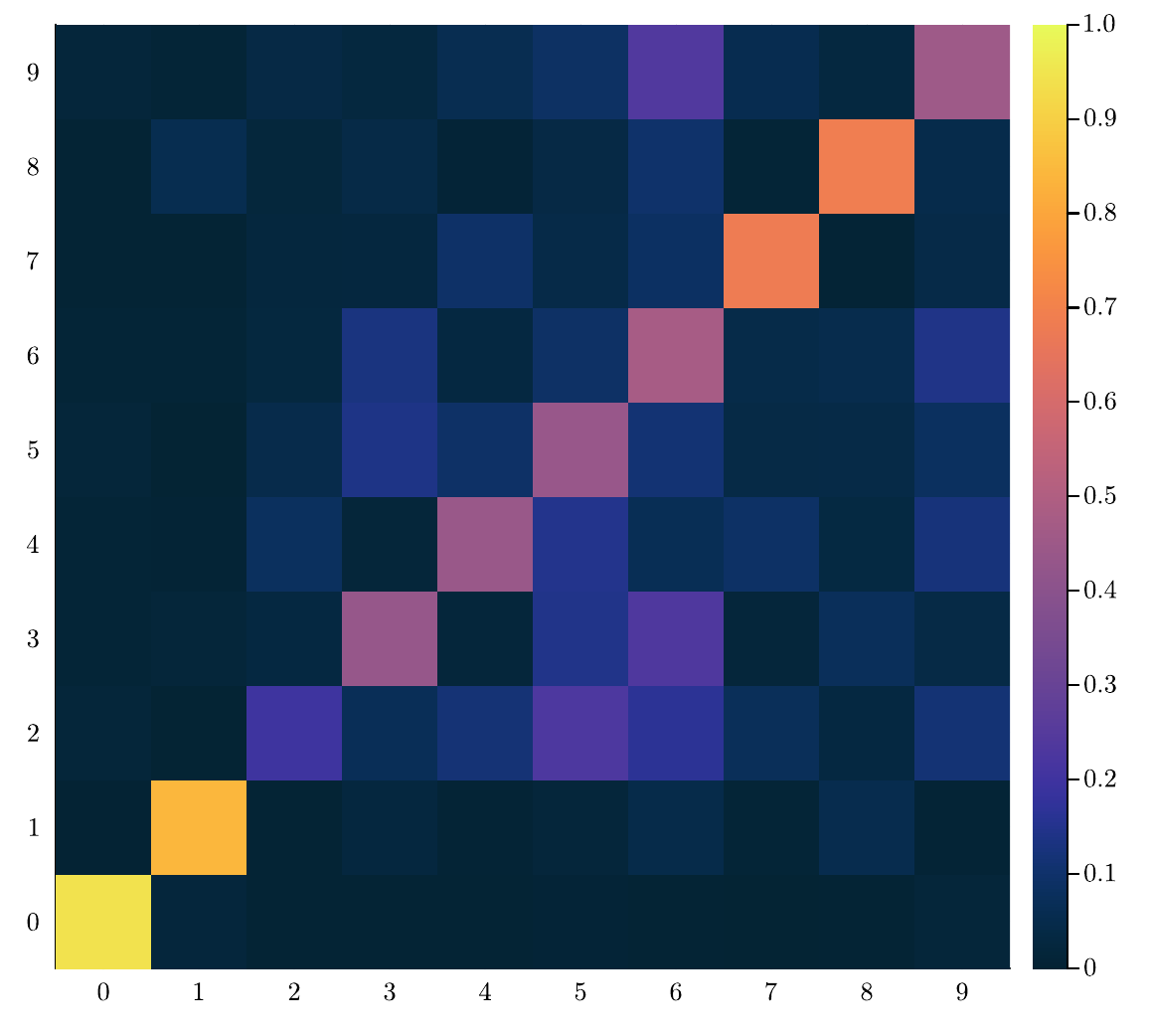}
        & \includegraphics[width=0.167\linewidth, clip=true, trim=10pt 10pt 70pt 0pt]{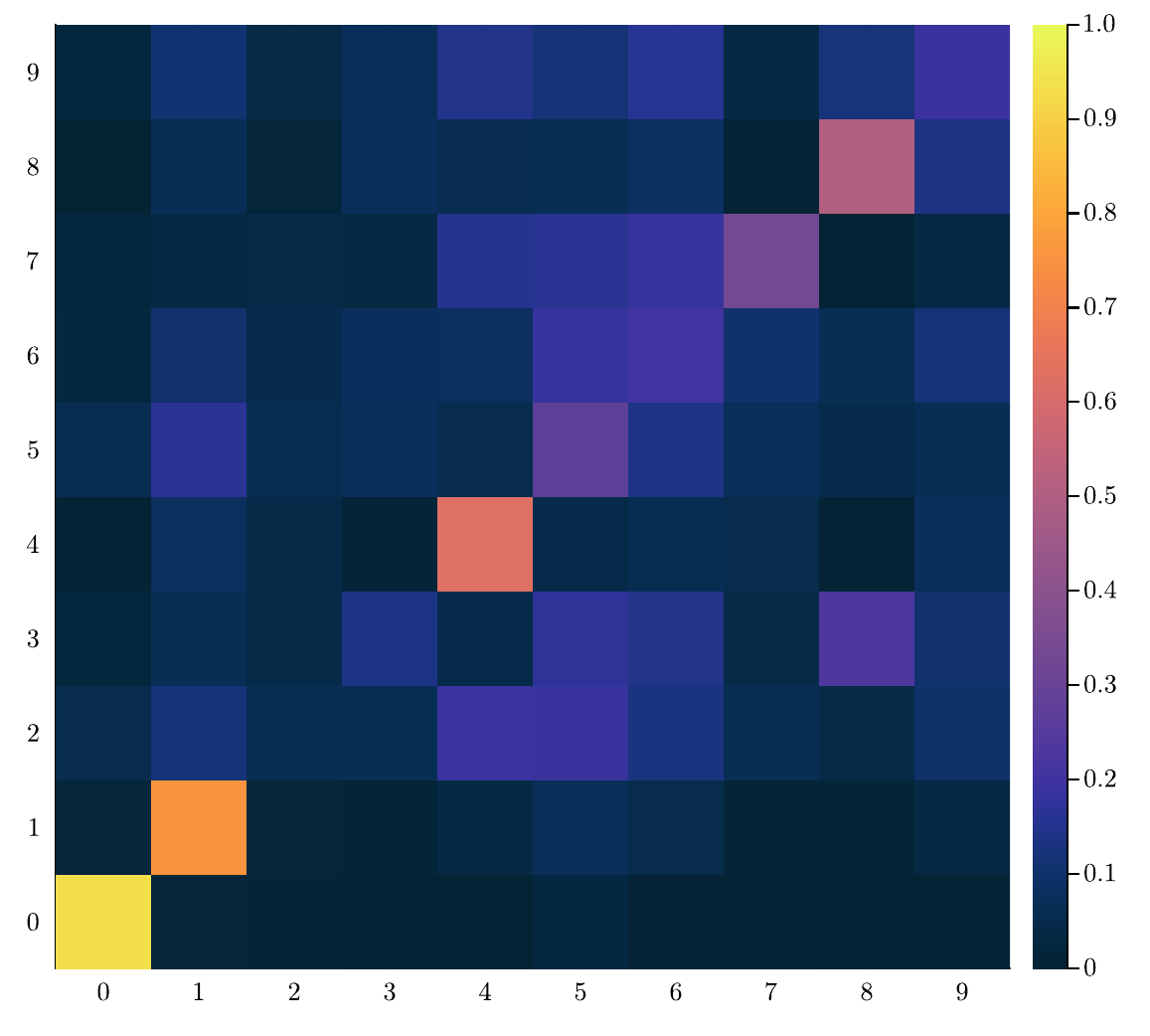}
        & \includegraphics[width=0.167\linewidth, clip=true, trim=10pt 10pt 70pt 0pt]{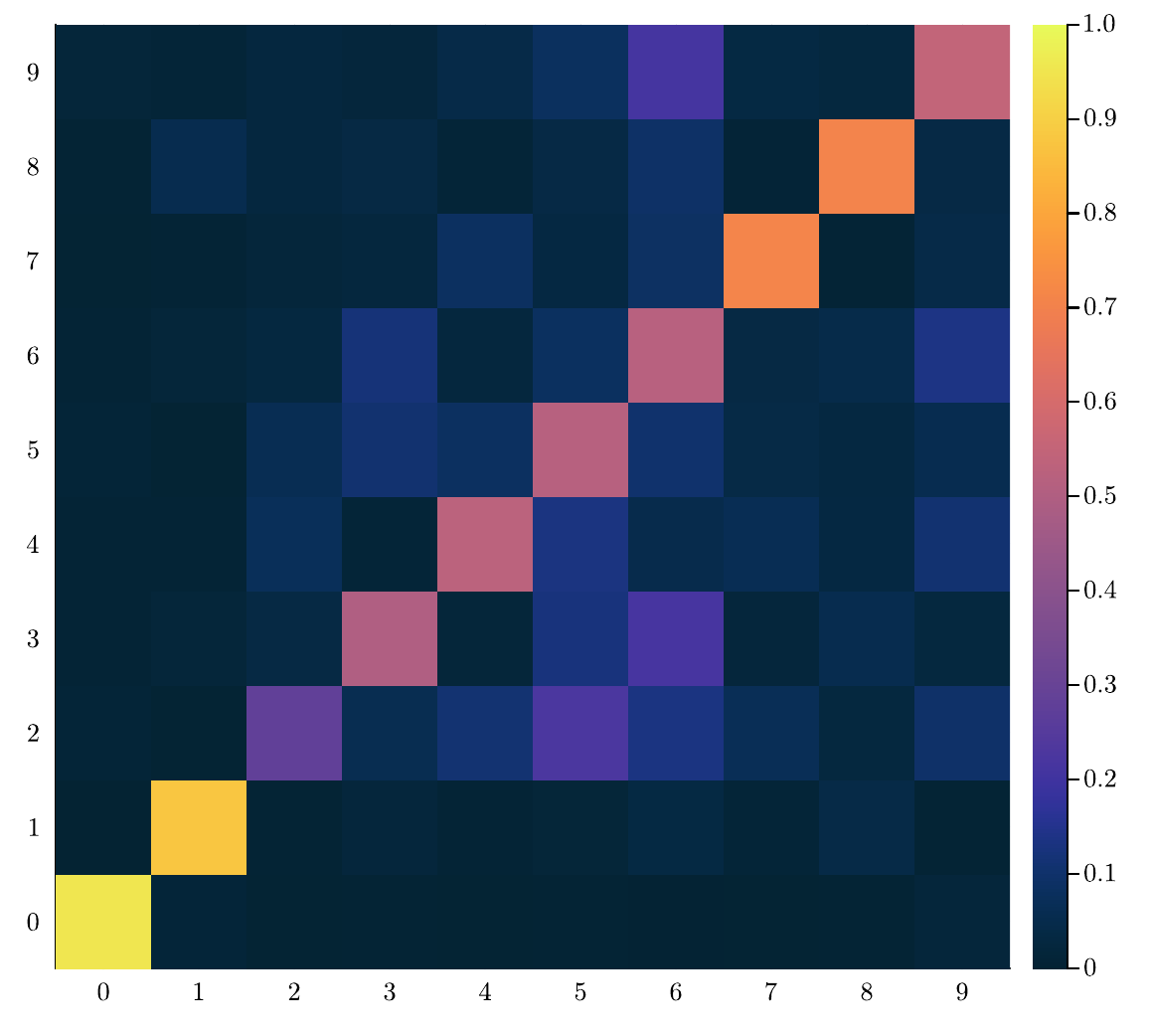}
        & \includegraphics[width=0.167\linewidth, clip=true, trim=10pt 10pt 70pt 0pt]{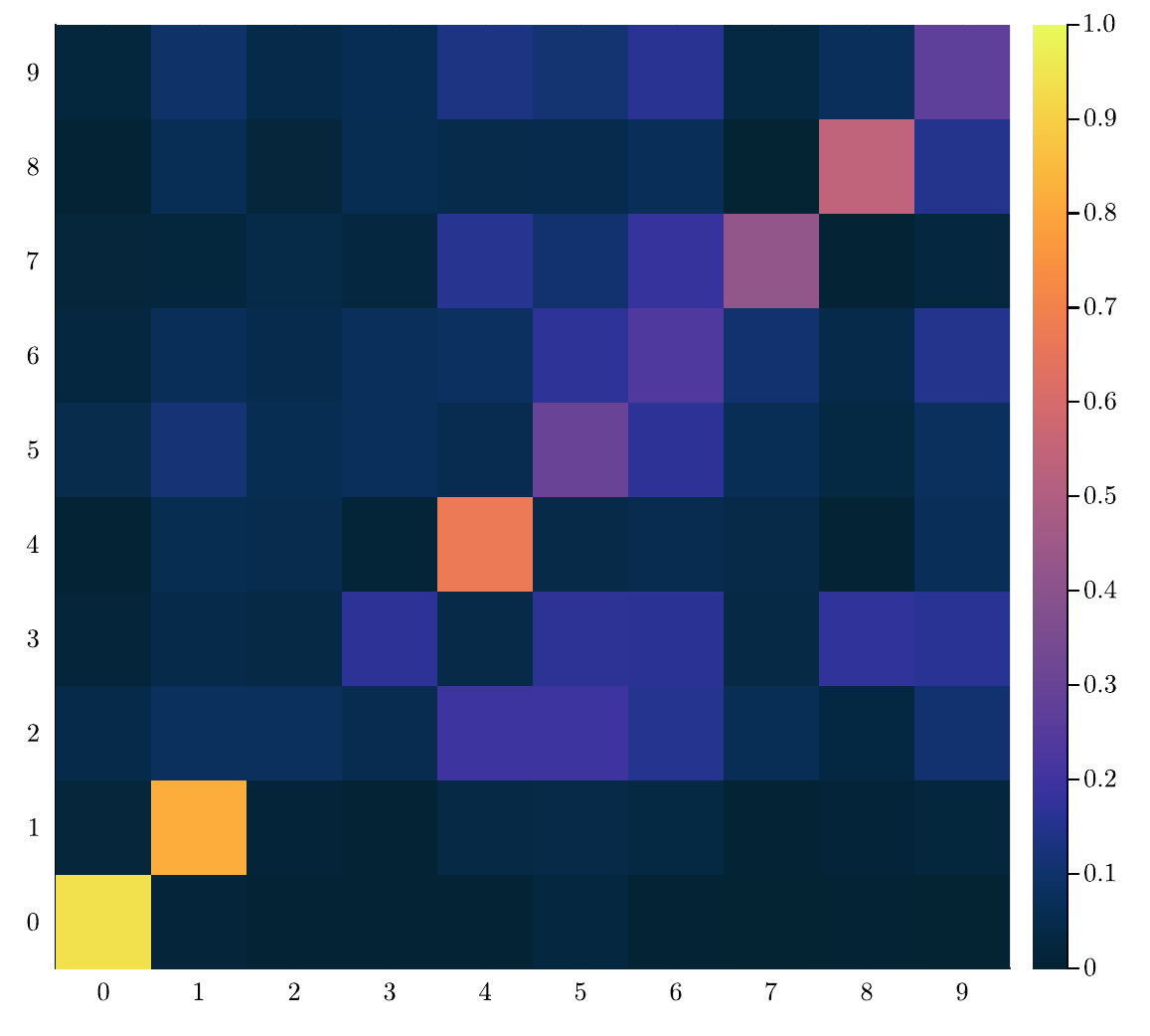}
        & \includegraphics[width=0.167\linewidth, clip=true, trim=10pt 10pt 70pt 0pt]{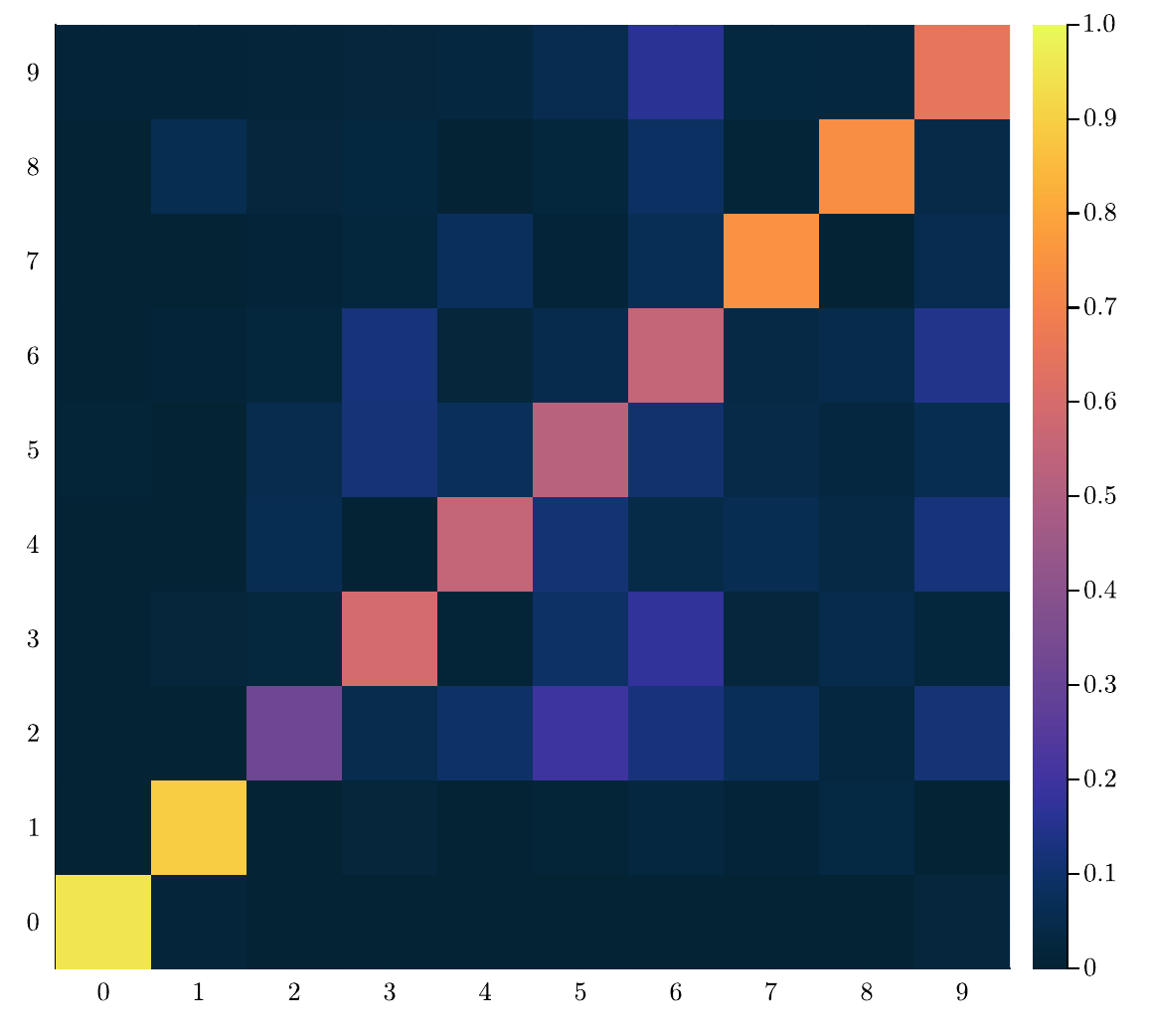}
        & \includegraphics[width=0.167\linewidth, clip=true, trim=10pt 10pt 70pt 0pt]{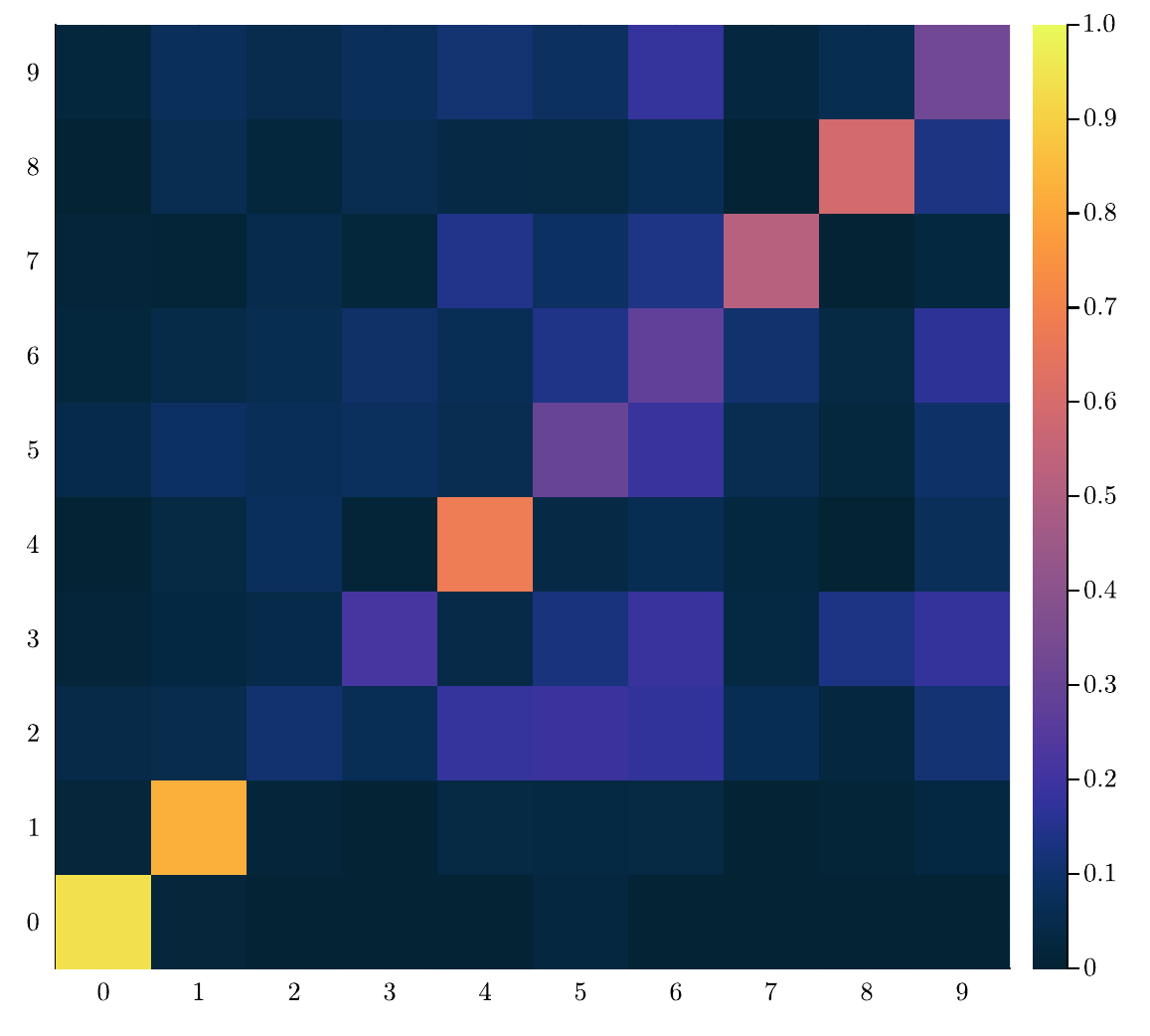} 
        &\includegraphics[height=0.172\linewidth, clip=true, trim=10pt 10pt 20pt 5pt]{Images/colorbar_thermal.pdf}\\
        $5$-NN & $11$-NN & $5$-NN & $11$-NN & $11$-NN & $11$-NN
    \end{tabular}}%
    \caption{
    Confusion maps for the best
    \mNRCDT{} and \aNRCDT{} results
    per number of training samples.
    }
    \label{fig:linmnist_conf_map}
\end{figure}

\section{Conclusion}
\label{sec:conclusion}

In this paper,
we continued our study
of the \mNRCDT{} introduced in~\cite{Beckmann2024a}
to enhance separability
by analysing its robustness
with respect to non-affine perturbations.
In addition,
we introduced the \aNRCDT{},
which shows an improved numerical performance
especially in the presence of impulsive noise.
In future works,
we wish to design refined approaches
for handling more severe and realistic noise models,
specifically for our motivating pattern recognition task in filigranology.
In so doing,
we aim to surmount the gap between mathematical theory and practice.
In particular,
our \mNRCDT{} and \aNRCDT{} feature representations
are to be used in real-world applications
like a fully automated watermark recognition and classification pipeline. 
Moreover,
we wish to improve our estimates
especially regarding the \mNRCDT{}
as this appears to perform
better than suggested by our theoretical findings.
In our proof-of-concept experiments,
we relied on NN classifiers 
to show the superior performance of the proposed feature extractors
in comparison with the Euclidean and R-CDT approach.
As in \cite{Shifat-E-Rabbi2021},
the NN classifier may be replaced by
more advanced classification methods
to improve the shown results.
Finally,
besides classification tasks,
we want to study the impact of our feature extractors
on clustering problems.

\vfill

%
%


\begin{thebibliography}{10}

\bibitem{Aldroubi2022}
{\sc A.~Aldroubi, R.~Diaz~Martin, I.~V. Medri, G.~K. Rohde, and S.~Thareja},
  {\em The signed cumulative distribution transform for 1-d signal analysis and
  classification}, Foundations of Data Science, 4 (2022), pp.~137--163,
  \url{https://doi.org/10.3934/fods.2022001}.

\bibitem{Aldroubi2021}
{\sc A.~Aldroubi, S.~Li, and G.~K. Rohde}, {\em Partitioning signal classes
  using transport transforms for data analysis and machine learning}, Sampling
  Theory, Signal Processing, and Data Analysis, 19 (2021), p.~6,
  \url{https://doi.org/10.1007/s43670-021-00009-z}.

\bibitem{Beckmann2024a}
{\sc M.~Beckmann, R.~Beinert, and J.~Bresch}, {\em Max-normalized radon
  cumulative distribution transform for limited data classification}, in Scale
  Space and Variational Methods in Computer Vision (SSVM), vol.~15667 of
  Lecture Notes in Computer Science, Springer, 2025, pp.~241--254,
  \url{https://doi.org/10.1007/978-3-031-92366-1_19}.

\bibitem{Beckmann2024}
{\sc M.~Beckmann and N.~Heilenk{\"o}tter}, {\em Equivariant neural networks for
  indirect measurements}, SIAM Journal on Mathematics of Data Science, 6
  (2024), pp.~579--601, \url{https://doi.org/10.1137/23M1582862}.

\bibitem{Beier2022}
{\sc F.~Beier, R.~Beinert, and G.~Steidl}, {\em On a linear
  {G}romov–{W}asserstein distance}, IEEE Transactions on Image Processing, 31
  (2022), pp.~7292--7305, \url{https://doi.org/10.1109/TIP.2022.3221286}.

\bibitem{Beinert2023}
{\sc R.~Beinert, C.~Heiss, and G.~Steidl}, {\em On assignment problems related
  to {G}romov–{W}asserstein distances on the real line}, SIAM Journal on
  Imaging Sciences, 16 (2023), pp.~1028--1032,
  \url{https://doi.org/10.1137/22M1497808}.

\bibitem{Bliem2022}
{\sc P.~Bliem}, {\em Handwritten chinese character hanzi datasets}, 2022,
  \url{https://www.kaggle.com/datasets/pascalbliem/handwritten-chinese-character-hanzi-datasets}.
\newblock Accessed: March 26, 2025.

\bibitem{Bonneel2015}
{\sc N.~Bonneel, J.~Rabin, G.~Peyr{\'e}, and H.~Pfister}, {\em Sliced and
  {R}adon {W}asserstein barycenters of measures}, Journal of Mathematical
  Imaging and Vision, 51 (2015), pp.~22--45,
  \url{https://doi.org/10.1007/s10851-014-0506-3}.

\bibitem{Cloninger2025}
{\sc A.~Cloninger, K.~Hamm, V.~Khurana, and C.~Moosm{\" u}ller}, {\em
  Linearized {W}asserstein dimensionality reduction with approximation
  guarantees}, Applied and Computational Harmonic Analysis, 74 (2025),
  p.~101718, \url{https://doi.org/10.1016/j.acha.2024.101718}.

\bibitem{Deng2012}
{\sc L.~Deng}, {\em The {MNIST} database of handwritten digit images for
  machine learning research}, IEEE Signal Processing Magazine, 29 (2012),
  pp.~141--142, \url{https://doi.org/10.1109/MSP.2012.2211477}.

\bibitem{DiazMartin2024}
{\sc R.~D{\' i}az~Mart{\' i}n, I.~V. Medri, and G.~K. Rohde}, {\em Data
  representation with optimal transport}, 2024,
  \url{https://doi.org/10.48550/arXiv.2406.15503}.
\newblock arXiv:2406.15503.

\bibitem{Givens1984}
{\sc C.~R. Givens and R.~M. Shortt}, {\em A class of {W}asserstein metrics for
  probability distributions}, Michigan Mathematical Journal, 31 (1984),
  pp.~231--240, \url{https://doi.org/10.1307/mmj/1029003026}.

\bibitem{Hauser2024}
{\sc D.~Hauser, M.~Beckmann, G.~Koliander, and H.~S. Stiehl}, {\em On image
  processing and pattern recognition for thermograms of watermarks in
  manuscripts -- a first proof-of-concept}, in International Conference on
  Document Analysis and Recognition (ICDAR), 2024, pp.~91--107,
  \url{https://doi.org/10.1007/978-3-031-70543-4_6}.

\bibitem{Helgason1999}
{\sc S.~Helgason}, {\em The Radon Transform}, Birkh{\"a}user, 2~ed., 1999.

\bibitem{Kolouri2016}
{\sc S.~Kolouri, S.~R. Park, and G.~K. Rohde}, {\em The {R}adon cumulative
  distribution transform and its application to image classification}, IEEE
  Transactions on Image Processing, 25 (2016), pp.~920--934,
  \url{https://doi.org/10.1109/TIP.2015.2509419}.

\bibitem{Kolouri2017}
{\sc S.~Kolouri, S.~R. Park, M.~Thorpe, D.~Slepcev, and G.~K. Rohde}, {\em
  Optimal mass transport}, IEEE Signal Processing Magazine, 34 (2017),
  pp.~43--59, \url{https://doi.org/10.1109/MSP.2017.2695801}.

\bibitem{Liu2024}
{\sc X.~Liu, R.~D{\' i}az~Mart{\' i}n, Y.~Bai, A.~Shahbazi, M.~Thorpe,
  A.~Aldroubi, and S.~Kolouri}, {\em Expected sliced transport plans}, 2024,
  \url{https://doi.org/10.48550/arXiv.2410.12176}.
\newblock arXiv:2410.12176.

\bibitem{Moosmueller2023}
{\sc C.~Moosm{\"u}ller and A.~Cloninger}, {\em Linear optimal transport
  embedding: provable {W}asserstein classification for certain rigid
  transformations and perturbations}, Information and Inference: A Journal of
  the IMA, 12 (2023), pp.~363--389,
  \url{https://doi.org/10.1093/imaiai/iaac023}.

\bibitem{Natterer2001}
{\sc F.~Natterer}, {\em The Mathematics of Computerized Tomography}, SIAM,
  Philadelphia, 2001, \url{https://doi.org/10.1137/1.9780898719284}.

\bibitem{Park2025}
{\sc S.~Park and D.~Slep{\u c}ev}, {\em Geometry and analytic properties of the
  sliced {W}asserstein space}, Journal of Functional Analysis, 289 (2025),
  p.~110975, \url{https://doi.org/10.1016/j.jfa.2025.110975}.

\bibitem{Park2018}
{\sc S.~R. Park, S.~Kolouri, S.~Kundu, and G.~K. Rohde}, {\em The cumulative
  distribution transform and linear pattern classification}, Applied and
  Computational Harmonic Analysis, 45 (2018), pp.~616--641,
  \url{https://doi.org/10.1016/j.acha.2017.02.002}.

\bibitem{Piening2025}
{\sc M.~Piening and R.~Beinert}, {\em Slicing the {G}aussian mixture
  {W}asserstein distance}, 2025,
  \url{https://doi.org/10.48550/arXiv.2504.08544}.
\newblock arXiv:2504.08544.

\bibitem{Quellmalz2023}
{\sc M.~Quellmalz, R.~Beinert, and G.~Steidl}, {\em Sliced optimal transport on
  the sphere}, Inverse Problems, 39 (2023), p.~105005,
  \url{https://doi.org/10.1088/1361-6420/acf156}.

\bibitem{Quellmalz2024}
{\sc M.~Quellmalz, L.~Buecher, and G.~Steidl}, {\em Parallelly sliced optimal
  transport on spheres and on the rotation group}, Journal of Mathematical
  Imaging and Vision, 66 (2024), pp.~951--976,
  \url{https://doi.org/10.1007/s10851-024-01206-w}.

\bibitem{Radon1917}
{\sc J.~Radon}, {\em {\"U}ber die {B}estimmung von {F}unktionen durch ihre
  {I}ntegralwerte l{\"a}ngs gewisser {M}annigfaltigkeiten}, in Berichte
  {\"u}ber die Verhandlungen der K\"oniglich-S{\"a}chsischen Gesellschaft der
  Wissenschaften zu Leipzig, vol.~69 of Mathematisch-Physische Klasse,
  Königlich Sächsische Gesellschaft der Wissenschaften, Leipzig, 1917,
  pp.~262--277.

\bibitem{Ramm1996}
{\sc A.~G. Ramm and A.~I. Katsevich}, {\em The Radon Transform and Local
  Tomography}, CRC Press, 1996, \url{https://doi.org/10.1201/9781003069331}.

\bibitem{Shifat-E-Rabbi2021}
{\sc M.~Shifat-E-Rabbi, X.~Yin, A.~H.~M. Rubaiyat, S.~Li, S.~Kolouri,
  A.~Aldroubi, J.~M. Nichols, and G.~K. Rohde}, {\em Radon cumulative
  distribution transform subspace modeling for image classification}, Journal
  of Mathematical Imaging and Vision, 63 (2021), pp.~1185--1203,
  \url{https://doi.org/10.1007/s10851-021-01052-0}.

\bibitem{Shifat-E-Rabbi2023}
{\sc M.~Shifat-E-Rabbi, Y.~Zhuang, S.~Li, A.~H.~M. Rubaiyat, X.~Yin, and G.~K.
  Rohde}, {\em Invariance encoding in sliced-{W}asserstein space for image
  classification with limited training data}, Pattern Recognition, 137 (2023),
  p.~109268, \url{https://doi.org/10.1016/j.patcog.2022.109268}.

\bibitem{Villani2003}
{\sc C.~Villani}, {\em Topics in Optimal Transportation}, American Mathematical
  Society, 2003, \url{https://doi.org/10.1090/gsm/058}.

\bibitem{Zhuang2025}
{\sc Y.~Zhuang, S.~Li, M.~Shifat-E-Rabbi, X.~Yin, A.~H.~M. Rubaiyat, and G.~K.
  Rohde}, {\em Local sliced {W}asserstein feature sets for illumination
  invariant face recognition}, Pattern Recognition, 162 (2025), p.~111381,
  \url{https://doi.org/10.1016/j.patcog.2025.111381}.

\end{thebibliography}
\end{document}